\documentclass[11pt,reqno]{amsart}

\usepackage[margin=1.0in]{geometry}
\usepackage{amsmath}
\usepackage{amsfonts}
\usepackage{mathrsfs}
\usepackage {amssymb,euscript}
\usepackage {amsmath}
\usepackage {amsthm}
\usepackage {amscd}
\usepackage{geometry}
\usepackage{hyperref}
\usepackage{color}
\usepackage{float}
\usepackage{epsfig}
\usepackage{caption}
\usepackage{graphicx}
\usepackage{euscript,bbm,color,slashed,enumerate,bm}
\usepackage{subcaption}
\usepackage{tikz}
\usepackage{tkz-euclide}
\usepackage[T1]{fontenc} 
\usepackage{ulem}
\usetkzobj{all}

\usepackage[foot]{amsaddr}

\usepackage{graphics}

\title[Spacelike Singularity]{Curvature blow-up rates in spherically symmetric gravitational collapse to a Schwarzschild black hole
}
\date{\today}

\author{Xinliang An$^*$$^1$}\author{Dejan Gajic$^{\dag}$$^2$}
\address{$^{1}$\small Department of Mathematics, National University of Singapore, 
10 Lower Kent Ridge Road, Singapore, 119076 }
\address{$^{2}$\small University of Cambridge, Department of Pure Mathematics and Mathematical Statistics, Wilberforce Road, Cambridge CB3 0WB, United Kingdom }

\email{$^*$matax@nus.edu.sg}
\email{$^{\dag}$d.gajic@dpmms.cam.ac.uk}

\newtheorem{lemma}{Lemma}[section]

\newtheorem{proposition}[lemma]{Proposition}
\newtheorem{theorem}[lemma]{Theorem}

\newtheorem{remark}{Remark}

\numberwithin{equation}{section}

\newtheorem{thmx}{Theorem}

\begin{document}

\newcommand{\ub}{\underline{u}}
\newcommand{\Cb}{\underline{C}}
\newcommand{\Lb}{\underline{L}}
\newcommand{\Lh}{\hat{L}}
\newcommand{\Lbh}{\hat{\Lb}}
\newcommand{\phib}{\underline{\phi}}
\newcommand{\Phib}{\underline{\Phi}}
\newcommand{\Db}{\underline{D}}
\newcommand{\Dh}{\hat{D}}
\newcommand{\Dbh}{\hat{\Db}}
\newcommand{\omb}{\underline{\omega}}
\newcommand{\omh}{\hat{\omega}}
\newcommand{\ombh}{\hat{\omb}}
\newcommand{\Pb}{\underline{P}}
\newcommand{\chib}{\underline{\chi}}
\newcommand{\chih}{\hat{\chi}}
\newcommand{\chibh}{\hat{\chib}}

\newcommand{\alb}{\underline{\alpha}}
\newcommand{\zeb}{\underline{\zeta}}
\newcommand{\beb}{\underline{\beta}}
\newcommand{\etb}{\underline{\eta}}
\newcommand{\Mb}{\underline{M}}
\newcommand{\oth}{\hat{\otimes}}


\def\a {\alpha}
\def\b {\beta}
\def\ab {\alphab}
\def\bb {\betab}
\def\nab {\nabla}


\def\ub {\underline{u}}
\def\th {\theta}
\def\Lb {\underline{L}}
\def\Hb {\underline{H}}
\def\chib {\underline{\chi}}
\def\chih {\hat{\chi}}
\def\chibh {\hat{\underline{\chi}}}
\def\omegab {\underline{\omega}}
\def\etab {\underline{\eta}}
\def\betab {\underline{\beta}}
\def\alphab {\underline{\alpha}}
\def\Psib {\underline{\Psi}}
\def\hot{\widehat{\otimes}}
\def\Phib {\underline{\Phi}}
\def\thb {\underline{\theta}}
\def\t {\tilde}
\def\st {\tilde{s}}

\def\rhoc{\check{\rho}}
\def\sigmac{\check{\sigma}}
\def\Psic{\check{\Psi}}
\def\kappab{\underline{\kappa}}
\def\betabc {\check{\underline{\beta}}}

\def\d {\delta}
\def\f {\frac}
\def\i {\infty}
\def\l {\bigg(}
\def\r {\bigg)}
\def\S {S_{u,\underline{u}}}
\def\o{\omega}
\def\O{\Omega}\
\def\be{\begin{equation}\begin{split}}
\def\en{\end{split}\end{equation}}
\def\at{a^{\frac{1}{2}}}
\def\af{a^{\frac{1}{4}}}
\def\od{\omega^{\dagger}}
\def\ombd{\underline{\omega}^{\dagger}}
\def\K{K-\frac{1}{|u|^2}}
\def\ut{\frac{1}{|u|^2}}
\def\Kb{K-\frac{1}{(u+\underline{u})^2}}
\def\bf{b^{\frac{1}{4}}}
\def\bt{b^{\frac{1}{2}}}
\def\de{\delta}
\def\ls{\lesssim}
\def\om{\omega}
\def\Om{\Omega}
\def\Orw{O(r_1^{\f{1}{100}})}
\def\Ort{O(r_2^{\f{1}{100}})}
\def\Or{O(r^{\f{1}{100}})}

\newcommand{\e}{\epsilon}
\newcommand{\et} {\frac{\epsilon}{2}}
\newcommand{\ef} {\frac{\epsilon}{4}}
\newcommand{\LH} {L^2(H_u)}
\newcommand{\LHb} {L^2(\underline{H}_{\underline{u}})}
\newcommand{\M} {\mathcal}
\newcommand{\TM} {\tilde{\mathcal}}
\newcommand{\p}{\psi\hspace{1pt}}
\newcommand{\q}{\underline{\psi}\hspace{1pt}}
\newcommand{\Li}{_{L^{\infty}(S_{u,\underline{u}})}}
\newcommand{\Lt}{_{L^{2}(S)}}
\newcommand{\da}{\delta^{-\frac{\epsilon}{2}}}
\newcommand{\db}{\delta^{1-\frac{\epsilon}{2}}}
\newcommand{\D}{\Delta}
\newcommand{\snabla}{\slashed{\nabla}}
\newcommand{\sg}{\slashed{g}}
\newcommand{\sD}{\slashed{{\Delta}}}
\newcommand{\R}{\mathbb{R}}
\newcommand{\s}{\mathbb{S}}
\newcommand{\C}{\mathbb{C}}
\newcommand{\Q}{\mathbb{Q}}
\newcommand{\Z}{\mathbb{Z}}
\newcommand{\N}{\mathbb{N}}
\newcommand{\T}{\mathbf{T}}


\renewcommand{\div}{\mbox{div }}
\newcommand{\curl}{\mbox{curl }}
\newcommand{\trchb}{\mbox{tr} \chib}
\def\trch{\mbox{tr}\chi}
\newcommand{\tr}{\mbox{tr}}

\newcommand{\Ls}{{\mathcal L} \mkern-10mu /\,}
\newcommand{\eps}{{\epsilon} \mkern-8mu /\,}

\newcommand{\xib}{\underline{\xi}}
\newcommand{\psib}{\underline{\psi}}
\newcommand{\rhob}{\underline{\rho}}
\newcommand{\thetab}{\underline{\theta}}
\newcommand{\gammab}{\underline{\gamma}}
\newcommand{\nub}{\underline{\nu}}
\newcommand{\lb}{\underline{l}}
\newcommand{\mub}{\underline{\mu}}
\newcommand{\Xib}{\underline{\Xi}}
\newcommand{\Thetab}{\underline{\Theta}}
\newcommand{\Lambdab}{\underline{\Lambda}}
\newcommand{\vphb}{\underline{\varphi}}

\newcommand{\ih}{\hat{i}}

\newcommand{\tcL}{\widetilde{\mathscr{L}}}

\newcommand{\sRic}{Ric\mkern-19mu /\,\,\,\,}
\newcommand{\sL}{{\cal L}\mkern-10mu /}
\newcommand{\sLh}{\hat{\sL}}
\newcommand{\seps}{\epsilon\mkern-8mu /}
\newcommand{\sd}{d\mkern-10mu /}
\newcommand{\sR}{R\mkern-10mu /}
\newcommand{\snab}{\nabla\mkern-13mu /}
\newcommand{\sdiv}{\mbox{div}\mkern-19mu /\,\,\,\,}
\newcommand{\scurl}{\mbox{curl}\mkern-19mu /\,\,\,\,}
\newcommand{\slap}{\mbox{$\triangle  \mkern-13mu / \,$}}
\newcommand{\sGamma}{\Gamma\mkern-10mu /}
\newcommand{\somega}{\omega\mkern-10mu /}
\newcommand{\somb}{\omb\mkern-10mu /}
\newcommand{\spi}{\pi\mkern-10mu /}
\newcommand{\sJ}{J\mkern-10mu /}
\renewcommand{\sp}{p\mkern-9mu /}
\newcommand{\su}{u\mkern-8mu /}

\def\be{\begin{equation}}
\def\ee{\end{equation}}
\def\bes{\begin{equation*}}
\def\ees{\end{equation*}}
\newcommand\ba{\begin{align}}
\newcommand\ea{\end{align}}
\newcommand\bas{\begin{align*}}
\newcommand\eas{\end{align*}} 
\def\Orw{O\big((M^{-1}r_1)^{\f{1}{100}}\big)}
\def\Ort{O\big((M^{-1}r_2)^{\f{1}{100}}\big)}
\def\Or{O\big((M^{-1}r)^{\f{1}{100}}\big)}
\def\ba{\begin{align}}
\def\ea{\end{align}}
\def\bas{\begin{align*}}
\def\eas{\end{align*}}

\begin{abstract}
We study the black hole interiors of spacetimes arising from gravitational collapse in the spherically symmetric Einstein--scalar field setting, and we investigate the precise blow-up rates of curvature and mass at the spacelike singularity near timelike infinity. We show in particular that the Kretschmann scalar blows up faster than in the Schwarzschild setting, due to mass inflation. Moreover, the blow-up rate is not constant and converges to the Schwarzschild rate towards timelike infinity and it depends on the precise late-time polynomial behaviour of the scalar field along the event horizon. This indicates a new blow-up phenomenon, driven by a PDE mechanism, rather than an ODE mechanism. 
\end{abstract}

\maketitle

\tableofcontents

\section{Introduction}

\subsection{Background}
A central feature of the celebrated Schwarzschild spacetime solutions \cite{schw1916} to the vacuum Einstein equations
\begin{equation*}
\mbox{Ric}_{\mu\nu}[g]=0,
\end{equation*}
is the presence of a black hole region which is bounded to the future by a \emph{spacelike singularity} with area radius $r=0$. The Kretschmann scalar of the Schwarzschild solutions satisfies the following equation:
\begin{equation*}
R_{\alpha \beta \gamma \delta}R^{\alpha \beta \gamma \delta}= \frac{48M^2}{r^6},
\end{equation*}
where $M>0$ is the Schwarzschild mass. An immediate corollary of the blow-up of the Kretschmann scalar as $r\downarrow 0$ is that the metric cannot be extended in $C^2$ across the singularity.\footnote{In fact, by \cite{sbierski18}, it cannot even be extended in $C^0$. }

Singularities of dynamical black holes were investigated by Christodoulou in a series of papers \cite{DC91, DC93, DC94, DC99} in the setting of the Einstein-scalar field system:
\begin{equation}\label{ES}
\begin{split}
&\mbox{Ric}_{\mu\nu}-\f12Rg_{\mu\nu}=2T_{\mu\nu},\\
&T_{\mu\nu}=\partial_{\mu}\phi \partial_{\nu}\phi-\f12g_{\mu\nu}\partial^{\sigma}\phi \partial_{\sigma}\phi,
\end{split}
\end{equation}
under the restriction of spherical symmetry.

Christodoulou {\color{black}showed} in the spherically symmetric setting that the maximal globally hyperbolic Cauchy development of generic initial data, with respect to an appropriate choice of norm, contain a future-complete null infinity and are future $C^0$-inextendible as spherically symmetric Lorentzian manifolds, affirming thus both the Weak and Strong Cosmic Censorship conjectures in this spherically symmetric setting; see for example \cite{dl-scc} for modern statements of these conjectures and a discussion of their history. Christodoulou demonstrated moreover that for suitably \emph{large} initial data, the corresponding spacetime solutions will have a black hole region, bounded to the future by a spacelike singularity $\mathcal{S}$ where the area radius of the round 2-spheres foliating the spacetime goes to zero and the Kretschmann scalar blows up. See Figure \ref{fig:chrBHs} for a Penrose diagrammatic representation of such black hole solutions.
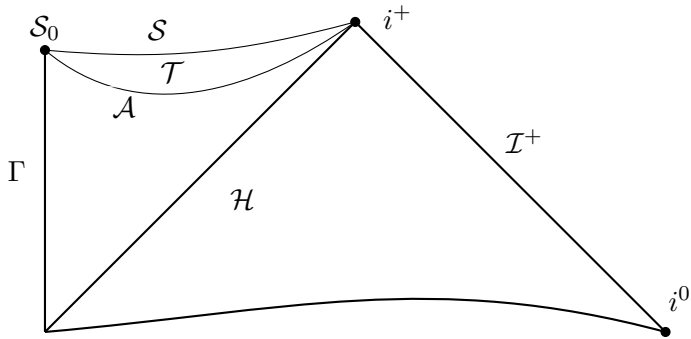
\begin{figure}[H]
\begin{center}
\begin{minipage}[!t]{0.4\textwidth}
\begin{tikzpicture}[scale=0.75]
\draw [white](-1, -2.5)-- node[midway, sloped, above,black]{$\Gamma$}(0, -2.5);
\draw [white](0, 0)-- node[midway, sloped, above,black]{$\mathcal{S}$}(4, 0);
\draw [white](0, -0.75)-- node[midway, sloped, above,black]{$\mathcal{T}$}(4.5, -0.75);
\draw [white](-1, 0)-- node[midway, sloped, above,black]{$\mathcal{S}_0$}(1, 0);
\draw [white](5.5, 0.2)-- node[midway, sloped, above,black]{$i^+$}(7, 0.2);
\draw [white](10, -4.8)-- node[midway, sloped, above,black]{$i^0$}(12.5, -4.8);
\draw (0,0) to [out=-5, in=195] (5.5, 0.5);
\draw (0,0) to [out=-40, in=215] (5.5, 0.5);
\draw [white](0, -3)-- node[midway, sloped, above,black]{$\mathcal{H}$}(7, -3);
\draw [white](7, -2)-- node[midway, sloped, above,black]{$\mathcal{I^+}$}(10, -2);
\draw [white](0, -0.65)-- node[midway, sloped, below,black]{$\mathcal{A}$}(2.8, -0.65);

\draw [thick] (0, -5)--(0,0);
\draw [thick] (5.5, 0.5)--(0,-5);
\draw[fill] (0,0) circle [radius=0.08];
\draw[fill] (5.5, 0.5) circle [radius=0.08];
\draw[fill] (11, -5) circle [radius=0.08];
\draw [thick] (5.5, 0.5)--(11,-5);
\draw [thick] (0,-5) to [out=5, in=165] (11, -5);
\end{tikzpicture}
\end{minipage}
\begin{minipage}[!t]{0.6\textwidth}
\end{minipage}
\hspace{0.05\textwidth}
\end{center}
\vspace{-0.2cm}
\caption{A Penrose diagram of the dynamical black hole solutions constructed by Christodoulou, with  $\Gamma$ denoting the center of spherical symmetry, future null infinity $\mathcal{I}^+$, event horizon $\mathcal{H}$ and spacelike singularity $\mathcal{S}$. The trapped region $\mathcal{T}$ is foliated by trapped spheres, $\mathcal{A}$ denotes apparent horizon, which is a hypersurface foliated by marginally trapped spheres and the point $S_0$ denotes the first singularity.}
	\label{fig:chrBHs}
\end{figure}
 
In the present paper, we investigate quantitatively the precise interior dynamics of black hole spacetimes constructed by Christodoulou. Assuming both quantitative inverse polynomial \emph{upper bounds} on the convergence of the black hole exterior to Schwarzschild (obtained in \cite{MDIR05}) and inverse polynomial \emph{lower bounds} (obtained in a linearized setting in \cite{aagprice}), \textbf{we establish blow-up of the Hawking mass at the spacelike singularity, resulting in a blow-up of the Kretschmann that is \emph{stronger} than in the non-dynamical Schwarzschild setting} and depends on the precise late-time behaviour of the scalar field along the event horizon. Along the way, \textbf{we show that appropriately renormalized dynamical quantities converge to their Schwarzschild values as we approach future timelike infinity in the black hole interior.}

\subsection{Lower bounds for the Kretschmann scalar from monotonicity}

In this section, we provide a sketch of the argument provided in \cite{DC91} for establishing blow-up of the Kretschmann scalar at the spacetime singularity. We can express the relevant spacetime metric $g$ in double null coordinates:
\begin{equation}\label{metric}
g=-\Omega^2(u,v)dudv+r^2(u,v)\big(d\theta^2+\sin^2\theta d\phi^2\big),
\end{equation}
where the level sets of $u$ and $v$ correspond to outgoing and ingoing null hypersurfaces, respectively; see Figure \ref{fig:DNintro} for an illustration.

\begin{figure}[H]
\begin{minipage}[!t]{0.4\textwidth}
	\begin{tikzpicture}[scale=0.9]
	\node[] at (1.25,3.25) {\LARGE $\mathcal{D}$};
	\begin{scope}[thick]
	\draw[->] (0,0) node[anchor=north]{$(u_0, v_0)$} -- (0,5)node[anchor = east]{$\Gamma$};
	\draw[->] (0,0) --node[anchor = north]{$v$} (3,3);
	\draw[->] (1.75,1.75) node[anchor=west]{$(u_0,v_1)$} --node[anchor=north]{$u$} (0,3.5)node[anchor = east]{$(u_1, v_1)$};
	\draw[->] (2.75,2.75) node[anchor=west]{$(u_0,v_2)$} -- (1,4.5);
	\end{scope}
	\begin{scope}[gray]
	\draw (2,2) -- (0.25,3.75);
	\draw (2.25,2.25) -- (0.5,4);
	\draw (2.5,2.5) -- (0.75,4.25);
	\draw(1.5,2) -- (2.5,3);
	\draw(1.25,2.25) -- (2.25,3.25);
	\draw(1,2.5) -- (2,3.5);
	\draw(0.75,2.75) -- (1.75,3.75);
	\draw(0.5,3) -- (1.5,4);
	\draw(0.25,3.25) -- (1.25,4.25);
	\draw(0,3.5) -- (1,4.5);
	\end{scope}
	\end{tikzpicture}
\end{minipage}
\caption{The double null foliation of the spacetime, with each point $(u,v)$ representing a $2$-sphere $S_{u,v}$ with area radius $r(u,v)$, and $\Gamma$ denoting the center of spherical symmetry.}
	\label{fig:DNintro}
\end{figure}
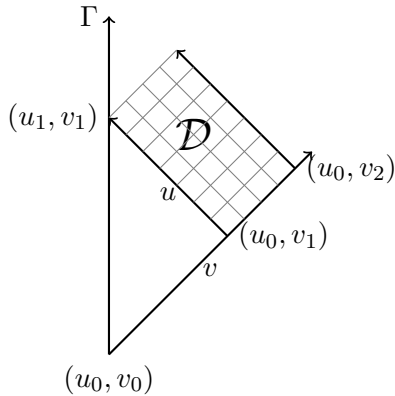

{\color{black} The Hawking mass is defined} as the quantity
\begin{equation*}
m(u,v)=\f{r}{2}\left(1+4\Omega^{-2}\partial_u r \partial_v r\right).
\end{equation*}

In \cite{DC91}, the following three properties are observed:
\begin{itemize}
\item[1)]  \begin{equation}\label{mass inequality}
R_{\a\b\gamma\delta}R^{\a\b\gamma\delta}(u,v)\geq \f{32\, m^2(u,v)}{r^6(u,v)}.
\end{equation}
\item [2)] The Hawking mass $m$ is monotonic in the trapped region $\mathcal{T}$: $\partial_u m(u,v)\geq 0$.
\item [3)] The apparent horizon $\mathcal{A}$ consists of points $(u_{\mathcal{A}} (v),v)$ such that $\partial_v r(u_{\mathcal{A}} (v),v)=0$ and hence $m(u_{\mathcal{A}} (v),v)=\frac{1}{2} r(u_{\mathcal{A}} (v),v)$.
\end{itemize}

From these properties, one can easily deduce the following uniform lower bound in $\mathcal{T}$:
\begin{equation}
\label{eq:prelimlowboundkretsch}
R_{\a\b\gamma\delta}R^{\a\b\gamma\delta}(u,v) \geq 8\frac{r^2(u_{\mathcal{A}}(v),v)}{r^6(u,v)}.
\end{equation}

It follows that the Kretschmann scalar blows up \emph{at least} as fast as $r^{-6}$ as we approach $\mathcal{S}$ along hypersurfaces of constant $v$. {\color{black}Recall} that in Schwarzschild, the $r^{-6}$ rate in the above lower bound is in fact {\color{black}precisely} attained.

\subsection{Polynomial upper bounds for the Kretschmann scalar} Following the arguments in the proof of {\color{black}the} \underline{qualitative} \textit{extension principle}\footnote{{\color{black}This} states that for characteristic initial data prescribed on initial incoming and outgoing hypersurfaces  {\color{black}{with}} $r\geq \e>0$, local existence towards the future can be proved {\color{black}for (\ref{ES})}.} in \cite{DC93}, fixing the final Bondi mass $M$ to be $1$ for notational convenience, one can prove that at $\mathcal{S}$:
$$R^{\alpha\beta\gamma\delta}R_{\alpha\beta\gamma\delta}(u,v)\lesssim \exp\big(\exp(\f{1}{r(u,v)})\big).$$
The extension principle proved in \cite{DC93} is qualitative in nature. To get a sharper upper bound, one needs to adapt a different approach and needs to improve all the estimates into \underline{quantitive, sharp} estimates. 

Recent work of the first author and Zhang \cite{AZ}, shows that the above \emph{double-exponential} upper bounds can be improved to \emph{polynomial} upper bounds (see Figure \ref{fig:anzhang}): 
\begin{thmx}[\cite{AZ}]
Consider the trapped region of dynamical black hole spacetimes of Christodoulou \cite{DC91}, solving (\ref{ES}) and let $(u_0,v_0)$ denote a point along the space singularity $\mathcal{S}$, away from the first singularity $\mathcal{S}_0$ and let $M=1$ denote the final Bondi mass. Then:
\begin{enumerate}[(i)]
\item There exist positive numbers $D_1$ and $D_2$, such that
 \begin{equation}
\label{eq:dphinearr0}
r^2(u_0,v_0)|\partial_u \phi|(u_0,v_0)\leq D_1, \quad \quad r^2(u_0,v_0)|\partial_v \phi|(u_0,v_0)\leq D_2.
\end{equation}
\item For $r_0>0$ suitably small, there exists a positive number $N_{u_0,v_0}$, such that 
\begin{equation}
\label{eq:upboundanzhang}
R^{\alpha\beta\gamma\delta}R_{\alpha\beta\gamma\delta}(u,v)\lesssim r(u,v)^{-N_{u_0,v_0}}
\end{equation}
for all $(u,v)$ lie both in the causal past of $(u_0,v_0)$  and the causal future of $\{r=r_0\}$.
\end{enumerate}
The values of $D_1, D_2$ and $N_{u_0,v_0}$ depend on the values of $\partial_u\phi$ and $\partial_v\phi$ in the causal past of $(u_0,v_0)$ along the hypersurface $\{r=r_0\}$.
\end{thmx}
\begin{remark}
The bounds \eqref{eq:upboundanzhang} and \eqref{eq:prelimlowboundkretsch} together imply that the blow-up of the Kretschmann scalar is bounded above and below by inverse polynomials in $r$. 
\end{remark}
\begin{remark}
The powers in $r$ appearing in \eqref{eq:dphinearr0} are sharp and are crucial for establishing \eqref{eq:upboundanzhang}.
\end{remark}
\begin{figure}[H]
	\begin{center}
\includegraphics[scale=0.75]{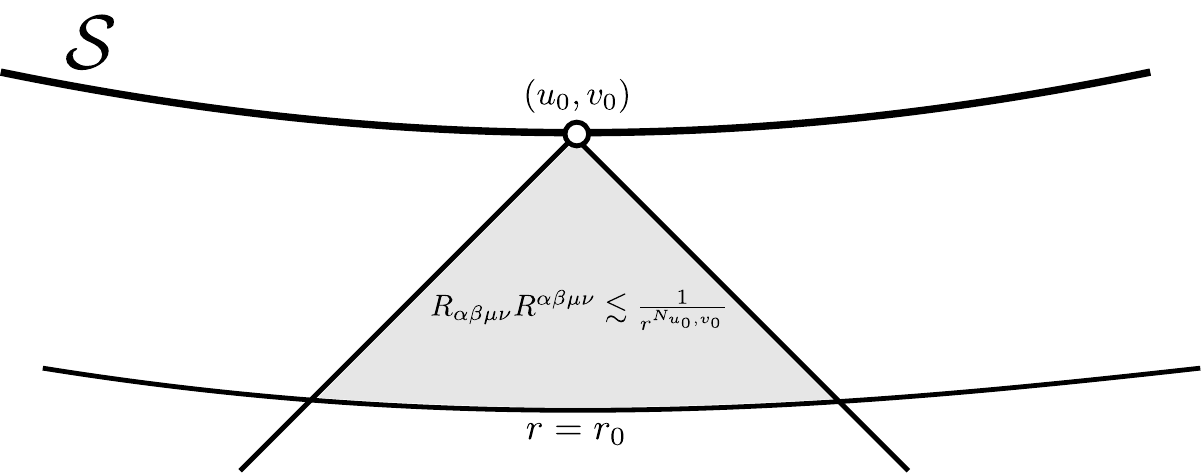}
\end{center}
\vspace{-0.2cm}
\caption{Local polynomial upper bound for Kretschmann scalar at $\mathcal{S}$.}
	\label{fig:anzhang}
\end{figure}

\subsection{Polynomial tails, mass inflation and precise blow-up of the Kretschmann scalar} In this section, we present a rough version of the main theorem of the present paper.

\begin{theorem}[Rough version of Theorem \ref{thm:precisemain}]
\label{main thm}
Consider the dynamical black hole spacetimes of Christodoulou \cite{DC91} with final Bondi mass $M=1$. Assume the following asymptotic scalar field behaviour along the event horizon $\mathcal{H}$:
\begin{equation}
\label{eq:priceassm}
D_1v^{-q} \leq \partial_v\phi|_{\mathcal{H}} \leq D_2 v^{-p},
\end{equation}
for some {\color{black}$1<p\leq q<{\color{black}3p-1}$}.

Then there exist positive numerical constants $0<\rho<\sigma$, with $\rho$ \underline{independent} and $\sigma$ dependent on {\color{black}the bounds in \eqref{eq:priceassm}}, so that we can estimate
\begin{align}
\liminf_{r\downarrow 0} \left[r^{6+\rho D_1^2v^{-2q} }R^{\a\b\mu\nu}R_{\a\b\mu\nu}\right](v,r)\gtrsim &\: 1,\\
\limsup_{r\downarrow 0} \left[r^{6+\sigma v^{-2p} }R^{\a\b\mu\nu}R_{\a\b\mu\nu}\right](v,r)\lesssim &\: 1
\end{align}
{\color{black}for suitably large $v$.}
\end{theorem}

We refer to Theorem \ref{thm:precisemain} for a more precise version of Theorem \ref{main thm} and Theorem \ref{thm:approachschw} for additional bounds on {\color{black} the Hawking mass}, the scalar field $\phi$ and the components of the metric $g$, which are crucial for obtaining Theorem \ref{thm:precisemain}.

\begin{remark}
Theorem \ref{main thm} {\color{black} establishes the existence} of a class of {\color{black} spherically symmetric} black hole {\color{black} spacetimes} with a spacelike singularity where the Kretschmann scalar blow-up is \emph{stronger} than in Schwarzschild. This may be contrasted with spacetime regions bounded by spacelike singularities at which the Kretschmann scalar blows up according to the Schwarzschild rate, with respect to a natural choice of area radius; see for example \cite{fourno16} and references therein.
\end{remark}

\begin{remark}
The upper bound of the estimate \eqref{eq:priceassm} was proved in \cite{MDIR05} in the setting of the Einstein--Maxwell--scalar field system and holds in particular for the black hole spacetimes of Christodoulou \cite{DC91}. While an $L^2$-version of the lower bound in \eqref{eq:priceassm} has been obtained for generic initial data in \cite{Luk2016b} in the same setting, it remains open (in $L^{\infty}$) in the setting of \eqref{ES}.

In \cite{aagprice}, it was shown that in the \emph{linearized setting} the scalar field has ``polynomial tails'' along the event horizon, i.e.\ the leading-order asymptotic behaviour of the scalar field along the event horizon is inverse polynomial. This is sometimes called ``Price's law'' in the literature as it originates from earlier heuristic work by Price \cite{Price1972}. In light of the linearized results in \cite{aagprice} and the $L^2$-lower bounds in \cite{Luk2016b}, we expect that for suitably regular and localized spherically symmetric initial data for \eqref{ES}, \eqref{eq:priceassm} should hold with $p=q=4$.
\end{remark}

\begin{remark}
Note that the constant $D_1$ in the lower bound for the Kretschmann scalar in Theorem \ref{main thm} arises directly from the lower bound {\color{black}assumed on} the scalar field along the event horizon. In the linearized setting of \cite{aagprice}, $D_1$ is proportional to the \emph{time-inverted Newman--Penrose constant} (introduced in \cite{aagprice}), which can either be expressed as an integral over the initial data in the black hole exterior (that is generically non-vanishing) or as a quantity defined at future null infinity \cite{aagprice, lukoh17}. The time-inverted Newman--Penrose constant can moreover be interpreted as a global conserved quantity for the wave equation \cite{paper-bifurcate}.
\end{remark}

\begin{remark}
The bounds in Theorem \ref{main thm} imply that the blow-up rate of the Kretschmann scalar is not constant along $\mathcal{S}$, and moreover, it converges to the Schwarzschild rate as we approach future timelike infinity along $\mathcal{S}$. The variation of the blow-up rate at different points of $\mathcal{S}$ illustrate how the blow-up mechanism is not of an ``ODE type'', i.e. it can \emph{not} be understood at the level of an ODE model equation, but instead can be thought of as being of ``PDE type''. This is in contrast with many of the blow-up mechanisms encountered in the setting of nonlinear dispersive equations, geometric evolutionary equations and fluids.
\end{remark}

\subsection{Previous work}
In this section, we give an overview of some related previous work on singularities and curvature blow-up in black hole interiors.

\subsubsection{Spherically symmetric models}
In \cite{DC91}, Christodoulou showed that for generic large data, spherically symmetric solutions to (\ref{ES}) have black hole regions bounded to the future by a spacelike singularity where the Kretschmann scalar blows up and the metric is $C^{0}$-inextendible under spherical symmetry; thus validating the $C^{ 0}$-formulation of the Strong Cosmic Censorship conjecture (SCC) for the spherically symmetric Einstein--scalar field system (see for example \cite{dl-scc} for a precise formulations of this conjecture). {\color{black}In \cite{DC99} Christodoulou showed moreover that generic initial data ({\color{black}in the bounded variation class}) lead to solutions that are either future geodesically complete or contain a black hole region, establishing therefore also the Weak Cosmic Censorship conjecture in this spherically symmetric setting (see \cite{schris} for a precise formulation of this conjecture).}

In a series of papers \cite{MD03, MD05c, MD12} Dafermos considered spherically symmetric dynamical black hole solutions to the Einstein--Maxwell--scalar field system with constant electromagnetic charge $0<e<M$, which may be viewed as toy models for dynamical non-spherically symmetric black holes with an angular momentum. He showed that dynamical black hole solutions approaching sub-extremal Reissner--Nordstr\"om (with non-zero charge) along the event horizon, with rates that agree with the upper bounds for the scalar field established in \cite{MDIR05}, have a non-empty Cauchy horizon across which the metric is $C^0$-\emph{extendible}, thus violating the $C^0$-formulation of the SCC. 

Under the additional assumptions of pointwise \emph{lower bounds} for the scalar field, however, which are consistent with the rates predicted in \cite{Price1972} and proved in the linearized setting in \cite{aagprice}, it was moreover shown that the Hawking mass blows up at the Cauchy horizon (this is known as ``mass inflation'') and hence, the Kretschmann scalar blows up, resulting in $C^1$-inextendibility of the metric, or inextendibility with $L^2_{loc}$ Christoffel symbols, which is in agreement with the Christodoulou reformulation of SCC in \cite{DC09}. The Einstein--Maxwell--scalar field setting of Dafermos was revisited by Luk--Oh \cite{Luk2016a,Luk2016b}, who obtained $C^2$-inextendibility for generic initial data by removing the additional lower bound \emph{assumptions} on the event horizon and \emph{proving} instead an $L^2$-lower bound. 

Black holes arising from asymptotically flat two-ended initial data for the Einstein--Maxwell--scalar field system must have charge $0\leq e<M$, see  \cite{kommemiphd} and Appendix A of \cite{Luk2016a} for a proof, i.e. they are sub-extremal in the limit. Nevertheless, if one considers incomplete data hypersurfaces or adds charge to the scalar field, the corresponding black hole may be \emph{extremal in the limit}. Such solutions were considered by the second author and Luk \cite{gajluk} and it was shown that under expected upper bound assumptions for the scalar field along the event horizon, the metric is $H^1_{loc}$ at the Cauchy horizon and it is moreover possible to extend the metric as a solution to the Einstein equations across the horizon. A general characterization of solutions to the Einstein--Maxwell--charged scalar field system was initiated in \cite{kommemiphd} and it provides a setting for studying more generally black hole solutions forming from one-ended asymptotically flat initial data. See \cite{moortel18, moortel19, moortel20} for recent results on the nature of black hole interior singularities in this setting. Note also the numerics in \cite{ches20}.

Let us moreover note that adding a cosmological constant term to the Einstein equations results in exponential decay of dynamical black hole solutions towards a stationary state \cite{HV2016}. Consequently, there can be a range of stability and instability phenomena at the Cauchy horizon \cite{Costa2015, Costa2014, Costa2014a}, depending on the precise {\color{black}value of the} exponential decay rate. 

\subsubsection{Linearized setting}
A first step towards understanding the interiors of black hole solutions outside of spherical symmetry is to develop a theory of stability and instability for solutions to the linear wave equation on the interior of a fixed black hole background. Fournodavlos--Sbierski considered the behaviour of linear waves $\phi$ on Schwarzschild interiors \cite{fournosbier19} and obtained the general blow-up profile $\phi$ at $r=0$, with logarithmic blow-up to leading-order.\footnote{We note that this blow-up profile is consistent with the behaviour in the spherically symmetric non-linear setting found in \cite{AZ}.}

Franzen established in \cite{franzen14} boundedness and $C^0$-extendibility of $\phi$ at the Cauchy horizon of sub-extremal Reissner--Nordstr\"om, whereas $H^1_{loc}$ blow-up was established by Luk--Oh \cite{lukoh17}. In contrast, it was shown by the second author in \cite{gaj17a} that $\phi$ is bounded in $H^1_{loc}$ at the inner horizon of extremal Reissner--Nordstr\"om, assuming the decay properties of \cite{aagprice}, with higher regularity depending delicately on the precise asymptotic behaviour of $\phi$ along the event horizon. The situation in the cosmological setting of near-extremal Reissner--Nordstr\"om--de Sitter spacetimes is similarly more delicately connected to precise asymptotics along the event horizon; see \cite{hprl18, costfranz16}.

 Stability and instability results for $\phi$ at the Cauchy horizon have been extended to non-spherically symmetric backgrounds: see for example \cite{luksbier16, dafshl16, hintz17, franzen19} for results in sub-extremal Kerr, \cite{gaj17b} for results in extremal Kerr and \cite{diasetal18} in sub-extremal Kerr--de Sitter.
 
 \subsubsection{Einstein equations in polarized axisymmetry}
 
{\color{black} Recently, \cite{fournoalex} Alexakis and Fournodavlos proved a stability result for the Schwarzschild singularity {\color{black}in the setting of the} Einstein vacuum equations in polarized axisymmetry. In contrast with the spherically symmetric setting, there is no longer a single natural parameter that measures the strength of the singularity (i.e. the area radius of the 2-spheres foliating the spacetime). 

The authors evolved small perturbations of Schwarzschild initial data on constant (small) $r$ level sets in the black hole interior via an ``asymptotically constant-mean-curvature spacetime foliation'', i.e.\ they considered a congruence of timelike geodesics emanating from the initial data hypersurface with a parameter $s$ and showed that the mean curvature $\mbox{tr}_g K$ of level sets of $s$ satisfies:
\[
\mbox{tr}_g K=\f32 s^{-\f32}+O(s^{-\f32+\f14}), 
\]
and the Kretschmann scalar blows up with the rate $\frac{1}{s^6}$ as $s\downarrow 0$. 

While in Schwarzschild spacetimes, the above coordinate $s$ coincides with the area radius $r$ of 2-spheres, this is no longer true in the dynamical setting. We refer to Section 1.2.3 of \cite{fournoalex} for additional comments on the blow-up rates of the Kretschmann scalar with respect to the area elements of the spheres in the asymptotically constant-mean-curvature foliation {\color{black}(an analogue of the area radius function in spherical symmetry)}.}

\subsubsection{Einstein equations without symmetry}
In \cite{fourno16}, Fournodavlos proved the backward stability of the Schwarzschild singularity for Einstein vacuum equations. In \cite{dl-scc}, Dafermos--Luk established $C^0$-extendibility across the Cauchy horizon for black hole spacetimes approaching sub-extremal Kerr solutions suitably fast along the event horizon.

\subsubsection{Cosmological singularities}
A related problem to understanding curvature blow-up near spacelike singularities inside black holes is the behaviour of past singularities in the setting of perturbations of cosmological spacetimes like the FLRW spacetimes. See for example \cite{ring17, rodspeck18, alho19} and references therein.

\subsection{\emph{Open problem}: interior of dynamical black holes approaching Schwarzschild outside of symmetry}
In light of the construction of dynamical black hole spacetimes approaching Schwarzschild in the exterior without any symmetries, following from a stability analysis the Schwarzschild exterior announced in \cite{dhrt} and similar results obtained recently in polarized axisymmetry \cite{klainszef17}, it is natural to investigate quantitatively the nature of the singularities in the corresponding black hole interiors near timelike infinity, {\color{black} and in particular, the effect of expected inverse polynomial behaviour of dynamical quantities along the event horizon on the blow-up rate of the Kretschmann scalar}. We believe that the quantitative upper and lower bound estimates developed in the present paper, together with the \emph{global} characterization of blow-up rates in $r^{-1}$ near $r=0$ may be useful for such an analysis outside of spherical symmetry. 

\section{Geometric preliminaries}

\subsection{Schwarzschild black holes}
The Schwarzschild black hole interior with mass $M>0$ is defined as the Lorentzian manifold $(\mathcal{M}_{\rm int},g_S)$, with $\mathcal{M}_{\rm int}=(\R_u\times \R_v\cap\{v+u<0\})\times \s^2$, with
\begin{equation*}
g_S=-\Omega^2_S(r_S(u,v))dudv+r_S^2(u,v) (d\theta^2+\sin^2\theta d\varphi^2),
\end{equation*}
where
\begin{equation*}
\Omega^2_S=\frac{2M}{r_S}-1
\end{equation*}
and $r_S$ is the unique solution to
\begin{equation*}
\frac{1}{2}(v+u)=r_*(r_S)=r_S+2M\log\left(\frac{2M-r_S}{2M}\right).
\end{equation*}
with $0<r_s<2M$.

 \begin{figure}[H]
	\begin{center}
\includegraphics[scale=0.7]{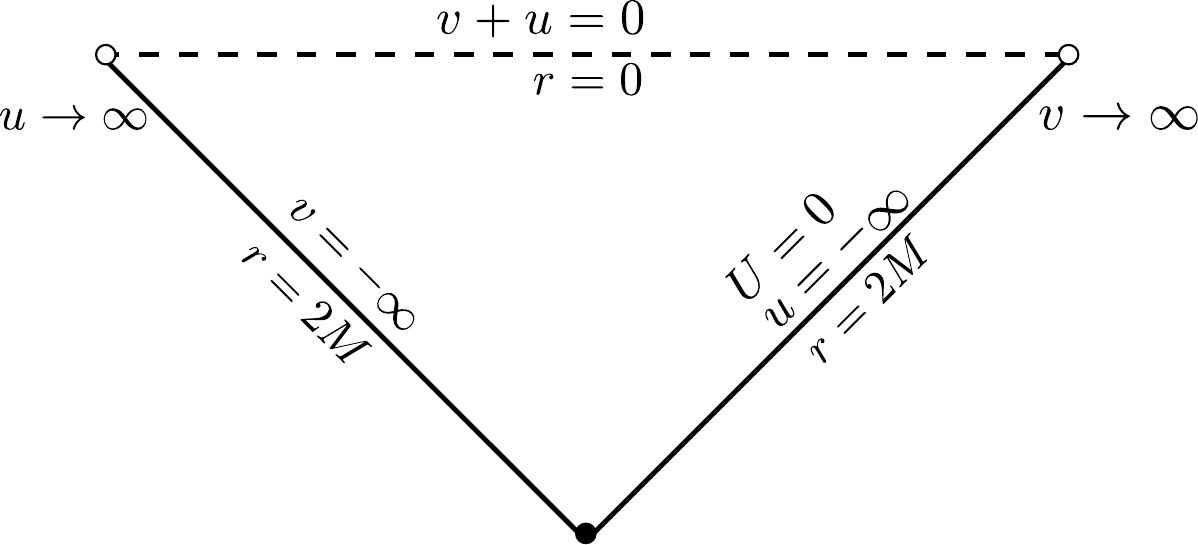} 
\end{center}
\vspace{-0.2cm}
\caption{A Penrose diagrammatic representation of the Schwarzschild black hole interior.}
	\label{fig:schw}
\end{figure}

Note that we can express
\begin{equation*}
\frac{r_S}{2M}\Omega^2_S(u,v)=e^{-\frac{r_S}{2M}}\cdot e^{\frac{1}{4M}(v-|u|)}\leq e^{\frac{1}{4M}(v-|u|)}.
\end{equation*}
Furthermore,
\begin{equation*}
\partial_ur_s=\partial_vr_s=-\frac{1}{2}\Omega^2_S,
\end{equation*}
so we can also estimate
\begin{equation}
\label{eq:durschwest}
|\partial_ur_s|=|\partial_vr_s|\leq \frac{M}{r_s} e^{\frac{1}{4M}(v-|u|)}.
\end{equation}

Let $U=4Me^{\frac{u}{4M}}$, then we can express 
\begin{equation*}
g_S=-\Omega^2_S(r_S(u,v))e^{\frac{|u|}{4M}}dUdv+r_S^2 (d\theta^2+\sin^2\theta d\varphi^2),
\end{equation*}
and we can extend $(\mathcal{M}_{\rm int},g_S)$ smoothly to the manifold-with-boundary 
\begin{equation*}
\overline{\mathcal{M}_{\rm int}}=\left([0,\infty)_U \times  \R_v\cap\left\{v+4M\log\left(\frac{U}{4M}\right)<0\right\} \right)\times \s^2.
\end{equation*}

\section{System of equations}
The Einstein--scalar system of equations \eqref{ES} for Lorentzian metrics of the form
\begin{equation*}
g=-\Omega^2(u,v)dudv+r^2(u,v)\big(d\theta^2+\sin^2\theta d\phi^2\big),
\end{equation*}
with $r$ and $\Omega^2$ strictly positive functions, reduces to a 1+1-dimensional system of equations for $(r,\Omega^2,\phi)$ on subsets of $\R^2$. We can split the system into:
\begin{enumerate}[1)]
\item Propagation equations for the metric components $r$ and $\Omega^2$,
\item Constraint equations for the metric components $r$ and $\Omega^2$,
\item Propagation equations for the scalar field $\phi$.
\end{enumerate}
\subsection{Propagation equations for the metric components} 
\begin{align}
\label{propeq:r1}
\partial_u(r\partial_v r)=&-\frac{1}{4}\Omega^2,\\
\label{propeq:r2}
\partial_v(r\partial_u r)=&-\frac{1}{4}\Omega^2,\\
\label{propeq:Omega}
\partial_u\partial_v\log \Omega^2=&-2\partial_u\phi\partial_v\phi+\frac{1}{2}r^{-2}\Omega^2+2r^{-2}\partial_ur \partial_v r.
\end{align}
It will be convenient to consider also the following equations for the rescaled quantities $r\Omega^2$ and $r^2\Omega^2$: 
\begin{align}
\label{propeq:Omegarescale1}
\partial_u\partial_v\log (r\Omega^2)=&\:\frac{1}{4}r^{-2}\Omega^2-2\partial_u\phi\partial_v\phi,\\
\label{propeq:Omegarescale2}
\partial_u\partial_v\log (r^2\Omega^2)=&-2r^{-2}\partial_ur \partial_v r-2\partial_u\phi\partial_v\phi.
\end{align}
The equations \eqref{propeq:Omegarescale1} and \eqref{propeq:Omegarescale2} follow from a combination of \eqref{propeq:Omega} with \eqref{propeq:r1}.

\subsection{Constraint equations for the metric components} 
\begin{align}
\label{conseq:r1}
\partial_u(\Omega^{-2}\partial_u r)=&-r \Omega^{-2}(\partial_u \phi)^2,\\
\label{conseq:r2}
\partial_v(\Omega^{-2}\partial_v r)=&-r \Omega^{-2}(\partial_v \phi)^2.
\end{align}
{\color{black} Equations \eqref{conseq:r1} and \eqref{conseq:r2} are known as the ``Raychaudhuri equations''.} 
\subsection{Propagation equations for the scalar field}
\begin{align}
\label{propeq:phi1}
\partial_u(r\partial_v \phi)=&-\partial_vr \partial_u \phi,\\
\label{propeq:phi2}
\partial_v(r\partial_u \phi)=&-\partial_ur \partial_v \phi.
\end{align}

Denote
\begin{equation}
\label{eq:defY}
Y=\frac{1}{-\partial_u r}\partial_u.
\end{equation}

It will be convenient to consider also the following propagation equation for the rescaled quantity $Y\phi$:
\begin{equation}
\label{propeq:Yphi}
\partial_v(r^2 Y\phi)=\left[-\frac{1}{4}r^{-1}\left(\frac{\Omega^2}{-\partial_u r}\right)+2r^{-1}\partial_vr\right] r^2 Y\phi+r\partial_v\phi.
\end{equation}
Equation \eqref{propeq:Yphi} follows from a combination of \eqref{propeq:phi2} with \eqref{propeq:r2}.

We also consider the Hawking mass, which is defined as follows:
\begin{equation}\label{Hawking mass}
m(u,v)=\f{r}{2}(1+4\Omega^{-2}\partial_u r \partial_v r).
\end{equation} 

\section{Initial data}
\label{sec:initialdata}
We consider a characteristic initial value problem for $(r,\widehat{\Omega}^2, \phi)$ satisfying the system of equations \eqref{propeq:r1}--\eqref{propeq:Omega},  \eqref{conseq:r1}--\eqref{propeq:phi2} with $\widehat{\Omega}^2$ replacing  ${\Omega}^2$ and $U$ replacing $u$, and initial data imposed on the hypersurfaces-with-boundary
\begin{align*}
H_0:=&\:\{(U,v)\in \R^2\,|\, U=0, v_0\leq v<\infty\},\\
\underline{H}_0:=&\:\{(U,v)\in \R^2\,|\, 0\leq U\leq U_0, v=v_0\}.
\end{align*}
We denote also $\mathcal{H}:=H_0$.

Let $M>0$ be a Schwarzschild mass parameter. We normalize the coordinates $(U,v)$ via the following gauge choices:
\begin{align}
\label{gaugeeq:Omega1}
\widehat{\Omega}^2(0,v)=&\widehat{\Omega}_S^2(0,v)=\:e^{-1+\frac{1}{4M} v}\quad \textnormal{for all $v\geq v_0$},\\
\label{gaugeeq:Omega2}
\widehat{\Omega}^2(U,v_0)=&\widehat{\Omega}_S^2(U,v_0)=\frac{2M}{r_s(U,v_0)} e^{\frac{-r_s(U,v_0)}{2M}}e^{\frac{1}{4M}v_0}\quad \textnormal{for all $0\leq U\leq U_0$,}
\end{align}
and we assume that $r$ approaches the Schwarzschild horizon radius along $H_0$, i.e.\
\begin{equation*}
\lim_{v \to \infty} r(0,v)=2M,
\end{equation*}
and moreover,
\begin{equation*}
\lim_{v \to \infty} \phi(0,v)=0.
\end{equation*}

We then take our freely prescribable data to be $rY\phi$ on $\underline{H}_0$, $\partial_v\phi$ on $H_0$ and $\partial_Ur(0,v_0)$. We making the following quantitative assumption on $\partial_v\phi|_{H_0}$: let $1<p\leq q<{\color{black}3}p-1$, then there exist positive dimensionless constants $D_{1}$ and $D_2$, such that
\begin{equation}
\label{eq:estdataphi}
D_1  (M^{-1}v)^{-q}\leq r\partial_v\phi(0,v)\leq D_2 (M^{-1}v)^{-p}.
\end{equation}

We moreover assume
\begin{equation}
\label{eq:estdataphi2}
\sup_{0\leq U\leq U_0}  |rY\phi|(U,v_0)+|\partial_Ur|(0,v_0)\leq D_3.
\end{equation}
for some dimensionless constant $D_3$ and, {\color{black}without loss of generality}, we take $rY\phi(U,v_0)>0$.

 \begin{figure}[H]
	\begin{center}
\includegraphics[scale=0.5]{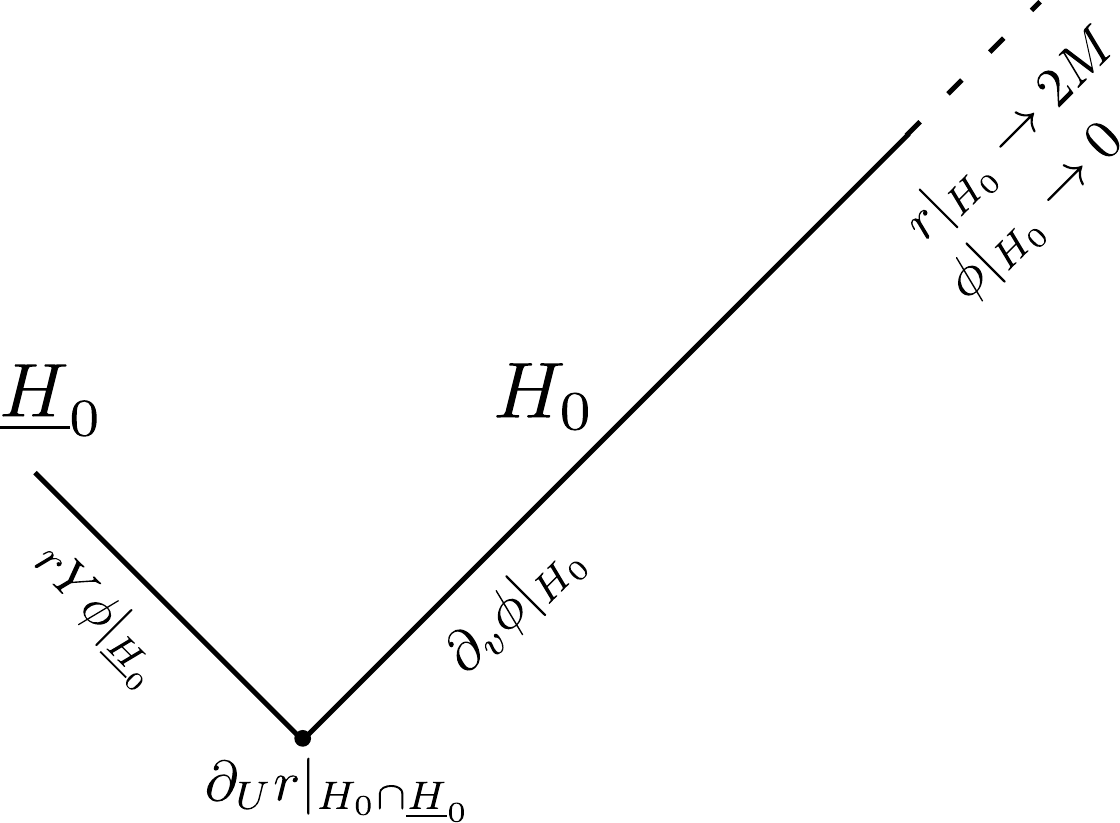} 
\end{center}
\vspace{-0.2cm}
\caption{The initial data.}
	\label{fig:contour}
\end{figure}

Using the constraint equations \eqref{conseq:r1} and \eqref{conseq:r2}, we obtain the following estimates for $\partial_v\phi(0,v)$ and $\partial_Ur(0,v)$:
\begin{lemma}
\label{lm:rdata}
There exists constants $c,C>0$ such that along ${H}_0$:
\begin{align}
\label{eq:estdatadvr}
 c D_1^2 (M^{-1}{v})^{-2q}\leq \partial_vr(0,v)\leq&\: C D_2^2(M^{-1}{v})^{-2p},\\
 \label{eq:estdatar}
 c MD_1^2(M^{-1}v)^{-2q+1}\leq 2M-r(0,v)\leq &\:C M D_2^2(M^{-1}{v})^{-2p+1},\\
 \label{eq:estdatadUr}
\left|\frac{\partial_Ur}{\widehat{\Omega}^2}(0,v)+\frac{1}{2}\right|\leq&\: CD_2^2 (M^{-1}v)^{-2p+1}+C D_3e^{-\frac{1}{4M}(v-v_0)}.
\end{align}
Furthermore, for all $\epsilon>0$, there exists a $v_1\geq v_0$ that depends moreover on $D_i$, $i=1,2,3$, such that for all $v\geq v_1$:
\begin{equation}
\label{eq:idlowboundYphi}
r^2Y\phi(0,v)\geq 4M(1-\epsilon)D_1(M^{-1}v)^{-q}.
\end{equation}
\end{lemma}

We will need to appeal to the following lemma:

\begin{lemma}
\label{lm:expint}
Let $f: [a,\infty) \to \R$ be a continuous function satisfying:
\begin{equation*}
f(x)\leq B x^{-p}
\end{equation*}
for all $x\in [a,\infty)$, with $p,B>0$.

Then for all $\alpha>0$, there exist constant $C>0$ such that
\begin{align}
\label{eq:expint1}
\int_a^b e^{- \alpha x} f(x)\,dx \leq &\:C\alpha^{-1}\cdot Be^{- \alpha a} a^{-p},\\
\label{eq:expint2}
\int_a^b e^{\alpha x} f(x)\,dx \leq &\: C\alpha^{-1}\cdot  Be^{\alpha b} b^{-p}.
\end{align}
\end{lemma}
\begin{proof}
See Appendix \ref{app:basic}. 
\end{proof}

\begin{proof}[Proof of Lemma \ref{lm:rdata}]
By integrating \eqref{conseq:r2} along $H_0$ (with $\Omega$ replaced by $\widehat{\Omega}$) and using \eqref{eq:estdataphi}, it follows that $e^{-\frac{1}{4M} v}\partial_vr(0,v)$ attains a finite limit as $v\to \infty$, which moreover is equal to 0, because $r(0,v)\to 2M$ as $v\to \infty$. Furthermore, by \eqref{conseq:r2}, we must have that $r(0,v)\geq 2M$.

We can therefore integrate \eqref{conseq:r2} and apply \eqref{eq:estdataphi} and Lemma \ref{lm:expint} to estimate:
\begin{equation*}
\begin{split}
 c D_1^2e^{-\frac{1}{4M} v} (M^{-1}{v})^{-2q}\leq e^{-\frac{1}{4M} v}\partial_vr(0,v)\leq &\: C D_2^2e^{-\frac{1}{4M} v} (M^{-1}{v})^{-2p}
\end{split}
\end{equation*}
for all $v\geq v_0$ and obtain \eqref{eq:estdatadvr}. We integrate \eqref{eq:estdatadvr} in $v$ to obtain \eqref{eq:estdatar}.

In order to estimate $\partial_Ur(0,v)$, we apply \eqref{propeq:r2} along $H_0$, with $u$ replaced by $U$ and $\Omega^2$ replaced by $\widehat{\Omega}^2$:
\begin{equation*}
\begin{split}
r\partial_U r(0,v)=&\:r\partial_U r(0,v_0)-\int_{v_0}^{v}\frac{1}{4} e^{-1+\frac{1}{4M}v'}\,dv'\\
=&\:r\partial_U r(0,v_0)-M(e^{-1+\frac{1}{4M}v}-e^{-1+\frac{1}{4M}v_0})
\end{split}
\end{equation*}
and therefore
\begin{align*}
\left|\frac{r\partial_Ur}{\widehat{\Omega}^2}(0,v)+M\right|\leq&\: CM(1+\partial_U r(0,v_0)) e^{-\frac{1}{4M}v},\\
\left|\frac{\partial_Ur}{\widehat{\Omega}^2}(0,v)+\frac{1}{2}\right|\leq&\: CD_2^2 (M^{-1}v)^{-2p+1}+ CD_3  e^{-\frac{1}{4M}v},
\end{align*}
where we applied \eqref{eq:estdatar} and to arrive at the second inequality.

By \eqref{propeq:Yphi}, together with \eqref{eq:estdatadvr}, \eqref{eq:estdatar} and \eqref{eq:estdatadUr}, we have that along $H_0$:
\begin{equation}
\label{eq:YphialongH0}
|\partial_v(e^{\frac{1}{4M}v}r^2Y\phi)-e^{\frac{1}{4M}v}r\partial_v\phi|\leq CM^{-1}[D_2^2 (M^{-1}v)^{-2p+1}+D_3e^{-\frac{1}{4M}v}]|e^{\frac{1}{4M}v}r^2Y\phi|.
\end{equation}
We integrate along $H_0$ and use Lemma \ref{lm:expint} and \eqref{eq:estdataphi} to obtain:
\begin{equation*}
\begin{split}
|e^{\frac{1}{4M}v}r^2Y\phi|(0,v)\leq &\: |e^{\frac{1}{4M}v_0}r^2Y\phi|(0,v_0)+ {\color{black}\int_{v_0}^v D_2e^{\frac{1}{4M}v'}(M^{-1}v')^{-p}\,dv'}\\
&+CM^{-1}\int_{v_0}^v[D_2^2 (M^{-1}v')^{-2p+1}+D_3e^{-\frac{1}{4M}v}]|e^{\frac{1}{4M}v'}r^2Y\phi|(0,v')\,dv'\\
= &\:  |e^{\frac{1}{4M}v_0}r^2Y\phi|(0,v_0)+CM D_2 {\color{black}e^{\frac{1}{4M}v}}(M^{-1}v)^{-p} \\
&+CM^{-1}\int_{v_0}^v[D_2^2 (M^{-1}v)^{-2p+1}+D_3e^{-\frac{1}{4M}v}]|e^{\frac{1}{4M}v}r^2Y\phi|(0,v')\,dv'
\end{split}
\end{equation*}
We can finally apply a Gr\"onwall inequality to obtain:
\begin{equation}
\label{eq:upboundYphialongH0}
{\color{black}|e^{\frac{1}{4M}v}r^2Y\phi|(0,v) \leq C(D_2,D_3,v_0,p) M(1+(M^{-1}v)^{-p} e^{\frac{1}{4M}v})}.
\end{equation}

{\color{black}We apply \eqref{eq:YphialongH0} again, together with the lower bound in \eqref{eq:estdataphi} and the upper bound in \eqref{eq:upboundYphialongH0}, to obtain:
\begin{equation}
\label{eq:lowboundYphialongH0}
\begin{split}
r^2Y\phi(0,v)\geq&\: e^{-\frac{1}{4M}(v-v_0)}r^2Y\phi(0,v_*)+ e^{-\frac{1}{4M}v}\int_{v_0}^v D_1e^{\frac{1}{4M}v'}(M^{-1}v')^{-q}\,dv'\\
&- CM^{-1}e^{-\frac{1}{4M}v}\int_{v_0}^v[D_2^2 (M^{-1}v)^{-2p+1}+D_3e^{-\frac{1}{4M}v}]|e^{\frac{1}{4M}v}r^2Y\phi|(0,v')\,dv' \\
\geq&\: e^{-\frac{1}{4M}(v-v_0)}rY\phi(0,v_0)+ 4Me^{-\frac{1}{4M}v}\int_{v_0}^v\frac{d}{dv}\left(D_1e^{\frac{1}{4M}v'}(M^{-1}v')^{-q}\right)\,dv'\\
&- C(D_2,D_3,v_0,p)M[e^{-\frac{1}{4M}v}+(M^{-1}v)^{-3p+1}]\\
\geq &\: 4MD_1(M^{-1}v)^{-q}+e^{-\frac{1}{4M}(v-v_0)}rY\phi(0,v_0)-4MD_1e^{-\frac{1}{4M}(v-v_0)}(M^{-1}v_0)^{-q}\\
&- C(D_2,D_3,v_0,p)M(M^{-1}v)^{-3p+1}.
\end{split}
\end{equation}

Using that $q<3p-1$, we conclude that for all $\epsilon>0$, there exists $v_1\geq v_0$ depending on $\epsilon$, $v_0$, $D_1$, $D_2$ and $D_3$ such that for all $v\geq v_1$:}
\begin{equation*}
r^2Y\phi(0,v)\geq 4M(1-\epsilon)D_1(M^{-1}v)^{-q}.
\end{equation*}
\end{proof}

\section{Precise statements of the main theorems}
Consider solutions $(r, \widehat{\Omega}^2, \phi)(U,v)$ to the system of equations \eqref{propeq:r1}--\eqref{propeq:Omega},  \eqref{conseq:r1}--\eqref{propeq:phi2} arising from initial data on $H_0$ and $\underline{H}_0$, as prescribed in Section \ref{sec:initialdata}. We denote with
\begin{equation*}
g=-\widehat{\Omega}^2(U,v) dUdv+r^2(U,v)(d\theta^2+\sin^2\theta d\varphi^2)
\end{equation*}
the Lorentzian metric corresponding to $(r, \widehat{\Omega}^2, \phi)$. Then $(g,\phi)$ form a solution to the Einstein--scalar field system of equations \eqref{ES}.

We denote with $R_{\alpha \beta \mu \nu}[g]$, $\alpha,\beta,\mu,\nu\in \{0,1,2,3\}$, the components of the Riemann curvature tensor corresponding to $g$.
\begin{theorem}
\label{thm:precisemain}
There exists a unique smooth solution $(r, \widehat{\Omega}^2, \phi)(U,v)$ in $[0,U_0]_U\times [v_0,\infty)_v\cap \{r>0\}$ to characteristic initial data prescribed in Section \ref{sec:initialdata} and the following bounds hold for the corresponding $R_{\alpha \beta \mu \nu}[g]$: there exist constants $c,C,\rho>0$ depending only on $M$ and a constant $\sigma>0$ depending on $M$, $D_2$ and $D_3$, and $v_1>v_0$ suitably large, such that for all $v\geq v_1$
\begin{equation*}
 \frac{cM^2}{r^{6+\rho D_1^2(v/M)^{-2p} }}\leq \sum_{\alpha,\beta,\mu,\nu=0}^3R^{\alpha\beta \mu \nu}[g]R_{\alpha \beta \mu \nu}[g]\leq \frac{ CM^2}{r^{6+\sigma (v/M)^{-2p} }}.
\end{equation*}
\end{theorem}

\begin{remark}
While Theorem \ref{thm:precisemain} applies in particular to the interior of dynamical black hole spacetimes of Christodoulou \cite{DC91} with the additional assumption of quantitative upper and lower bounds on $\phi$ along the event horizon, it actually provides a self-contained treatment of the black hole interior regions corresponding to dynamical black holes spacetimes approaching Schwarzschild.
\end{remark}

The bounds in Theorem \ref{thm:precisemain} rely on upper and lower bounds for the triple $(r,\Omega^2,\phi)$ and its derivatives. These are summarized in Theorem \ref{thm:approachschw} below. Theorem \ref{thm:precisemain} follows directly from Proposition \ref{thm 8.1}.
\begin{theorem}
\label{thm:approachschw}
Let $(r, \widehat{\Omega}^2, \phi)$ denote the solution from Theorem \ref{thm:precisemain}. Then there exist constants $c,C,\rho>0$ depending only on $M$, a constant $\sigma>0$ depending on $M$, $D_2$ and $D_3$, and $v_1>v_0$ suitably large, such that the following bounds hold in $(0,U_0]_U\times [v_1,\infty)_v\cap \{r>0\}$, with $u=-4M \log(\frac{4M}{U})$ and $\Omega^2=\frac{dU}{du}\widehat{\Omega}^2$:
\begin{align*}
|r\partial_ur+M|(u,v)\leq&\: M(r/M)^{\frac{1}{100}}+ CM (v/M)^{-p},\\
|r\partial_vr+M|(u,v)\leq&\: M(r/M)^{\frac{1}{100}}+ CM (v/M)^{-p},\\
 c M \left(\frac{r}{M}\right)^{\sigma (v/M)^{-2p}} \leq r\Omega^2(u,v)\leq&\: C M \left(\frac{r}{M}\right)^{\rho D_1^2 (v/M)^{-2q}} &&\textnormal{in $\left\{r\leq \frac{1}{10}M\right\}$},\\
  c M \left(\frac{r}{M}\right)^{-\rho D_1^2 (v/M)^{-2q}} \leq m(u,v) \leq&\: C M \left(\frac{r}{M}\right)^{-\sigma (v/M)^{-2p}},\\
\rho D_1 M(v/M)^{-q} \leq  |r^2\partial_v\phi|(u,v)\leq&\:  \sigma M (v/M)^{-p}\\
\rho D_1 M(v/M)^{-q} \leq  |r^2\partial_u\phi|(u,v)\leq&\: \sigma M(v/M)^{-p} &&\textnormal{in $\left\{r\leq \frac{1}{10}M\right\}$},\\
\rho D_1 (v/M)^{-q} \leq  |rY\phi|(u,v)\leq&\:  \sigma (v/M)^{-p} &&\textnormal{in $\left\{r>\frac{1}{10}M\right\}$}.
\end{align*}
\end{theorem}
Theorem \ref{thm:approachschw} follows directly by combining the results of Theorem \ref{partial r}, Propositions \ref{Omega 03},  \ref{mass inflation}, \ref{prop:lowboundsposr0}, \ref{prop:lowboundsposr} and \ref{phi}.

\begin{remark}
The lower bound for the Hawking mass $m$ in Theorem \ref{thm:approachschw} implies that $m$ blows up at $r=0$, which is a phenomenon known as \emph{mass-inflation}. Note that this contrasts the Schwarzschild setting, where $m$ is constant (and equal to the Schwarzschild mass parameter $M$). Mass-inflation in the dynamical setting can be viewed as the reason why the Kretschmann scalar in the dynamical setting blows up at a stronger rate than in the static Schwarzschild setting.
\end{remark} 

\begin{remark}
{\color{black}In \cite{AZ} it is shown that the quantities $r\partial_u r$ and $ r\partial_v r$ attain finite limits along $\mathcal{S}$ (where $r=0$), so from the first two estimates in Theorem \ref{thm:approachschw} it follows that they moreover approach their corresponding Schwarzschild values:
\begin{align*}
 r\partial_u r|_{\mathcal{S}}(v)\rightarrow -M,\\
 r\partial_u r|_{\mathcal{S}}(v)\rightarrow -M,
\end{align*}
as $v\to \infty$.}
\end{remark} 

\section{Main ideas in the proofs of Theorem \ref{thm:precisemain} and \ref{thm:approachschw}}
As a first step, we establish $L^{\infty}$ estimates for the quantities $(r,\Omega,\phi)$ \emph{away from $r=0$} via a continuity argument. We \emph{assume} bootstrap estimates for $(r,\Omega,\phi)$ in a spacetime region $$D_{v_{\infty},r_0}=[0,U_0]_U\times [v_0,v_{\infty}]\cap\{r\geq r_0\},$$ with $r_0$ an arbitrarily small positive number, which are certainly valid \emph{locally} (i.e.\ for small values of $v_{\infty}$) and we close the bootstrap estimates by \emph{improving} them in $D_{v_{\infty}}$. We then complete the continuity argument by taking $v_{\infty}\to \infty$. 

As a second step, we employ the estimates on $\{r=r_0\}$ to extend $(r,\Omega,\phi)$ to a larger region $D_{\infty,0}=[0,U_0]_U\times [v_0,v_{\infty}]\cap\{r>0\}$ and derive appropriate upper and lower bound estimates (without the need for bootstrap assumptions).

\subsection{Red-shift region}
We consider a region where $v-|u|<-\delta^{-1}$, with $\delta>0$ suitably small. This corresponds to the spacetime region where the \emph{red-shift effect} {\color{black}plays a role}. The red-shift effect has been used in a similar setting in previous works \cite{MD03, MD05c,Luk2016a} to establish smallness and decay of the difference quantities $(r-r_S,\log \frac{\Omega^2}{\Omega^2_s})$ in the red-shift region, together with upper bound estimates on derivatives of $\phi$. We follow an analogous approach to obtain sufficient control on the curve $v-|u|=-\delta^{-1}$. We moreover identify precisely the location of the apparent horizon $\mathcal{A}$ in this region.

\subsection{No-shift region}
Next, we consider a region where $v-|u|>-\delta^{-1}$ but $r\geq r_0$, which we call the \emph{no-shift region}. The estimates in this region deviate from the estimates in the ``no-shift regions'' of \cite{MD03, MD05c,Luk2016a} and other arguments in the Einstein--Maxwell--scalar field setting as we can take $r_0$ arbitrarily small and $r$ will attain arbitrarily small values. In this region, we derive new estimates to identify the leading-order behaviour of (derivatives of) $(r,\Omega^2,\phi)$ in $r$ with error terms that decay in $v$ and $|u|$.

The upper bound estimates in the red-shift and no-shift region allow us to complete the continuity argument, as discussed above.

\subsection{Lower bounds}
We subsequently establish additional lower bound estimates in $\{r\geq r_0\}$ for the derivatives of $\phi$, {\color{black}propagating the lower bounds along $\mathcal{H}$}. Here it is important to keep precise track of the dependence of our estimates on $r_0$.

\subsection{Small-$r$ region}
The estimates in this region may be interpreted as \emph{global versions} of {\color{black}the} main local estimates in \cite{AZ}, {\color{black} incorporating appropriate $u$-weights and $v$-weights that reflect the decay behaviour along $\{r=r_0\}$ derived in the previous steps.} To calculate the precise {\color{black}value} for {\color{black}the} blow-up rate of {\color{black}the} Kretschmann scalar, given initial data, we trace the change of every quantity very carefully.  We employ renormalized equations to trace the changes of lower order terms. For example, in the small-$r$ region, to estimate $r\partial_v r$ and $r\partial_u r$. We prove \textbf{Proposition} \ref{partial r}:

\noindent For $(u,v)\in J^-(\t q)$, with $|u|\geq |u_1|+2v_{r=0}(u_1)$ and $r(u,v)\leq r_0$, we have
\begin{equation}\label{d 1.11}
|r\partial_u r(u,v)+M-f_1(u_{\t q})|\leq M[M^{-1}r(u,v)]^{\f{1}{100}},
\end{equation}
\begin{equation}\label{d 1.12}
|r\partial_v r(u,v)+M-f_2(v_{\t q})|\leq M[M^{-1}r(u,v)]^{\f{1}{100}},
\end{equation}
where 
\begin{align*}
|f_1(u_{\t q})|\lesssim& M (M^{-1}v_{\t q)^{-2p}}\lesssim M|M^{-1}[u_{\t q}-u_1-v_0(u_1)]|^{-2p}, \\
|f_2(v_{\t q})|\lesssim& M|M^{-1}[u_{\t q}-u_1-v_0(u_1)]|^{-2p} \lesssim M (M^{-1}v_{\t q})^{-2p}.
\end{align*}

\begin{center}
\begin{minipage}[!t]{0.4\textwidth}
\begin{tikzpicture}[scale=0.95]
\draw [white](-1, 0)-- node[midway, sloped, above,black]{$\big(u_1, v_0(u_1)\big)$}(1, 0);
\draw [white](2.9, 0)-- node[midway, sloped, above,black]{$\t q$}(3.1, 0);
\draw [white](1.7, 0)-- node[midway, sloped, above,black]{$\mathcal{S}$}(2, 0);
\draw [white](5.5, 0.2)-- node[midway, sloped, above,black]{$i^+$}(7, 0.2);
\draw (0,0) to [out=-5, in=195] (5.5, 0.5);
\draw [white](0, -3)-- node[midway, sloped, above,black]{$\mathcal{H}$}(7, -3);
\draw [white](0, -2.8)-- node[midway, sloped, above, black]{$r=r_0$}(0.5,-2.8);
\draw [white](-2.5, -2.5)-- node[midway, sloped, above, black]{$u=u_1$}(0,0);
\draw [white](3.58, -0.70)-- node[midway, sloped, below, black]{$\t q_2$}(3.78, -0.70);
\draw [white](1.29, -1.79)-- node[midway, sloped, below, black]{$\t q_1$}(1.38, -1.79);
\draw [thick] (-2.5,-2.5)--(0,-5);
\draw [dashed](-2.5, -2.5)--(0,0);
\draw [thick] (5.5, 0.5)--(0,-5);
\draw (-2,-2) to [out=-5, in=215] (5.5, 0.5);
\draw[fill] (0,0) circle [radius=0.08];
\draw[fill] (5.5, 0.5) circle [radius=0.08];
\draw[fill] (-2,-2) circle [radius=0.08];
\draw[fill] (3.68,-0.68) circle [radius=0.08];
\draw[fill](1.29, -1.71) circle [radius=0.08];
\fill[black!30!white] (2.47, -0.53)--(3,0)--(3.53, -0.53)--(3,-1.06);
\draw [thin](1.29, -1.71)--(3,0);
\draw[fill] (3,0) circle [radius=0.08];
\draw [thin](2.47, -0.53)--(3,-1.06);
\draw [thin](3.53, -0.53)--(3,-1.06);
\draw [thin](3.68, -0.68)--(3,0);
\draw[fill] (3, -1.06) circle [radius=0.08]; 

\end{tikzpicture}
\end{minipage}
\begin{minipage}[!t]{0.5\textwidth}
\end{minipage}
\hspace{0.05\textwidth}
\end{center}
Here deriving the decay rates of $f_1(u_{\t q})$ and $f_2(v_{\t q})$ is crucial for the proof. Instead of using equation $\partial_u(r\partial_v r+M)$, we use equation of $\partial_u(r\partial_v r+M-r/2)$ to trace the lower order term and to explore the cancellations. 

Similarly, in \textbf{Proposition} \ref{Omega 02} and \textbf{Proposition} \ref{Omega 03}, we prove sharp estimates for both the lower and upper bounds of $\O^2(u,v)$:
for $(u,v)$, with $|u|\geq |u_1|+2v_0(u_1)$ and $r(u,v)\leq r_0$, it holds {\color{black}that}
$$\f{1}{[M^{-1}r(u,v)]^{1-{\f{6\t D_1 \t D_2}{1-\epsilon}\cdot\f{1}{M^{-2p}|u-u_1+v_0(u_1)|^{2p}}}}} \lesssim\O^2(u,v)\lesssim  \f{1}{[M^{-1}r(u,v)]^{1-{\f{D_1' D_2'}{6(1+\e)}\cdot\f{1}{M^{-2p}|u-u_1+v_{r=0}(u_1)|^{2q}}}}},$$
where $0\leq \e\ll 1$ and $\t D_1, \t D_2, D'_1, D'_2$ constants {\color{black}defined} in Proposition \ref{phi}. To obtain the inequalities above, instead of using equation for $\partial_u \partial_v \log(\O^2)$, we use renormalized equations of $\partial_u \partial_v \log(r\O^2)$ and $\partial_u \partial_v \log(r^2 \O^2)$ to cancel some borderline terms. \textit{The use of these renormalized equations is new}.

By {\color{black} applying an} algebraic calculation {\color{black}from} \cite{DC91}, the Kretschmann scalar obeys
$$R^{\a\b\mu\nu}R_{\a\b\mu\nu}(u,v)\geq \f{32 m(u,v)^2}{r(u,v)^6}.$$
Here the Hawking mass $m(u,v)=\f{r}{2}(1+4\Omega^{-2}\partial_u r \partial_v r)$ is defined in \eqref{Hawking mass}. Clearly, the sharp upper bound on $\O^2(u,v)$ would lead to a sharp lower bound on the Kretschmann scalar. And it is \underline{different} from the Schwarzschild value.

\section{Estimates away from the singularity}\label{Section7}

\subsection{Bootstrap assumptions}
Let $\frac{r_0}{M}$ be an arbitrarily small number such that
\begin{equation*}
r_0< \inf_{0\leq U\leq U_0}r(U,v_0).
\end{equation*} 

By a standard local existence argument, there exist a smooth solution $(r,\hat{\Omega},\phi)$ satisfying the system of equations \eqref{propeq:r1}--\eqref{propeq:Omega},  \eqref{conseq:r1}--\eqref{propeq:phi2}, with initial data as prescribed in Section \ref{sec:initialdata}, in $[0,U_0]\times [v_0,v_0+\epsilon)$, for $\epsilon>0$ suitably small and hence the set
\begin{equation*}
D_{v_{\infty},r_0}=[0,U_0)\times [v_0,v_{\infty})\cap\{r\geq r_0\}
\end{equation*}
is well-defined and non-empty for $v_{\infty}$ suitably small. We will now assume $D_{v_{\infty},r_0}$ is well-defined for some $v_{\infty}>v_0$, with $(r,\hat{\Omega},\phi)$ a smooth solution.

We moreover denote the apparent horizon in $D_{v_{\infty},r_0}$ as follows:
\begin{equation*}
\mathcal{A}=\{\partial_vr=0\}\cap D_{v_{\infty},r_0}.
\end{equation*}

By \eqref{propeq:r1} it follows that $\mathcal{A}$ has only (possible empty) spacelike or outgoing null segments.

\begin{lemma}
Suppose $(r,\Omega^2,\phi)$ is a well-defined smooth solution to the system \eqref{propeq:r1}--\eqref{propeq:Omega}, \eqref{conseq:r1}--\eqref{propeq:phi2} in $D_{v_{\infty},r_0}$ corresponding to initial data as prescribed in Section \ref{sec:initialdata}. Then, for $U_0$ suitably small, 
\begin{equation*}
\gamma_{r_0}:=\{r=r_0\}\cap D_{v_{\infty},r_0}
\end{equation*}
is either empty or a connected, spacelike curve intersecting $\{U=U_0\}$ at $v>v_0$, and, if non-empty, we can write
\begin{equation*}
D_{v_{\infty},r_0}=[0,U_0)\times [v_0,v_{\infty})\cap J^-(\gamma_{r_0}).
\end{equation*}
\end{lemma}
\begin{proof}
First of all, $\partial_vr(U,v_0)>0$ for $|U_0|$ suitably small by $\partial_vr(0,v_0)>0$ and continuity. We also have that $\partial_Ur<0$ in $D_{v_{\infty},r_0}$ by \eqref{propeq:r2}. Hence, $r>r_0$ in $D_{v_{\infty},r_0}\cap \{\partial_vr\geq 0\}=D_{v_{\infty},r_0}\cap J^-(\mathcal{A})$. Since $\{r=r_0\}\subset J^+(\mathcal{A})\setminus \mathcal{A}$ and $\partial_vr,\partial_Ur<0$ in $ J^+(\mathcal{A})\setminus \mathcal{A}$, we can conclude by Rolle's theorem that $\{r=r_0\}$ must be connected and moreover spacelike. By $\partial_vr,\partial_Ur<0$ it moreover follows that $r\geq r_0$ in $J^-(\gamma_{r_0})$.
\end{proof}

It will be more convenient to work with an Eddington--Finkelstein-type coordinate $u$ instead of $U$ in $D_{v_{\infty},r_0}\cap \{U>0\}$, which is related to $U$ as follows: 
\begin{equation}
\label{eq:reluU}
u=-4M \log \left(\frac{4M}{U}\right).
\end{equation}
Accordingly, we define
\begin{equation}
\label{eq:relhatomega}
\Omega^2(u(U),v):=\widehat{\Omega}^2(U,v)\frac{dU}{du}=\widehat{\Omega}^2(U(u),v)e^{\f{-|u|(U)}{4M}}
\end{equation}
and we consider $(r,\Omega^2,\phi)$ as functions of $(u,v)$, satisfying \eqref{propeq:r1}--\eqref{propeq:Omega},  \eqref{conseq:r1}--\eqref{propeq:phi2}. We moreover denote $u_0:=u(U_0)$.

We define for $\delta>0$
\begin{align*}
\mathcal{R}_{\delta}=&\:D_{v_{\infty},r_0}\cap \{v-|u|\leq -\delta^{-1}\},\\
\mathcal{N}_{\delta,r_0}=&\: D_{v_{\infty},r_0}\setminus \mathcal{R}_{\delta},
\end{align*}
and we denote
\begin{equation*}
\gamma_{\delta}:= D_{v_{\infty},r_0}\cap\{v-|u|= -\delta^{-1}\}.
\end{equation*}

We moreover introduce the constant $r_{\delta}:=r_S(u,v_{\gamma_{\delta} (u)})$.
 \begin{figure}[H]
	\begin{center}
\includegraphics[scale=0.5]{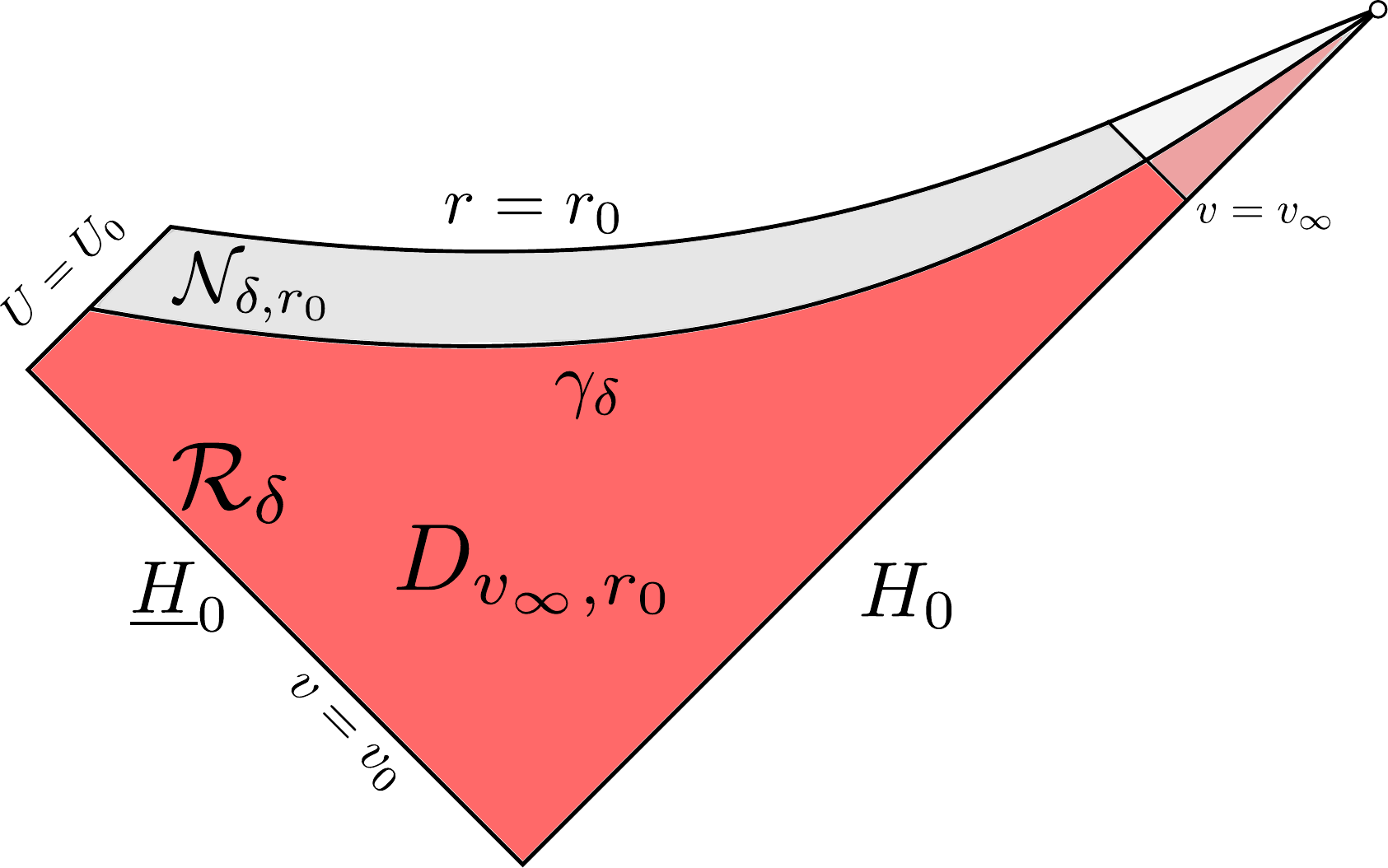}
\end{center}
\vspace{-0.2cm}
\caption{The domain under consideration, away from $r=0$.}
	\label{fig:contour}
\end{figure}

We make the following \textbf{bootstrap assumptions} on the solution $(r,\Omega,\phi)$:
\begin{align}
\label{ba:dvphi}
r^2| \partial_v\phi|(u,v)\leq &\: 2{\Delta_1}M (M^{-1}v)^{-p}\quad \textnormal{for all $(u,v)\in D_{v_{\infty},r_0}$},\\
\label{ba:Yphi} 
r|Y\phi|(u,v)\leq &\: 2{\Delta_1} (M^{-1}v)^{-p}\quad \textnormal{for all $(u,v)\in D_{v_{\infty},r_0}$},\\
\label{ba:Omega}
\left| \partial_u \log \left(\frac{\Omega^2}{\Omega^2_S}\right)\right|\leq&\: \Delta_2 e^{\frac{1}{4M}(v-|u|)}(M^{-1}v)^{-2p+1}\quad \textnormal{for all $(u,v)\in \mathcal{R}_{\delta}\subseteq D_{v_{\infty},r_0}$}.
\end{align}
where $\Delta_1,\Delta_2$ are dimensionless constants that will be chosen suitably large later.

\subsection{Estimates for $r$ and $\Omega^2$} 
\label{sec:estromega}
We will first derive estimates for $r$ and $\Omega^2$, assuming \eqref{ba:dvphi} and \eqref{ba:Yphi}.
\begin{proposition}
\label{prop:metricest}
Assuming \eqref{ba:dvphi} and \eqref{ba:Yphi}, we have that for all $(u,v)\in D_{v_{\infty},r_0}$:
\begin{align}
\label{eq:durest1prop}
\left|- \Omega^{-2}\partial_u r-\frac{1}{2}\right|(u,v)\leq&\:C(\Delta_1,D_2,D_3,v_0) (M^{-1}v)^{-2p+1},\\
\label{eq:dvrestprop}
\left|r \partial_vr +(M-\frac{1}{2}r)\right|(u,v)\leq&\: C(\Delta_1,D_2,D_3,v_0) M (M^{-1}v)^{-2p+1},
\end{align}
where $C(\Delta_1,D_2,D_3,v_0)>0$ is a constant depending on $\Delta_1,D_2,D_3$ and $v_0$.
\end{proposition}
\begin{proof}
We can rewrite \eqref{conseq:r1} to obtain
\begin{equation*}
\partial_u\log(- \Omega^{-2}\partial_u r)=\frac{r}{-\partial_ur}(\partial_u\phi)^2=r(-\partial_ur)(Y\phi)^2.
\end{equation*}

By monotonicity of $\partial_ur$ (which follows directly from \eqref{conseq:r1}), we have that $\partial_ur<0$. We can therefore integrate in $u$ and change the integration variable to $r$ to obtain:
\begin{equation*}
\log(- \Omega^{-2}\partial_u r)(u,v)-\log(- \Omega^{-2}\partial_u r)(-\infty,v)=\int_{r(u,v)}^{r_{\mathcal{H}^+}(v)} r(Y\phi)^2\,dr'\leq C {\Delta_1}^2 (M^{-1}v)^{-2p},
\end{equation*}
where we moreover applied \eqref{ba:Yphi} and we used that $r$ is bounded, by \eqref{eq:estdatar} and monotonicity of $\partial_ur$.

By applying \eqref{eq:estdatadUr} and using that $\Omega^{-2}\partial_ur =\widehat{\Omega}^{-2} \partial_Ur$, we therefore obtain
\begin{equation*}
\begin{split}
\left|- \Omega^{-2}\partial_u r(u,v)-\frac{1}{2}\right|\leq &\:CD_2^2 (M^{-1}v)^{-2p+1}+CD_3 e^{-\frac{1}{4M}v}+(e^{C {\Delta_1}^2 D_2^2 (M^{-1}v)^{-2p}}-1).
\end{split}
\end{equation*}
The estimate \eqref{eq:durest1prop} follows by taking $C=C(\Delta_1,D_2,D_3,v_0)>0$ to be suitably large.

Subsequently, we integrate \eqref{propeq:r1} in $u$ to obtain
\begin{equation*}
\begin{split}
r\partial_vr(u,v)=&\:r\partial_vr(-\infty,v)+\int_{-\infty}^u\frac{-\frac{1}{4}\Omega^2}{\partial_u r} \partial_u r du'\\
=&\:r\partial_vr(-\infty,v)-\frac{1}{2}(r_{\mathcal{H}^+}(v)-r(u,v))-\int^{r_{\mathcal{H}^+}(v)}_r \left[-\frac{1}{4}{\Omega^2}({\partial_u r})^{-1}-\frac{1}{2}\right]\,dr'.
\end{split}
\end{equation*}
Hence,
\begin{equation*}
\left|r \partial_vr+M-\frac{1}{2}r\right|\leq r\partial_vr(-\infty,v)+ \frac{1}{2}(2M-r_{\mathcal{H}^+}(v))+\int^{r_{\mathcal{H}^+}(v)}_r \left|-\frac{1}{4}{\Omega^2}({\partial_u r})^{-1}-\frac{1}{2}\right|\,dr',
\end{equation*}
so we can apply Lemma \ref{lm:rdata} and \eqref{eq:durest1prop} to obtain \eqref{eq:dvrestprop}.
\end{proof}

In the following proposition, we derive additional estimates for the difference quantities $r-r_S$ and $\frac{\Omega^2}{\Omega^2_S}$ in the region $\mathcal{R}_{\delta}$. In contrast with Proposition \ref{prop:metricest}, we will additionally make use of the bootstrap assumption \eqref{ba:Omega}.

\begin{proposition}
\label{lm:metricredshift}
There exist a suitably large $C(\Delta_1,D_2,D_3,v_0)>0$, such that for any $\epsilon>0$ and $\delta=\delta(\epsilon)>0$ appropriately small:
\begin{align}
\label{eq:redshiftOmega}
\left|\frac{\Omega^2}{\Omega^2_S}-1\right|\leq&\: C\epsilon  \Delta_2 (M^{-1}v)^{-2p+1},\\
\label{eq:redshiftdur}
|\partial_u(r-r_S)| \leq&\: C\epsilon \Delta_2 e^{\frac{1}{4M}(v-|u|)}  (M^{-1}v)^{-2p+1}+C(\Delta_1,D_2,D_3,v_0)e^{\frac{1}{4M}(v-|u|)}(M^{-1}v)^{-2p+1},\\
\label{eq:redshiftr}
|r-r_S| \leq&\:  C\epsilon \Delta_2 Me^{\frac{1}{4M}(v-|u|)} (M^{-1}v)^{-2p+1}+C(\Delta_1,D_2,D_3,v_0) M (M^{-1}v)^{-2p+1},\\
\label{eq:redshiftdvr}
|\partial_v(r-r_S)| \leq&\:  C\epsilon \Delta_2 e^{\frac{1}{4M}(v-|u|)} (M^{-1}v)^{-2p+1}+C(\Delta_1,D_2,D_3,v_0)(M^{-1}v)^{-2p+1},
\end{align}
for all $(u,v)\in \mathcal{R}_{\delta}$.

We moreover have that
\begin{equation}
\label{eq:redshiftOmega2}
\left|\log \left(\frac{r^2\Omega^2}{2Mr-r^2}\right)\right|\leq C\epsilon  \Delta_2 Me^{\frac{1}{4M}(v-|u|)} (M^{-1}v)^{-2p+1}+C(\Delta_1,D_2,D_3,v_0)(M^{-1}v)^{-2p+1}\quad \textnormal{on $\gamma_{\delta}$}.
\end{equation}
\end{proposition}
\begin{proof}
We obtain \eqref{eq:redshiftOmega} by integrating \eqref{ba:Omega} in the $u$ direction, using that $\frac{\Omega^2}{\Omega^{2}_S}(-\infty,v)=1$ and
\begin{equation*}
e^{\frac{1}{4M}(v-|u|)}\leq e^{-\frac{1}{\delta}}
\end{equation*}
in $\mathcal{R}_{\delta}$.

By combining \eqref{eq:redshiftOmega} with \eqref{eq:durest1prop}, we obtain \eqref{eq:redshiftdur}. We then integrate \eqref{eq:redshiftdur} and apply \eqref{eq:estdatar} to obtain \eqref{eq:redshiftr}. The estimate \eqref{eq:redshiftdvr} follows by combining \eqref{eq:dvrestprop} with \eqref{eq:redshiftr}. Finally, \eqref{eq:redshiftOmega2} follows from a combination of \eqref{eq:redshiftOmega} with \eqref{eq:redshiftr}, using that $(r_S)_{\gamma_{\delta}}<2M$ and $(r_S)_{\gamma_{\delta}}$ is constant.
\end{proof}

\begin{proposition}
\label{prop:apparenthordecay}
For $|u_0|$ suitably large, $\mathcal{A}\subset \mathcal{R}_{\delta}$ and
\begin{align}
\label{eq:rA1}
|r_{\mathcal{A}}(v)-2M|\leq &\:C(\Delta_1,D_2,D_3,v_0)M (M^{-1}v)^{-2p+1},
\end{align}
\end{proposition}
\begin{proof}
We use that $\partial_vr=0$ along $\mathcal{A}$ and apply \eqref{eq:dvrestprop} to obtain \eqref{eq:rA1}. Together with \eqref{eq:redshiftr} and the fact that $(r_S)_{\gamma_{\delta}}<2M$ is constant, this implies that $\mathcal{A}\subset \mathcal{R}_{\delta}$ and taking $\delta>0$ suitably small.
\end{proof}
 \begin{figure}[H]
	\begin{center}
\includegraphics[scale=0.5]{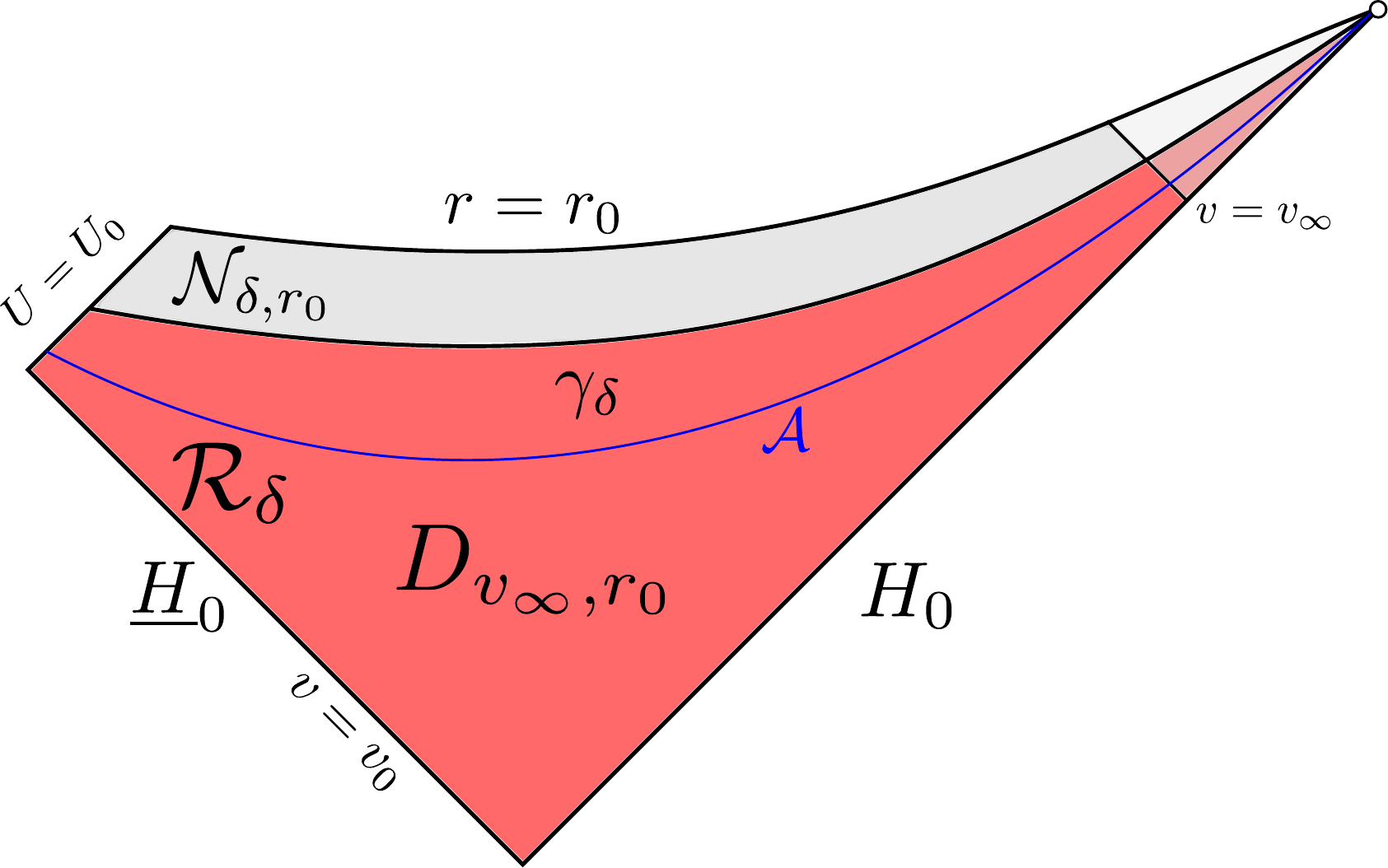}
\end{center}
\vspace{-0.2cm}
\caption{The location of the apparent horizon $\mathcal{A}$.}
	\label{fig:contour}
\end{figure}
\begin{proposition}
\label{prop:durdvrr0} 
We have that in $\mathcal{N}_{\delta, r_0}$
\begin{align}
\label{eq:rescaledOmegaestprop}
\left| \frac{r^2\Omega^2}{2Mr-r^2}-1\right|\leq&\: C_{\delta,r_0}(C(\Delta_1,D_2,D_3,v_0)+\epsilon \Delta_2)(M^{-1}v)^{-2p+2},\\
\label{eq:prelimestOmega}
|\partial_v\log (r^2\Omega^2)(u,v)-r^{-1}+Mr^{-2}|(u,v)\leq&\: C_{r_0}\cdot C(\Delta_1,D_2,D_3,v_0) (M^{-1}v)^{-2p+1},\\
\label{eq:durest2prop}
\left|- \partial_u r- \left(Mr^{-1}-\frac{1}{2}\right)\right|(u,v)\leq&\: C_{\delta,r_0}(C(\Delta_1,D_2,D_3,v_0)+\epsilon \Delta_2)(M^{-1}v)^{-2p+2},\\
\label{eq:ratiodurdvr}
\left|\frac{\partial_u r}{\partial_vr}-1\right|(u,v)\leq&\: C_{\delta,r_0} (C(\Delta_1,D_2,D_3,v_0)+\epsilon \Delta_2)(M^{-1} v)^{-2p+2}.
\end{align}
\end{proposition}
\begin{proof}
We integrate \eqref{propeq:Omegarescale2} to obtain
\begin{equation*}
\begin{split}
\partial_v\log (r^2\Omega^2)(u,v)=&\:\partial_v\log (r^2\Omega^2)(-\infty,v)+\int_{-\infty}^u[-2r^{-2}\partial_ur \partial_v r-2\partial_u\phi\partial_v\phi] (u',v)\,du'\\
=&\:\partial_v\log (r^2\Omega^2)(-\infty,v)+\int_{r_{\mathcal{H}^+}(v)}^{r(u,v)} [-2r^{-2} \partial_v r-2 Y \phi\partial_v\phi]|_{v'=v}(r')\,dr'\\
=&\: \partial_v\log (r^2)(-\infty,v)+ \partial_v\log(\widehat{\Omega}^2)(-\infty,v)-\int_{r_{\mathcal{H}^+}(v)}^{r(u,v)}[r'^{-2}-2Mr'^{-3} ]\,dr'\\
&\:+\int_{r_{\mathcal{H}^+}(v)}^{r(u,v)} [-2r^{-2} (\partial_v r-\frac{1}{2}+Mr^{-1})-2 Y \phi\partial_v\phi]|_{v'=v}(r')\,dr'.
\end{split}
\end{equation*}
where we used that $\partial_ur<0$ to reparametrize $v=const.$ lines with $r$ instead of $u$.

Note that we can estimate
\begin{equation*}
\left|\partial_v\log (r^2)(-\infty,v)+ \partial_v\log(\widehat{\Omega}^2)(-\infty,v)-\frac{1}{4M}\right|=2r^{-1}\partial_vr(-\infty,v) \leq CM^{-1} D_2^2 (M^{-1}v)^{-2p}.
\end{equation*}
Furthermore,
\begin{equation*}
-\int_{r_{\mathcal{H}^+}(v)}^{r(u,v)} [r'^{-2}-2Mr'^{-3}]\,dr'= r^{-1}-Mr^{-2}-r_{\mathcal{H}^+}^{-1}+Mr_{\mathcal{H}^+}^{-2},
\end{equation*}
so
\begin{equation*}
\left| \int_{r_{\mathcal{H}^+}(v)}^{r(u,v)} r'^{-2}-2Mr'^{-3}\,dr'+\left[r^{-1}-Mr^{-2}-\frac{1}{4M}\right] \right| \leq CM^{-1} D_2^2 (M^{-1}v)^{-2p+1}.
\end{equation*}
We also apply \eqref{ba:dvphi}, \eqref{ba:Yphi} and \eqref{eq:dvrestprop} to estimate:
\begin{equation*}
\int_{r_{\mathcal{H}^+}(v)}^{r(u,v)} [2r^{-2} (\partial_v r-\frac{1}{2}+Mr^{-1})-2 Y \phi\partial_v\phi]|_{v'=v}(r')\,dr'\leq C_{r_0}\cdot C(\Delta_1,D_2,D_3,v_0)(M^{-1}v)^{-2p+1}.
\end{equation*} 
We combine the above estimates to obtain:
\begin{equation*}
|\partial_v\log (r^2\Omega^2)(u,v)-r^{-1}+Mr^{-2}|(u,v)\leq C_{r_0}\cdot C(\Delta_1,D_2,D_3,v_0) (M^{-1}v)^{-2p+1},
\end{equation*}
which concludes the proof of \eqref{eq:prelimestOmega}.

Consider
\begin{equation*}
\begin{split}
\partial_v\log\left(\frac{r^2\Omega^2}{2Mr-r^2}\right)=&\:\partial_v\log (r^2\Omega^2)-\partial_v \log(2Mr-r^2)\\
=&\: \partial_v\log (r^2\Omega^2)-\frac{2M-2r}{2Mr-r^2}\partial_vr\\
=&\: \partial_v\log (r^2\Omega^2)-r^{-1}+Mr^{-2}- \frac{2M-2r}{2Mr-r^2}(\partial_vr-\frac{1}{2}+Mr^{-1}).
\end{split}
\end{equation*}
Note that
\begin{equation*}
(2M-r_S)^{-1}= (2M-r)^{-1}\frac{2M-r}{2M-r_S}= (2M-r)^{-1} (1+\frac{r_s-r}{2M-r_S}),
\end{equation*}
so we can estimate on $\gamma_{\delta}$:
\begin{equation*}
(2M-r)^{-1}(u,v)\leq \frac{1}{2M-r_{\delta}} \frac{1}{1+\frac{(r-r_s)(u,v)}{2M-r_\delta}}.
\end{equation*}
By \eqref{eq:redshiftr} together with $v-|u|\leq -\frac{1}{\delta}$ along $\delta$, we obtain
\begin{equation*}
(2M-r)^{-1}(u_{\delta}(v),v)\leq C_{\delta} M^{-1}[1+\epsilon \Delta_2 (M^{-1}v)^{-2p+1}]+ C(\Delta_1,D_2,D_3,v_0) M^{-1} (M^{-1}v)^{-2p+1}.
\end{equation*}
By $\partial_u r<0$ the above estimate holds in fact for all $(u,v)$ in $\mathcal{N}_{\delta,r_0}$.

We conclude by using \eqref{eq:prelimestOmega} and \eqref{eq:dvrestprop} that in $\mathcal{N}_{\delta,r_0}$:
\begin{equation}
\label{eq:dvrescaledOmegaest}
\left|\partial_v\log\left(\frac{\Omega^2}{2Mr^{-1}-1}\right)\right|\leq  C_{\delta,r_0}(\epsilon\Delta_2+C(\Delta_1,D_2,D_3,v_0)) M^{-1} (M^{-1}v)^{-2p+1}.
\end{equation}
We integrate in $v$ and  and apply \eqref{eq:redshiftOmega2} and \eqref{eq:dvrescaledOmegaest} to obtain
\begin{equation}
\label{eq:rescaledOmegaest}
\begin{split}
\left| \log \frac{r^2\Omega^2}{2Mr-r^2}\right|\leq&\:  \left|\log \left(\frac{r^2\Omega^2}{2Mr-r^2}\right)\right|(u,v_{\gamma_{\delta}}(u))\\
&+ C_{\delta,r_0}(\epsilon\Delta_2+C(\Delta_1,D_2,D_3,v_0)) M^{-1} (M^{-1}v_{\gamma_{\delta}}(u))^{-2p+2}\\
\leq&\: C_{\delta,r_0}(C(\Delta_1,D_2,D_3,v_0)+ \epsilon \Delta_2 ) (M^{-1}|u|)^{-2p+2}.
\end{split}
\end{equation}
This concludes the proof of \eqref{eq:rescaledOmegaestprop}.

By combining the above estimate with \eqref{eq:durest1prop}, we moreover obtain \eqref{eq:durest2prop}. We finally combine \eqref{eq:durest2prop} with \eqref{eq:dvrestprop} to obtain \eqref{eq:ratiodurdvr}.
\end{proof}

\begin{proposition}
\label{prop:closingbaOmega}
Let $\delta>0$ be suitably small. Then we can choose $\Delta_2>0$ suitably large so that
\begin{equation*}
\left| \partial_u \log \left(\frac{r^2\Omega^2}{r_S^2\Omega^2_S}\right)\right|\leq \frac{1}{2}\Delta_2 (M^{-1}v)^{-2p+1}\quad \textnormal{for all $(u,v)\in \mathcal{R}_{\delta}$},
\end{equation*}
thereby improving the bootstrap assumption \eqref{ba:Omega}.
\end{proposition}
\begin{proof}
We will apply \eqref{propeq:Omegarescale2}. We decompose:
\begin{align*}
-2r^{-2}\partial_ur\partial_vr=&-2r^{-2}[\partial_ur_S\partial_vr_S+\partial_u(r-r_S)\partial_vr_S+\partial_vr_S\partial_u(r-r_S)+\partial_u(r-r_S)\partial_v(r-r_S)],\\
r^{-1}=&\:r_S^{-1}\left(1-\frac{r_S-r}{r_S}\right)^{-1},\\
\partial_u\phi\partial_v\phi=&-\partial_ur Y\phi\partial_v\phi=-\partial_u r_S Y\phi\partial_v\phi- \partial_u(r-r_S) Y\phi\partial_v\phi.
\end{align*}
Note first that we can take $r_S>M$ in $\mathcal{R}_{\delta}$ for suitably small $\delta>0$. By the above expression for $r^{-1}$, together with \eqref{eq:redshiftr}, we can moreover conclude that for $v\geq v_1$, with $v_1$ suitably large depending on $D_i$, with $i=1,2,3$ and $\Delta_1$, we moreover have that $r>\frac{1}{2}M$ in $\mathcal{R}_{\delta}\cap \{v\geq v_1\}$. We can moreover take $|u_0|$ suitably large depending on $v_1$ to conclude that in $\mathcal{R}_{\delta}\leq \{v\leq v_1\}$:
\begin{equation*}
r\geq \frac{1}{2} \inf_{\underline{H}_0} r|_{\underline{H}_0}.
\end{equation*}
By \eqref{propeq:Omegarescale2}, together with the above observations, we can therefore estimate:

\begin{equation*}
\begin{split}
\left|\partial_v\partial_u\log\left(\frac{r^2\Omega^2}{r_S^2\Omega^2_S}\right)\right| \leq&\: C M^{-2}\Big[ |\partial_vr_s|( |\partial_u(r-r_S)|+|\partial_v(r-r_S)|+M^2|Y\phi||\partial_v\phi|)\\
&+|\partial_u(r-r_S)||\partial_v(r-r_S)|+M^2|\partial_u(r-r_S)||Y\phi||\partial_v\phi|\Big].
\end{split}
\end{equation*}

We use \eqref{eq:durschwest} together with the estimates in Proposition \ref{lm:metricredshift} and the bootstrap assumptions \eqref{ba:dvphi} and \eqref{ba:Yphi} to conclude that for $\epsilon\Delta_2<<1$:
\begin{equation*}
\left|\partial_v\partial_u\log\left(\frac{r^2\Omega^2}{r_S^2\Omega^2_S}\right)\right|\leq  C M^{-2} e^{\frac{1}{4M}(v-|u|)} [\epsilon^2 \Delta_2 (M^{-1}v)^{-2p+1}+C(\Delta_1,D_2,D_3,v_0)(M^{-1}v)^{-2p+1}],
\end{equation*}
We conclude the proof by integrating the above inequality in $v$, applying \eqref{eq:expint2} and choosing $\Delta_2$ suitably large compared to $C(\Delta_1,D_2,D_3,v_0)$.
\end{proof}

\subsection{Upper bound estimates for $\phi$}
\label{sec:upboundphi}
In this section, we derive upper bounds for $\phi$. We moreover improve the bootstrap assumptions \eqref{ba:dvphi} and \eqref{ba:Yphi}. 

In the following proposition we improve the bootstrap assumptions \eqref{ba:dvphi} and \eqref{ba:Yphi} the region $\mathcal{R}_{\delta}$:
\begin{proposition}
\label{prop:baimpphired}
In $\mathcal{R}_{\delta}$ with $\delta>0$ suitably small and $\Delta_1$ suitably large, we can estimate
\begin{align}
\label{eq:dvphiredshift0}
r|\partial_v\phi|(u,v)\leq &\: \frac{1}{4}{\Delta_1} (M^{-1}v)^{-p},\\
\label{eq:Yphiredshift0}
|Y\phi|(u,v)\leq &\: \frac{1}{4}M^{-1}{\Delta_1}  (M^{-1}v)^{-p}.
\end{align}
\end{proposition}
\begin{proof}
By integrating \eqref{propeq:phi1} in $u$ and applying \eqref{eq:durschwest}, \eqref{ba:Yphi},  \eqref{eq:dvrestprop} and \eqref{lm:metricredshift}, we moreover have that
\begin{equation}
\label{eq:dvphiredshiftest}
\begin{split}
|r\partial_v\phi|(u,v)\leq&\: |r\partial_v\phi|(-\infty,v)+\int_{-\infty}^u |\partial_vr||\partial_ur||Y\phi|(u',v)\,du'\\
\leq &\: C D_2(M^{-1}v)^{-p}+C(\Delta_1,D_2,D_3,v_0)(M^{-1}v)^{-p}e^{\frac{1}{4M}(v-|u|)}\\
\leq &\: C D_2(M^{-1}v)^{-p}+\epsilon C(\Delta_1,D_2,D_3,v_0)(M^{-1}v)^{-p}.
\end{split}
\end{equation}
Hence, for $\epsilon>0$ suitably small compared to $C(\Delta_1,D_2,D_3,v_0)$ and $\Delta_1$ suitably large compared to $D_2$, we obtain \eqref{eq:dvphiredshift0}.

We use \eqref{propeq:Yphi} and apply the estimates in Lemma \ref{lm:metricredshift} to obtain
\begin{equation*}
\partial_v(|r^2Y\phi|)\leq -\frac{1}{4M}(1-\epsilon) |r^2Y\phi|+r|\partial_v\phi|,
\end{equation*}
in $\mathcal{R}_{\delta}$, for $\epsilon>0$ arbitrarily small, with $\delta$ chosen appropriately small. We rearrange the above inequality to obtain:
\begin{equation*}
\partial_v(e^{\frac{1}{4M}(1-\epsilon)v}|r^2Y\phi|)\leq e^{\frac{1}{4M}(1-\epsilon)v}r|\partial_v\phi|.
\end{equation*}
By integrating the above equation in $v$ and applying \eqref{eq:dvphiredshiftest} and Lemma \ref{lm:expint} , we obtain
\begin{equation}
\label{eq:Yphiredshiftest}
\begin{split}
|r^2Y\phi|(u,v)\leq&\: e^{-\frac{1}{4M}(1-\epsilon)v} |r^2Y\phi|(u,v_0)+ e^{-\frac{1}{4M}(1-\epsilon)v} \int_{v_0}^v e^{\frac{1}{4M}(1-\epsilon)v'} r|\partial_v\phi|(u,v')\,dv'\\
\leq &\: e^{-\frac{1}{4M}(1-\epsilon)v} D_{3}+e^{-\frac{1}{4M}(1-\epsilon)v} \int_{v_0}^v e^{\frac{1}{4M}(1-\epsilon)v'} r|\partial_v\phi|(u,v')\,dv\\
\leq &\: CMe^{-\frac{1}{4M}(1-\epsilon)v} D_3+C MD_2(M^{-1}v)^{-p}+\epsilon C(\Delta_1,D_2,D_3,v_0)M(M^{-1}v)^{-p}.
\end{split}
\end{equation}
For $\epsilon>0$ suitably small depending on $r_0$ and $\Delta_1$ suitably large, depending on $D_2$ and $D_3$, we therefore obtain \eqref{eq:Yphiredshift0}.
\end{proof}

\begin{lemma}
\label{lm:relatuvconstr}
On every curve $\gamma_{r_1}=\{r=r_1\}$ in $\mathcal{N}_{\delta,r_0}$, we can estimate
\begin{equation}
\label{eq:reluvconstr}
\left|v_{\gamma_{r_1}}(u)-|u||\right|\leq   C_{\delta,r_0}(C(\Delta_1,D_2,D_3,v_0)+ \epsilon \Delta_2) (M^{-1}v)^{-2p+2}
\end{equation}
Furthermore, for all $(u,v)\in \mathcal{N}_{\delta,r_0}$ we have that
\begin{equation}
\label{eq:reluvN}
\left|1-\frac{v}{|u|}\right|\leq  C_{\delta,r_0}(C(\Delta_1,D_2,D_3,v_0)+ \epsilon \Delta_2) (M^{-1}v)^{-2p+1}.
\end{equation}
\end{lemma}
\begin{proof}
We have that $r(u,v_{\gamma_{r_1}})=r_1$ by definition, so
\begin{equation*}
0=\partial_ur(u,v_{\gamma_{r_1}})+\frac{dv_{\gamma_{r_1}}}{du} \partial_vr(u,v_{\gamma_{r_1}}),
\end{equation*}
which implies that 
\begin{equation*}
\left|\frac{dv_{\gamma_{r_1}}}{du} +1\right|\leq \left|1-\frac{\partial_ur}{\partial_v r}\right|(u,v_{\gamma_{r_1}})\leq  C_{\delta,r_0}(C(\Delta_1,D_2,D_3,v_0)+ \epsilon \Delta_2) (M^{-1}v)^{-2p+1},
\end{equation*}
where we applied \eqref{eq:ratiodurdvr}.

We obtain \eqref{eq:reluvconstr} by integrating the above equation in $u$.

To obtain \eqref{eq:reluvN}, we simply observe that for all $(u,v)\in \mathcal{N}_{\delta,r_0}$, there exists an $r_1\geq r_0$ such that $v_{\gamma_{r_1}}\leq v \leq v_{\gamma_{r_0}}$, so that
\begin{equation*}
v_{\gamma_{r_1}}(u)-|u|\leq v-|u|\leq v_{\gamma_{r_0}}(u)-|u|,
\end{equation*}
and we apply \eqref{eq:reluvconstr}.
\end{proof}

Now, we improve the bootstrap assumptions \eqref{ba:dvphi} and \eqref{ba:Yphi} the region $\mathcal{N}_{\delta,r_0}$:
\begin{proposition}
\label{prop:baimpphinoshift}
In $\mathcal{N}_{\delta,r_0}$ with $\delta>0$ suitably small and $|u_0|$ suitably large depending on $\delta$, $\Delta_1$, $\Delta_2$, $D_2$, $D_3$, $v_0$ and $r_0$, we can estimate
\begin{align}
\label{eq:dvphinoshift}
r^2|\partial_v\phi|(u,v)\leq &\: M  {\Delta_1}  (M^{-1}v)^{-p},\\
\label{eq:duphinoshift}
r^2|\partial_u\phi|(u,v)\leq &\: M  {\Delta_1}  (M^{-1}v)^{-p},\\
\label{eq:Yphirnoshift}
r|Y\phi|(u,v)\leq &\: M {\Delta_1}  (M^{-1}v)^{-p}.
\end{align}
\end{proposition}
\begin{proof}
For $\eta>0$ sufficiently small and $|u_0|$ appropriately large, the curve $\{r=r_{\delta}+\eta\}\cap \{v\geq v_{\gamma_{\delta}}(u_0)\}$ is contained in $\mathcal{R}_{\delta}$. Furthermore, we can arrange $v_{\gamma_{\delta}}(u_0)$ to be arbitrarily large, by taking $|u_0|$ appropriately large. In particular, we may use that $v_{\gamma_{\delta}}(u_0)>v_0$ and $\{r=r_{\delta}+\eta\}\cap \{v\geq v_{\gamma_{\delta}}(u_0)\}$ is spacelike (since it is contained in $J^+(\mathcal{A})$).

Denote:
\begin{equation*}
\overline{\Phi}(s):=\max \Big \{\sup_{\{r(u,v)= s\}\cap \{v\geq v_{\gamma_{\delta}}(u_0)\}} (M^{-1}|v|)^{p}||r\partial_v\phi| , \sup_{\{r(u,v)= s\}\cap\{v\geq v_{\delta}(u_0)\}}  (M^{-1}|u|)^{p}|r\partial_u\phi|\Big \},
\end{equation*}
with $r_0\leq s\leq r_{\delta}+\eta$.
 \begin{figure}[H]
	\begin{center}
\includegraphics[scale=0.5]{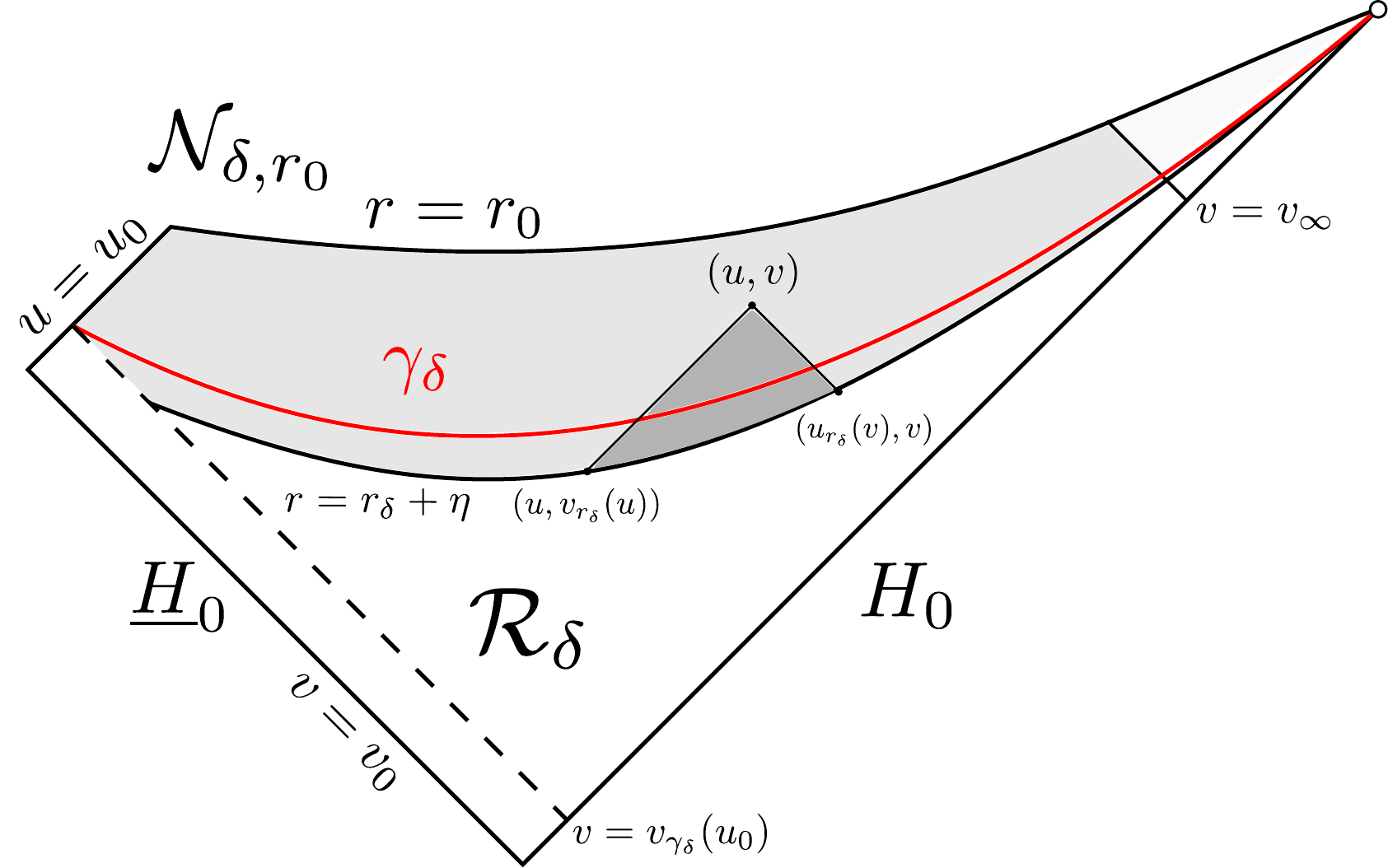}
\end{center}
\vspace{-0.2cm}
\caption{Relevant region for Gr\"onwall inequality.}
	\label{fig:gronwall1}
\end{figure}

We integrate \eqref{propeq:phi1} in $u$ and \eqref{propeq:phi2} in $v$ to obtain the estimate:
\begin{equation*}
\begin{split}
\overline{\Phi}(r(u,v))\leq &\: \overline{\Phi}(r_{\delta}+\eta) +\max\Bigg \{ \int_{u_{r_{\delta}+\eta}(v)}^u r^{-1}|\partial_vr| (M^{-1}v)^{p}|r\partial_u\phi|(u',v)\,du' ,\\
&\int_{v_{r_{\delta}+\eta}(u)}^v r^{-1}|\partial_ur|(M^{-1}|u|)^{p} |r\partial_v\phi|(u,v')\,dv'\Bigg\}\\
\leq &\: \overline{\Phi}(r_{\delta}+\eta)+\int^{r_{\delta}+\eta}_{r(u,v)} r'^{-1} \overline{\Phi}(r')\,dr'+\Bigg|\int^{r_{\delta}+\eta}_{r(u,v)} r'^{-1}\left[\frac{v^p}{|u'|^p}\frac{|\partial_vr|}{|\partial_ur|}-1\right] (M^{-1}|u'|)^p|r\partial_u\phi|\Big|_{v'=v}(r')\,dr'\Bigg |\\
&+\Bigg|\int^{r_{\delta}+\eta}_{r(u,v)} r'^{-1}\left[\frac{|u|^p}{{v'}^p}\frac{|\partial_ur|}{|\partial_vr|}-1\right] (M^{-1}v')^p|r\partial_v\phi|\Big|_{u'=u}(r')\,dr' \Bigg|.
\end{split}
\end{equation*}

We further estimate the last two integrals by applying \eqref{eq:ratiodurdvr}, \eqref{eq:reluvN}, \eqref{ba:dvphi} and \eqref{ba:Yphi} and then we apply a standard Gr\"onwall inequality to obtain:
\begin{equation*}
\begin{split}
\overline{\Phi}(r(u,v))\leq&\: \left(\overline{\Phi}(r_{\delta}+\eta)+C_{\delta,r_0} \cdot C(\Delta_1,D_2,D_3,v_0)(M^{-1}v)^{-2p+2}\right)\frac{r_{\delta}+\eta}{r}.
\end{split}
\end{equation*}

We moreover have in $\mathcal{N}_{\delta,r_0}$ that
\begin{equation*}
v\geq |u|-\frac{1}{\delta}>\frac{1}{2}|u_0|
\end{equation*} 
for $|u_0|$ suitably large depending on $\delta$. Hence, by taking $|u_0|$ to be suitably large depending on $C(\Delta_1,D_2,D_3,v_0)$, $\Delta_2$ and applying and Propositions \ref{prop:baimpphired} and \ref{prop:durdvrr0} to estimate $\overline{\Phi}(r_{\delta}+\eta)$, we obtain:
\begin{align*}
(M^{-1}v)^{p}|r^2\partial_v\phi|(u,v)\leq&\:  \frac{1}{2} M\Delta_1+C_{\delta,r_0} M\cdot (M^{-1}v)^{-2p+2},
\\
(M^{-1}v)^{p}|r^2\partial_u\phi|(u,v)\leq&\:  \frac{1}{2} M\Delta_1+C_{\delta,r_0} M\cdot (M^{-1}v)^{-2p+2}.
\end{align*}
We need to additionally take $|u_0|$ suitably large compared to $C_{\delta,r_0}$ to establish \eqref{eq:dvphinoshift} and \eqref{eq:duphinoshift}. Finally, we apply \eqref{eq:durest2prop} to obtain also \eqref{eq:Yphirnoshift}.
\end{proof}

\subsection{A continuity argument}\label{continuity argument}  
\begin{proposition}
\label{prop:continuityarg}
For $|U_0|$ suitably small, there exists a smooth solution $(r,\widehat{\Omega}^2, \phi)$ to the system \eqref{propeq:r1}--\eqref{propeq:Omega},  \eqref{conseq:r1}--\eqref{propeq:phi2} in $(U,v)$ coordinates (with $\widehat{\Omega}^2$ replacing ${\Omega}^2$) in the region
\begin{equation*}
D_{\infty,r_0}:= [0,U_0)\times [v_0, \infty)\cap\{r\geq r_0\},
\end{equation*}
arising from initial data as prescribed in Section \ref{sec:initialdata}, satisfying moreover the estimates \eqref{ba:dvphi}--\eqref{ba:Omega} and all the estimates in Sections \ref{sec:estromega} and \ref{sec:upboundphi} with $v_{\infty}=\infty$, where $u$ and $\Omega^2$ are related to $U$ and $\widehat{\Omega}^2$ according to \eqref{eq:reluU} and \eqref{eq:relhatomega}, respectively.
\end{proposition}
\begin{proof}
We introduce the following conditions for $v_0\leq v_{\infty}\leq \infty$:
\begin{enumerate}[(A)]
\item A smooth solution $(r,\widehat{\Omega}^2, \phi)$ exists in $D_{v_{\infty},r_0}$ with the prescribed initial data.
\item The estimates \eqref{ba:dvphi}--\eqref{ba:Omega} hold in $D_{v_{\infty},r_0}$.
\end{enumerate}
Now consider the set $\mathcal{V}\subseteq [v_0,\infty)$, defined as
\begin{equation*}
\mathcal{V}=\{V\in [v_0,\infty)\,|\, \textnormal{ (A) and (B) hold for all $v_0\leq v_{\infty}< V$}\}.
\end{equation*}
By standard local existence theory, $\mathcal{V}$ is non-empty. Closedness of $\mathcal{V}$ follows immediately from the definition of $\mathcal{V}$. In order to conclude that $\mathcal{V}=[v_0,\infty)$, it remains to establish openness of $\mathcal{V}$.

Suppose $V\in \mathcal{V}$. 

By the boundedness properties established in $D_{V,r_0}$ following from Sections \ref{sec:estromega} and \ref{sec:upboundphi}, together with a standard propagation of regularity argument, the condition (A) holds in
\begin{equation*}
[0,U_0)\times [v_0, V]\cap\{r\geq r_0\}.
\end{equation*}
Applying local existence theory again, with the restrictions of $(r,\widehat{\Omega}^2, \phi)$ on $v=V$ (and $H_0$) taken as initial data, it follows that there exists $\delta>0$ such that (A) holds in $D_{V+\delta,r_0}$.

Finally, since the estimates in Sections \ref{sec:estromega} and \ref{sec:upboundphi} improve the bootstrap estimates \eqref{ba:dvphi}--\eqref{ba:Omega} in $D_{V,r_0}$, it follows by continuity that (B) also holds in $D_{V+\delta,r_0}$, provided $\delta>0$ is taken to be suitably small. Hence, $(v-\delta,v+\delta)\subset \mathcal{V}$ and openness follows.
\end{proof}
\subsection{Additional lower bound estimates for $\phi$}
In this section, we derive additional lower bound estimates for $\partial_u\phi$ and $\partial_v\phi$ on $\{r\geq r_0\}$. We restrict to the region $D_{\infty,r_0}$.
\begin{proposition}
\label{prop:lowboundsposr0}
Let $\epsilon>0$ and $2M<r_1\leq r_0$. Then there exists $v_1\geq v_0$ suitably large, depending on $r_0$, $\epsilon$, $D_1$ and $\Delta_1$, so that for all $v\geq v_1$ and suitably large $|u_0|$:
\begin{align}
\label{eq:lowbounddvphi0}
r\partial_v\phi(u,v)\geq&\: (D_1-\epsilon) (M^{-1}v)^{-q},\\
\label{eq:lowboundYphi0}
r^2 Y\phi(u,v)\geq&\: \left(2Mr_1^{-2}-\frac{1}{2}r_1^{-1}\right)^{-1} (D_1-\epsilon)(M^{-1}v)^{-q},
\end{align}
for all $(u,v)$ in $\{r\geq r_1\}$, where $2M<r_1\leq r_0$.

Furthermore, along $\{r=r_1\}$, we obtain for $v\geq v_1$:
\begin{equation}
\label{eq:lowboundduphi}
r\partial_u\phi|_{r=r_1}\geq \frac{M-\frac{1}{2}r_1}{2M-\frac{1}{2}r_1}(D_1-\epsilon)(M^{-1}v)^{-q}.
\end{equation}
\end{proposition}
\begin{proof}
By \eqref{eq:idlowboundYphi} together with continuity of $Y\phi(U,v_1)$ in $U$ at $U=0$, it follows that for suitably large $|u_0|$, we can estimate:
\begin{equation}
\label{eq:Yphiv1}
r^2Y\phi(u,v_1)\geq 4M(1-2\epsilon)D_1(M^{-1}v_1)^{-q}.
\end{equation}
By \eqref{propeq:Yphi} it then follows that for $v\geq v_1$ we have that $Y\phi\geq 0$.

By integrating \eqref{propeq:phi1} and using that $Y\phi(u,v)\geq 0$ for $v\geq v_1$, we obtain for all $v\geq v_1$:
\begin{equation*}
\begin{split}
r\partial_v\phi(u,v)=&\:r\partial_v\phi(-\infty,v)+ \int_{-\infty}^u [\partial_vr \partial_ur Y\phi](u',v)\,du'\\
=&\:r\partial_v\phi(-\infty,v)-\int_{r(u,v)}^{r_{\mathcal{A}}(v)} \partial_v r Y\phi\Big|_{v'=v}(r')\,dr' -\int^{r_{\mathcal{H}^+}(v)}_{r_{\mathcal{A}}(v)} \partial_v r Y\phi\Big|_{v'=v}(r')\,dr'\\
\geq &\: r\partial_v\phi(-\infty,v) -\int^{r_{\mathcal{H}^+}(v)}_{r_{\mathcal{A}}(v)} \partial_v r Y\phi\Big|_{v'=v}(r')\,dr'\\
\geq &\: r\partial_v\phi(-\infty,v)- (|r_{\mathcal{H}^+}(v)-2M|+ |r_{\mathcal{A}}(v)-2M|) \sup_{-\infty\leq u'\leq u_{\mathcal{A}}(v)}|\partial_v r||Y\phi|(u',v)\\
\geq &\: D_1 (M^{-1}v)^{-q}- C(\Delta_1,D_2,D_3,v_0) (M^{-1}v)^{-5p+2},
\end{split}
\end{equation*}
where we used \eqref{eq:estdatar}, \eqref{eq:dvrestprop} \eqref{ba:Yphi} and moreover, we used that by \eqref{eq:dvrestprop}:
\begin{equation*}
|r_{\mathcal{A}}-2M|\leq  C(\Delta_1,D_2,D_3,v_0) M (M^{-1}v)^{-2p+1}.
\end{equation*}

Hence, using that $q< 5p-2$, it follows that we can choose $v_1$ suitably large, depending on $D_1$, $\Delta_1$ and $\epsilon$, so that
\begin{equation*}
r\partial_v\phi(u,v)\geq (D_1-\epsilon) (M^{-1}v)^{-q}
\end{equation*}
for all $v\geq v_1$.

Furthermore, for an arbitrary small constant $\epsilon_0>0$ and
\begin{equation*}
\delta_1:= (D_1-\epsilon)(2Mr_1^{-2}-\frac{1}{2}r_1^{-1}+\epsilon_0M^{-1})^{-1},
\end{equation*}
we can write
\begin{equation}
\label{eq:auxestdvphi}
r\partial_v\phi(u,v)\geq \delta_1 (2Mr_1^{-2}-\frac{1}{2}r_1^{-1}+\epsilon_0M^{-1}) (M^{-1}v)^{-q}.
\end{equation}
Note that $\delta_1$ is well-defined and positive since $r_1<2M$. 

We will now apply \eqref{propeq:Yphi}. Observe first that by \eqref{eq:durest1prop} and \eqref{eq:dvrestprop} and $r\geq r_0$, $v\geq v_1$:
\begin{equation*}
\begin{split}
\left|\frac{1}{4}\frac{\Omega^2}{r\partial_u r}+2\frac{\partial_vr}{r}+\frac{2M}{r^2}-\frac{1}{2}r^{-1}\right|\leq&\: C_{r_0}(\Delta_1,D_2,D_3,v_0)M^{-1} (M^{-1}v)^{-2p+1}\\
\leq&\: C_{r_0}(\Delta_1,D_2,D_3,v_0)M^{-1}  (M^{-1}v_1)^{-2p+1}=:\epsilon_0M^{-1}.
\end{split}
\end{equation*}
By definition, $\epsilon_0>0$ is arbitrarily small for suitably large $v_1$.

By \eqref{propeq:Yphi}  and \eqref{eq:Yphiv1} it then follows that for $v\geq v_1$:
\begin{equation*}
\partial_v(r^2 Y\phi)\geq -\left[2Mr_1^{-2}-\frac{1}{2}r_1^{-1}+\epsilon_0M^{-1}\right]r^2Y\phi+r\partial_v\phi.
\end{equation*}
By rearranging the equation, we conclude that
\begin{equation*}
\partial_v(e^{(2Mr_1^{-2}-\frac{1}{2}r_1^{-1}+\epsilon_0M^{-1})v} r^2 Y\phi)\geq e^{(2Mr_1^{-2}-\frac{1}{2}r_1^{-1}+\epsilon_0M^{-1})v} r\partial_v\phi
\end{equation*}
and hence, for $v\geq v_1$, we have that
\begin{equation*}
\begin{split}
e^{(2Mr_1^{-2}-\frac{1}{2}r_1^{-1}+\epsilon_0M^{-1})v} r^2 Y\phi(u,v)\geq&\: e^{(2Mr_1^{-2}-\frac{1}{2}r_1^{-1}+\epsilon_0M^{-1})v_1} r^2 Y\phi(u,v_1)\\
&+\int_{v_1}^v e^{(2Mr_1^{-2}-\frac{1}{2}r_1^{-1}+\epsilon_0M^{-1})v'} r \partial_v\phi(u',v)\,dv'.
\end{split}
\end{equation*}

We apply \eqref{eq:auxestdvphi} and \eqref{eq:Yphiv1} to conclude that
\begin{equation*}
\begin{split}
r^2 Y\phi(u,v)\geq&\:  e^{(2Mr_1^{-2}-\frac{1}{2}r_1^{-1}+\epsilon_0M^{-1})(v_1-v)} r^2 Y\phi(u,v_1) \\
&+(2Mr_1^{-2}-\frac{1}{2}r_1^{-1}+\epsilon_0M^{-1})e^{-(2Mr_1^{-2}-\frac{1}{2}r_1^{-1}+\epsilon_0M^{-1})v}\cdot \\
&\cdot \int_{v_1}^v e^{(2Mr_1^{-2}-\frac{1}{2}r_1^{-1}+\epsilon_0M^{-1})v'} \delta_1 (M^{-1}v')^{-q}\,dv'\\
\geq&\: e^{(2Mr_1^{-2}-\frac{1}{2}r_1^{-1}+\epsilon_0M^{-1})(v_1-v)} r^2 Y\phi(u,v_1)\\
&+e^{-(2Mr_1^{-2}-\frac{1}{2}r_1^{-1}+\epsilon_0M^{-1})v} \int_{v_1}^v \frac{d}{dv}\left(e^{(2Mr_1^{-2}-\frac{1}{2}r_1^{-1}+\epsilon_0M^{-1})v'}  \delta_1 (M^{-1}v')^{-q}\right)\,dv'\\
\geq &\:e^{(2Mr_1^{-2}-\frac{1}{2}r_1^{-1}+\epsilon_0M^{-1})(v_1-v)} [4M(1-2\epsilon) D_1- \delta_1](M^{-1}v_1)^{-q}+   \delta_1 (M^{-1}v)^{-q}\\
\geq &\: \delta_1 (M^{-1}v)^{-q}
\end{split}
\end{equation*}
if $\delta_1<4M(1-2\epsilon) D_1$, which can be arranged provided $r_1<2M$. We conclude that
\begin{equation*}
r^2 Y\phi(u,v)\geq (D_1-\epsilon)(2Mr_1^{-2}-\frac{1}{2}r_1^{-1}+\epsilon_0M^{-1})^{-1}.
\end{equation*}
In order to obtain \eqref{eq:lowboundYphi0}, we simply replace $\epsilon>0$ in the above argument with $\frac{\epsilon}{2}$ and absorb the $\epsilon_0$ term, using that $\epsilon_0$ can be taken suitably small for $v_1$ suitably large.

By \eqref{eq:durest2prop} and $v_1$ suitably large compared to $r_0^{-1}$, we obtain as a corollary:
\begin{equation*}
r\partial_u\phi|_{r=r_1}=(-\partial_ur) rY\phi|_{r=r_1}\geq \frac{M-\frac{1}{2}r_1}{2M-\frac{1}{2}r_1+r_1^2\epsilon_0 M^{-1}}(D_1-\epsilon)(M^{-1}v)^{-q}.
\end{equation*}
Again, we can replace $\epsilon>0$ in the above argument with $\frac{\epsilon}{2}$ and choose $v_1$ so that $\epsilon_0$ is appropriately small to obtain \eqref{eq:lowboundduphi}.
\end{proof}

\begin{proposition}
\label{prop:lowboundsposr}
Let $\epsilon>0$. Then there exists $v_1\geq v_0$, $|u_1|>|u_0|$, with $v_1,|u_0|, |u_1|$ suitably large, depending on $r_0$, $\epsilon$, $D_1$ and $\Delta_1$, so that for all $v\geq v_1$, $|u|\geq |u_1|$, and $r_0\leq r\leq (4-2\sqrt{2})M$:
\begin{align}
\label{eq:lowbounddvphi}
r^2\partial_v\phi(u,v)\geq&\: (6-4\sqrt{2})M(D_1-\epsilon) (M^{-1}v)^{-q},\\
\label{eq:lowboundduphi}
r^2\partial_u\phi(u,v)\geq&\: (6-4\sqrt{2})M(D_1-\epsilon)(M^{-1}v)^{-q}.
\end{align}
\end{proposition}
\begin{proof}
Denote:
\begin{equation*}
\underline{\Phi}(s):=\min \Big \{\inf_{\{r(u,v)= s\}\cap\{|u|\geq |u_1|,v\geq v_1\}} (M^{-1}|v|)^{q} r\partial_v\phi , \inf_{\{r(u,v)= s\}\cap\{|u|\geq |u_1|,v\geq v_1\}}  (M^{-1}|u|)^{q} r\partial_u\phi \Big \},
\end{equation*}
with $r_0\leq s\leq r_1$, where $r_1<2M$ will be determined later and $|u_1|>|u_0|$, such that $v_{r_1}(u)>v_1$ for all $|u|\geq |u_1|$.
 \begin{figure}[H] 
	\begin{center}
\includegraphics[scale=0.5]{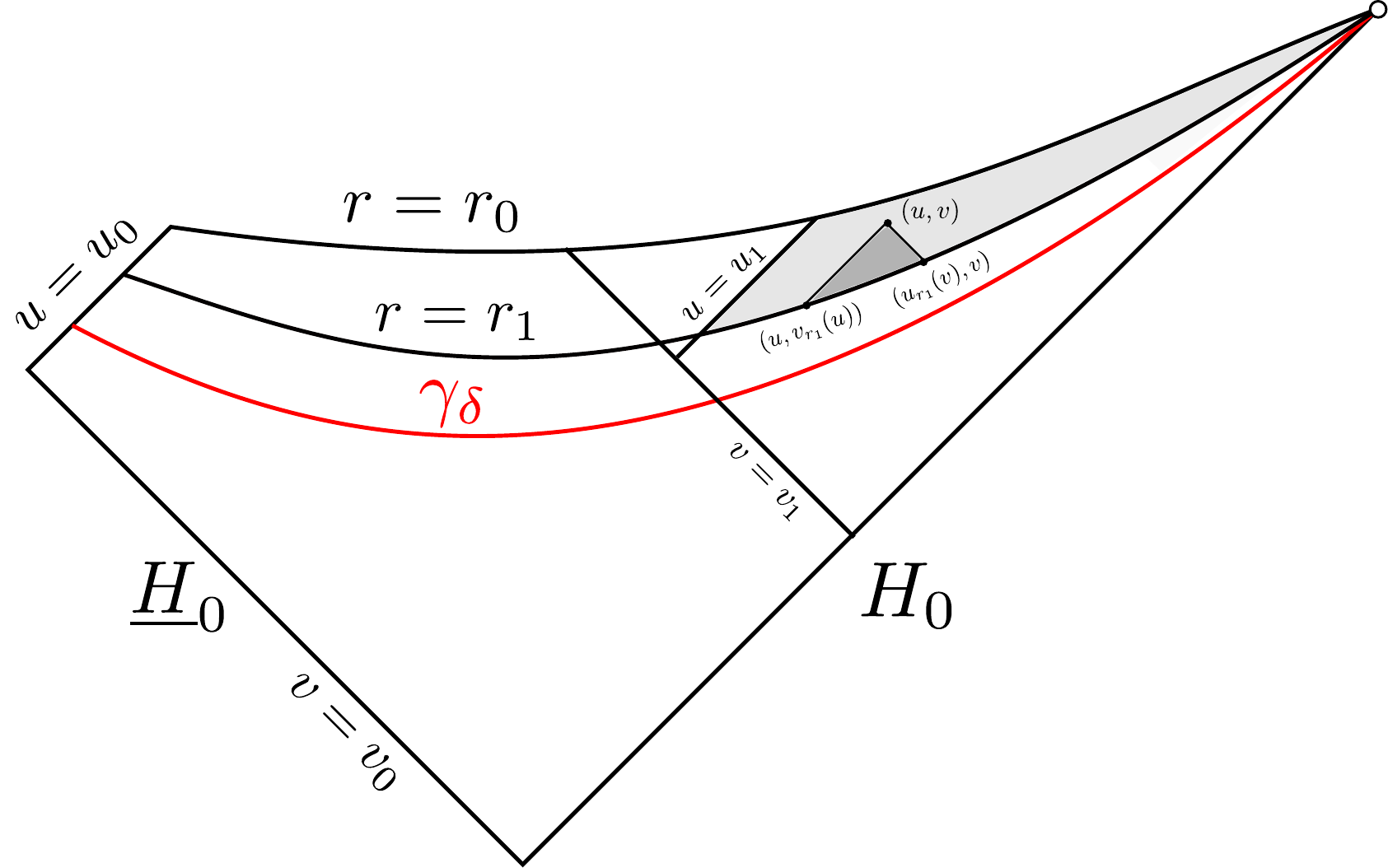}
\end{center}
\vspace{-0.2cm}
\caption{Relevant region for reverse Gr\"onwall inequality.}
	\label{fig:gronwall2}
\end{figure}
We integrate \eqref{propeq:phi1} in $u$ and \eqref{propeq:phi2} in $v$ and use that $\partial_v\phi>0$ and $\partial_u\phi>0$ in $\{r\geq r_0\}\cap\{v\geq v_1\}$ by Proposition \ref{prop:lowboundsposr0} to obtain
\begin{equation*}
\begin{split}
\underline{\Phi}(r(u,v))\geq &\: \underline{\Phi}(r_1) +\min \Bigg \{ \int_{u_{r_{1}}(v)}^u r^{-1}(-\partial_vr) (M^{-1}v)^{q}r\partial_u\phi(u',v)\,du' ,\\
&\int_{v_{r_{1}}(u) }^v r^{-1}(-\partial_ur)(M^{-1}|u|)^{q} r\partial_v\phi(u,v')\,dv'\Bigg\}\\
\geq &\:  \underline{\Phi}(r_1)+\int^{r_1}_{r(u,v)} r'^{-1} \underline{\Phi}(r')\,dr'-\Bigg|\int^{r_{1}}_{r(u,v)} r'^{-1}\left[\frac{v^q}{|u'|^q}\frac{|\partial_vr|}{|\partial_ur|}-1\right] (M^{-1}|u'|)^q|r\partial_u\phi|\Big|_{v'=v}(r')\,dr' \Bigg |\\
&-\Bigg | \int^{r_{1}}_{r(u,v)} r'^{-1}\left[\frac{|u|^q}{{v'}^q}\frac{|\partial_ur|}{|\partial_vr|}-1\right] (M^{-1}v')^q|r\partial_v\phi|\Big|_{u'=u}(r')\,dr'\Bigg|.
\end{split}
\end{equation*}
We further estimate the last two integrals by applying \eqref{eq:ratiodurdvr}, \eqref{eq:reluvN}, \eqref{ba:dvphi} and \eqref{ba:Yphi}, and then we apply Proposition \ref{prop:revgron} to obtain:
\begin{equation*}
\begin{split}
\underline{\Phi}(r(u,v))\geq \left(\underline{\Phi}(r_1)-C_{\delta,r_0} \cdot C(\Delta_1,D_2,D_3,v_0)(M^{-1}v)^{-2q+2}\right)\frac{r_1}{r}.
\end{split}
\end{equation*}
Together with Proposition \ref{prop:lowboundsposr0}, if $|u_0|$ is suitably large, we obtain
\begin{align*}
r^2\partial_v\phi(u,v)\geq&\: \frac{Mr_1-\frac{1}{2}r_1^2}{2M-\frac{1}{2}r_1}(D_1-\epsilon)(M^{-1}v)^{-q},\\
r^2\partial_u\phi(u,v)\geq&\: \frac{Mr_1-\frac{1}{2}r_1^2}{2M-\frac{1}{2}r_1}(D_1-\epsilon)(M^{-1}v)^{-q}.
\end{align*}
The $r_1$-dependent coefficient on the right-hand sides is maximized for $\frac{r_1}{M}=4-2\sqrt{2}$. With this choice, we arrive at \eqref{eq:lowbounddvphi} and \eqref{eq:lowboundduphi}.
\end{proof}

\section{Estimates near spacelike singularity}
\label{sec:estnearsing}
The continuity argument in Section \ref{continuity argument} works for any small $r_0>0$. And the proof is uniformly independent of $r_0$. In particular, with initial data prescribed along $H_0$ and $\underline{H}_0$ satisfying \eqref{eq:estdataphi} and \eqref{eq:estdataphi2}, we hence deduce that there exists a smooth solution $(r,\widehat{\Omega}^2, \phi)$ to the system \eqref{propeq:r1}--\eqref{propeq:Omega},  \eqref{conseq:r1}--\eqref{propeq:phi2} in $(U,v)$ coordinates (with $\widehat{\Omega}^2$ replacing ${\Omega}^2$) in the region $D_{\infty, 0}:= [0,U_0)\times [v_0, \infty)\cap\{r>0\}$ where $|U_0|$ is small. In this section, we will derive more precise bounds for $r\partial_u r, r\partial_v r$, $r^2\partial_u \phi, r^2\partial_v \phi$, $\partial_u \log(r\O^2), \partial_v \log(r\O^2), r\O^2$ in $(u,v)$ coordinates. And with these bounds, we will prove the main conclusions of this paper.

{\color{black} Let $r_0>0$ be a fixed radius, as defined in Section \ref{Section7}. We will assume in this section that $\frac{r_0}{M}$ is suitably small and consider the region}
$$|u|\geq |u_1|+2v_{r=0}(u_1), \quad v\geq \max\{v_1, \t v_1\} \quad \mbox{and} \quad r(u,v)\leq r_0.$$ 
Here $u_1, v_1$ are chosen as in Proposition \ref{prop:lowboundsposr}. And $\t v_1$ is required to satisfy another property: for $v\geq \t v_1$ along $r=r_0$, it holds
\be\label{additional v}
|\partial_v \log(r\O^2)(u_{r_0(v)},v)|\leq\f{1}{r_0}.
\ee

Note that with {\color{black}the} estimates in Proposition \ref{prop:durdvrr0} and Proposition \ref{prop:metricest}, {\color{black} we have that \eqref{additional v} is satisfied along $r=r_0$ for suitably large $\t v_1$.} {\color{black}Let us} first recall Proposition \ref{prop:durdvrr0}:
$$|\partial_v\log (r^2\Omega^2)-r^{-1}+Mr^{-2}|(u_{r_0(v)},v)\leq C_{r_0}\cdot C(\Delta_1,D_2,D_3,v_0) (M^{-1}v)^{-2p}.$$
which is equivalent to
\begin{equation}\label{log r O middle}
|\partial_v \log(r\O^2)+\f{1}{r^2}(r\partial_v r+M-\f{r}{2})-\f{1}{2r}|(u_{r_0(v)},v)\leq C_{r_0}\cdot C(\Delta_1,D_2,D_3,v_0) (M^{-1}v)^{-2p}.
\end{equation}
By {\color{black} applying} triangle inequalities{\color{black}, it follows that}
\bes
\begin{split}
&|\partial_v \log(r\O^2)(u_{r_0(v)},v)|\\
\leq& |\partial_v \log(r\O^2)+\f{1}{r^2}(r\partial_v r+M-\f{r}{2})-\f{1}{2r}|(u_{r_0(v)},v)+\f{1}{r_0^2}|r \partial_vr-M+\frac{1}{2}r|(u_{r_0(v)},v)+\f{1}{2r_0}.
\end{split}
\ees
Together with Proposition \ref{prop:metricest}{\color{black}, we therefore have that}
$$|r \partial_vr-M+\frac{1}{2}r|(u_{r_0(v)},v)\leq C(\Delta_1,D_2,D_3,v_0) M (M^{-1}v)^{-2p+1}.$$
{\color{black} For $v\geq \t v_1$ with $\t v_1$ large, we conclude that}
$$|\partial_v \log(r\O^2)(u_{r_0(v)},v)|\leq\f{1}{r_0}.$$
\begin{remark}\label{v large}
In the {\color{black} sections below}, when requiring $v$ sufficiently large, we mean $v\geq \max\{v_1, \t v_1\}$ with $v_1$ chosen as in Proposition \ref{prop:lowboundsposr} and $\t v_1$ chosen as above. And we consider the region 
$$|u|\geq |u_1|+2v_{r=0}(u_1), \quad v \mbox{ sufficiently large, } \quad \mbox{and} \quad r(u,v)\leq r_0.$$
\end{remark}

\begin{figure}[H]\label{P8.1}
\begin{center}
\begin{minipage}[!t]{0.4\textwidth}
\begin{tikzpicture}[scale=0.95]
\draw [white](-1, 0)-- node[midway, sloped, above,black]{$\big(u_1, v_{r=0}(u_1)\big)$}(1, 0);
\draw [white](1.5, 0)-- node[midway, sloped, above,black]{$\mathcal{S}$}(1.6, 0);
\draw [white](5.5, 0.2)-- node[midway, sloped, above,black]{$i^+$}(7, 0.2);
\draw (0,0) to [out=-5, in=195] (5.5, 0.5);
\draw [white](0, -4.5)-- node[midway, sloped, above,black]{$v=v_1$}(-4, -4.5);
\draw [white](0, -3)-- node[midway, sloped, above,black]{$\mathcal{H}$}(7, -3);
\draw [white](0, -2.8)-- node[midway, sloped, above, black]{$r=r_0$}(0.5,-2.8);
\draw [white](-2.5, -2.5)-- node[midway, sloped, above, black]{$u=u_1$}(0,0);
\draw [white](2, -0.79)-- node[midway, sloped, above, black]{$(u,v)$}(2.5, -0.79);
\draw [white](2, 0)-- node[midway, sloped, above,black]{$\t q$}(2.2, 0);
\draw[fill] (2.1,-0.05) circle [radius=0.08];
\draw [thick] (-2.5,-2.5)--(0,-5);
\draw [dashed](-2.5, -2.5)--(0,0);
\draw [thick] (5.5, 0.5)--(0,-5);
\draw (-2,-2) to [out=-5, in=215] (5.5, 0.5);
\draw[fill] (0,0) circle [radius=0.08];
\draw[fill] (5.5, 0.5) circle [radius=0.08];
\draw[fill] (-2,-2) circle [radius=0.08];
\draw[fill] (2.65, -1.23) circle [radius=0.08];
\draw[fill] (2.21,-0.79) circle [radius=0.08];
\draw[fill](1.31, -1.69) circle [radius=0.08];
\draw [dashed](1.31, -1.69)--(2.21,-0.79);
\draw [dashed](2.65, -1.23)--(2.21,-0.79);
\end{tikzpicture}
\end{minipage}
\begin{minipage}[!t]{0.5\textwidth}
\end{minipage}
\hspace{0.05\textwidth}
\end{center}
\caption{The domain under consideration near $r=0$.}
\end{figure}
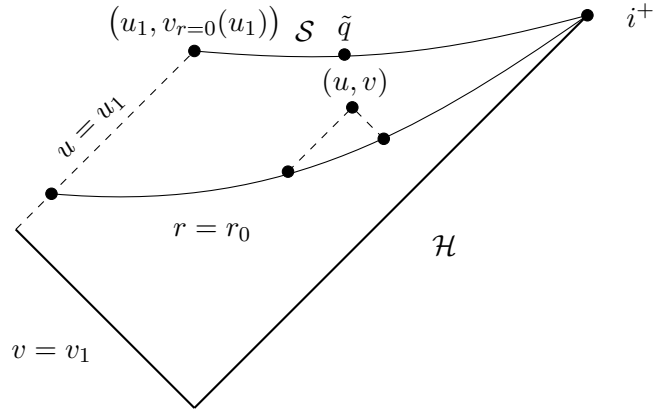

\subsection{Preliminary Estimates for $r\partial_u r$ and $r\partial_v r$}

\noindent In this section, we use {\color{black} a} monotonicity argument to extend the estimates for $r\partial_u r$ and $r\partial_v r$ from Section \ref{Section7} to the region $\{r\leq r_0\}$. 
\begin{proposition}\label{1st partial r}

\noindent For $(u,v)$ with $|u|\geq |u_1|+2v_{r=0}(u_1)$ and $r(u,v)\leq r_0$, it holds {\color{black}that}
\be\begin{split}\label{rvr1}
|(r\partial_v r+M-\f12 r)(u,v)|\leq& C(\Delta_1,D_2,D_3,v_0) (M^{-1}v)^{-2p+1}\cdot M\\
&+[1+C(\Delta_1,D_2,D_3,v_0) (M^{-1}v)^{-2p+1}]\cdot [r_0-r(u,v)]\\
\leq& \epsilon M,
\end{split}\ee

\be\begin{split}\label{rur1}
|(r\partial_u r+M-\f12 r)(u,v)|\leq& C_{\delta,r_0}(C(\Delta_1,D_2,D_3,v_0)+\epsilon \Delta_2)(M^{-1} |u|)^{-2p+1}\cdot M\\
&+[1+C(\Delta_1,D_2,D_3,v_0) (M^{-1}v)^{-2p+1}]\\
&\quad\times [1+C_{\delta,r_0} (C(\Delta_1,D_2,D_3,v_0)+\epsilon \Delta_2)(M^{-1} |u|)^{-2p+1}] \\
&\quad\times [r_0-r(u,v)]\\
\leq& \epsilon M,
\end{split}\ee
with $0<\epsilon\ll 1$.
\end{proposition}

\begin{proof}
To estimate $r\partial_v r(u, v)$, {\color{black} we first use \eqref{propeq:r1} to write:}
\bes
\partial_u(r\partial_v r+M-\f12 r)=-\f14\O^2-\f12\partial_u r.
\ees
Integrating both sides with respect to $u$, we have
\be\label{r partial v r}
\begin{split}
&(r\partial_v r+M-\f12 r)(u,v)\\
=&(r\partial_v r+M-\f12 r)(u_{r_0(v)},v)-\int_{u_{r_0(v)}}^u\f{\O^2}{4\partial_u r}\partial_u r d{\color{black}u'}-\f12\int^u_{u_{r_0(v)}}\partial_u r d{\color{black}u'}\\
=&(r\partial_v r+M-\f12 r)(u_{r_0(v)},v)-\int^u_{u_{r_0(v)}}\bigg(\f{\O^2}{4\partial_u r}+\f12\bigg) \partial_u r d{\color{black}u'}.
\end{split}
\ee

\noindent To bound the first term, we use Proposition \ref{prop:metricest}{\color{black}:}
\begin{align*}
\left|r \partial_vr-M+\frac{1}{2}r\right|(u_{r_0(v)},v)\leq&\: C(\Delta_1,D_2,D_3,v_0) M (M^{-1}v)^{-2p+1}.
\end{align*}

\noindent To control term $\f{\O^2}{-\partial_u r}$, we appeal to \eqref{conseq:r1}{\color{black}:}
$$\partial_u\bigg(\f{\partial_u r}{\O^2}\bigg)=-r\f{(\partial_u \phi)^2}{\O^2}.$$
This gives 
$$\partial_u \bigg(\log\f{\O^2}{-\partial_u r}\bigg)=\f{r}{\partial_u r}(\partial_u \phi)^2\leq 0.$$
Hence{\color{black},}
\bes
\begin{split}
\log\f{\O^2(u,v)}{-\partial_u r(u,v)}\leq \log\f{\O^2(u_{r_0(v)},v)}{-\partial_u r(u_{r_0(v)},v)}, \mbox{ and } \f{\O^2(u,v)}{-\partial_u r(u,v)}\leq \f{\O^2(u_{r_0(v)},v)}{-\partial_u r(u_{r_0(v)},v)} .
\end{split}
\ees
Together with Proposition \ref{prop:metricest} {\color{black}, we obtain}
\begin{align*}
\left|- \Omega^{-2}\partial_u r-\frac{1}{2}\right|(u_{r_0(v)},v)\leq&\:C(\Delta_1,D_2,D_3,v_0) (M^{-1}v)^{-2p+1},\\
\end{align*}
we have {\color{black}that}
\bes\begin{split}   
&|\int^u_{u_{r_0(v)}}\bigg(\f{\O^2}{4\partial_u r}+\f12\bigg)(u',v)\cdot \partial_u r(u',v) du'|\\
=&|\int^{r(u,v)}_{r_0}\bigg(\f{\O^2}{4\partial_u r}+\f12\bigg)(u',v) dr(u',v) |\\
\leq&\int_{r(u,v)}^{r_0}\bigg(|\f{\O^2}{4\partial_u r}|(u',v)+\f12\bigg) dr(u',v)\\
\leq&\int_{r(u,v)}^{r_0}\bigg(|\f{\O^2}{4\partial_u r}|(u_{r_0(v)},v)+\f12\bigg) dr(u',v)\\
\leq&|\int_{r(u,v)}^{r_0}[1+C(\Delta_1,D_2,D_3,v_0) (M^{-1}v)^{-2p+1}]\cdot dr(u',v)|\\
\leq&[1+C(\Delta_1,D_2,D_3,v_0) (M^{-1}v)^{-2p+1}]\cdot [r_0-r(u,v)].
\end{split}\ees
Back to (\ref{r partial v r}){\color{black}:} for $r(u,v)\leq r_0$ we hence get
\bes\begin{split}
|(r\partial_v r+M-\f12 r)(u,v)|\leq& C(\Delta_1,D_2,D_3,v_0)M (M^{-1}v)^{-2p+1}\\
&+[1+C(\Delta_1,D_2,D_3,v_0) (M^{-1}v)^{-2p+1}]\cdot [r_0-r(u,v)].
\end{split}\ees
Since $v$ is large and $r_0$ is small, we thus obtain
$$|r\partial_v r(u,v)+M|\leq \e M.$$

We then consider $\partial_v(r\partial_u r+M-\f12 r)$. In the same fashion, with the help of Proposition \ref{prop:durdvrr0}{\color{black}, we obtain}
\begin{align*}
\left|- \partial_u r- \left(Mr^{-1}-\frac{1}{2}\right)\right|(u, v_{r_0(u)})\leq&\: C_{\delta,r_0}(C(\Delta_1,D_2,D_3,v_0)+\epsilon \Delta_2)(M^{-1} |u|)^{-2p+1},\\
\left|\frac{\partial_u r}{\partial_vr}-1\right|(u, v_{r_0(u)})\leq&\: C_{\delta,r_0} (C(\Delta_1,D_2,D_3,v_0)+\epsilon \Delta_2)(M^{-1} |u|)^{-2p+1},
\end{align*}
we derive
\bes\begin{split}
&|(r\partial_u r+M-\f12 r)(u,v)|\leq C_{\delta,r_0}(C(\Delta_1,D_2,D_3,v_0)+\epsilon \Delta_2)(M^{-1} |u|)^{-2p+1}\cdot r_0\\
&+[1+C(\Delta_1,D_2,D_3,v_0) (M^{-1}v)^{-2p+1}]\cdot [1+C_{\delta,r_0} (C(\Delta_1,D_2,D_3,v_0)+\epsilon \Delta_2)(M^{-1} |u|)^{-2p+1}] \cdot [r_0-r(u,v)].
\end{split}\ees
Since $v, |u|$ are large and $r_0$ is small, we thus obtain
$$|r\partial_u r(u,v)+M|\leq \e M.$$
\end{proof}

\begin{remark} 
{\color{black}Along the spacelike singularity $\mathcal{S}$, for each $\t q\in \mathcal{S}$ with coordinate $(u_{\t q}, v_{\t q})$, via the same arguments as in \cite{DC91} by Christodoulou, we have}
$$\lim_{u\rightarrow u_{\t q}} r\partial_v r(u, v_{\t q}) \mbox{ exists, and } \lim_{v\rightarrow v_{\t q}} r\partial_u r(u_{\t q}, v) \mbox{ exists}. $$
Denote
\be\label{f q}
\lim_{u\rightarrow u_{\t q}} r\partial_v r(u, v_{\t q})=-M+f_2(v_{\t q}), \mbox{ and }  \lim_{v\rightarrow v_{\t q}} r\partial_u r(u_{\color{black} \t q}, v)=-M+f_1(u_{\t q}).
\ee
In \cite{AZ}, the first author and Zhang showed that, with $r_0>0$ sufficiently small, for any $(u,v)\in J^-(\t q)$ and $r(u,v)\leq r_0$ it holds {\color{black}that}
\begin{equation}\label{An Zhang r}
\begin{split}
|r\partial_u r(u,v)+M-f_1(u_{\t q})|\leq& M[M^{-1}r(u,v)]^{\f{1}{100}},\\
|r\partial_v r(u,v)+M-f_2(v_{\t q})|\leq& M[M^{-1}r(u,v)]^{\f{1}{100}},
\end{split}
\end{equation}
where $f_1(u)$ and $f_2(v)$ are continuous functions with respect to $u$ and $v$. Together with (\ref{rvr1}) (\ref{rur1}) and the triangle inequalities, we obtain
\be\label{1st f1 f2}
|f_1(u_{\t q})|\leq \e M+M[M^{-1}r_0]^{\f{1}{100}}\leq 2\e M, \quad |f_2(v_{\t q})|\leq \e M+M[M^{-1}r_0]^{\f{1}{100}}\leq 2\e M.
\ee

We will use (\ref{1st f1 f2}) to derive some basic estimates. And in {\color{black}Theorem} \ref{partial r}, we will revisit $r\partial_u r(u_{\t q}, v_{\t q})$ and $r\partial_v r(u_{\t q}, v_{\t q})$ to derive refined estimates. In Proposition \ref{partial r} we will {\color{black}moreover} improve (\ref{1st f1 f2}) and show that, for $(u_{\color{black} \t q},v_{\color{black}\t q})$ with $|u_{\t q}|\geq |u_1|+2v_{r=0}(u_1)$ and $r(u_{\t q},v_{\t q})=0$, the following inequalities hold:
\begin{align*}
|f_1(u_{\t q})|\leq& \f{|u_{\t q}-u_1-v_{r=0}(u_1)|^{-p}M^p\cdot M}{2}\leq |M^{-1}v_{\t q}|^{-p} M,\\
|f_2(v_{\t q})|\leq& \f{|M^{-1}v_{\t q}|^{-p}M}{2}\leq |u-u_1-v_{r=0}(u_1)|^{-p}M^p\cdot M.
\end{align*}
{\color{black}Together with} (\ref{An Zhang r}), the above two inequalities improve (\ref{rvr1}) {\color{black}and} (\ref{rur1}){\color{black}, so we will finally be able to conclude:}
\begin{equation*}
\begin{split}
|r\partial_u r(u,v)+M|\leq&M[M^{-1}r(u,v)]^{\f{1}{100}}+|u_q-u_1-v_{r=0}(u_1)|^{-p}M^p\cdot M,\\
|r\partial_v r(u,v)+M|\leq& M[M^{-1}r(u,v)]^{\f{1}{100}}+|u_q-u_1-v_{r=0}(u_1)|^{-p}M^p\cdot M.
\end{split}
\end{equation*}
\end{remark}

\subsection{Estimates for {\color{black} global coordinates} $u$ and $v$} 

For any $0\leq r_1\leq r_0$, along $r=r_1$, letting $r=r(u, v_{r_1}(u))$ we first prove a useful relation between $v_{r_1}(u)$ and $u$:

\begin{center}
\begin{figure}[h]
\begin{minipage}[!t]{0.4\textwidth}
\begin{tikzpicture}[scale=0.75]
\draw [white](-1, 0)-- node[midway, sloped, above,black]{$\big(u_1, v_{r=0}(u_1)\big)$}(1, 0);
\draw [white](2.9, 0)-- node[midway, sloped, above,black]{$\t q$}(3.1, 0);
\draw [white](1.7, 0)-- node[midway, sloped, above,black]{$\mathcal{S}$}(2, 0);
\draw [white](5.5, 0.2)-- node[midway, sloped, above,black]{$i^+$}(7, 0.2);
\draw (0,0) to [out=-5, in=195] (5.5, 0.5);
\draw [white](0, -3)-- node[midway, sloped, above,black]{$\mathcal{H}$}(7, -3);
\draw [white](0, -4.5)-- node[midway, sloped, above,black]{$v=v_1$}(-4, -4.5);
\draw [white](0, -1.8)-- node[midway, sloped, above, black]{$r=r_1$}(0.5,-1.8);
\draw [white](0, -2.8)-- node[midway, sloped, above, black]{$r=r_0$}(0.5,-2.8);
\draw [white](-2.5, -2.5)-- node[midway, sloped, above, black]{$u=u_1$}(0,0);
\draw [thick] (-2.5,-2.5)--(0,-5);
\draw [dashed](-2.5, -2.5)--(0,0);
\draw [thick] (5.5, 0.5)--(0,-5);
\draw (-1,-1) to [out=-5, in=210] (5.5, 0.5);
\draw (-2,-2) to [out=-5, in=215] (5.5, 0.5);
\draw[fill] (0,0) circle [radius=0.08];
\draw[fill] (5.5, 0.5) circle [radius=0.08];
\draw[fill] (-2,-2) circle [radius=0.08];
\draw[fill] (3,0) circle [radius=0.08];
\draw[fill] (3.68,-0.68) circle [radius=0.08];
\draw[fill](1.31, -1.69) circle [radius=0.08];
\draw [dashed](1.31, -1.69)--(3,0);
\draw [dashed](3.68, -0.68)--(3,0);
\end{tikzpicture}
\end{minipage}
\begin{minipage}[!t]{0.5\textwidth}
\end{minipage}
\hspace{0.05\textwidth}
\caption{A global relation between $v$ and $u$ along $r=r_1$.}. 
\end{figure}
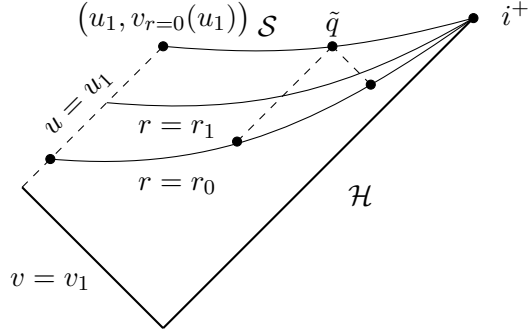
\end{center}

\begin{proposition}\label{v u global} 
Let $0\leq r_1\leq r_0$. Along $r=r_1$, for $|u|\geq |u_1|+2v_{r=0}(u_1)$, it holds {\color{black}that}
\begin{equation*}
\begin{split}
-[u-u_1-v_{r=0}(u_1)](1-6\e)\leq v_{r_1}(u)\leq -[u-u_1-v_{r=0}(u_1)](1+6\e).
\end{split}
\end{equation*}
\end{proposition}

\begin{proof}
Along $r=r_1$, we have $r_1^2/2=r(u, v_{r_1}(u))^2/2$. Differentiating this equation with respect to $u$, we have
\begin{equation*}
\begin{split}
0=&(r\partial_u r)(u, v_{r_1}(u))+v'_{r_1}(u)\cdot (r\partial_v r)(u, v_{r_1}(u))\\
=&[(r\partial_u r)(u, v_{r_1}(u))+M-f_1(u)]-M+f_1(u)\\
&+v'_{r_1}(u)\cdot [(r\partial_v r)(u, v_{r_1}(u))+M-f_2(v)]+v'_{r_1}(u)\cdot[-M+f_2(v)].
\end{split}
\end{equation*}
Using \eqref{An Zhang r} and \eqref{1st f1 f2} this gives
$$-1-5\e-3(M^{-1}r_1)^{\f{1}{100}}\leq v_{r_1}'(u)\leq -1+5\e+3(M^{-1}r_1)^{\f{1}{100}},$$
and
$$-(u-u_1)[-1-5\e-3(M^{-1}r_1)^{\f{1}{100}}] \leq v_{r_1}(u)-v_{r_1}(u_1)\leq -(u-u_1)[-1+5\e+3(M^{-1}r_1)^{\f{1}{100}}].$$
The above identity is equivalent to
\begin{equation}\label{vr1-v0}
\begin{split}
&-(u-u_1)[-1-5\e-3(M^{-1}r_1)^{\f{1}{100}}]+v_{r_1}(u_1)-v_{r=0}(u_1)\\
\leq& v_{r_1}(u)-v_{r=0}(u_1)\\
&\leq -(u-u_1)[-1+5\e+3(M^{-1}r_1)^{\f{1}{100}}]+v_{r_1}(u_1)-v_{r=0}(u_1).
\end{split}
\end{equation}
Along $u=u_1$, we connect $(u_1, v_{r_1}(u_1))$ to $(u_1, v_{r=0}(u_1))\in \mathcal{S}$ and have
\begin{equation*}
\begin{split}
-r_1^2(u_1, v_{r_1}(u_1))=&r^2(u_1, v_{r=0}(u_1))-r^2_1(u_1, v_{r_1}(u_1))\\
=&\int_{v_{r_1}(u_1)}^{v_{r=0}(u_1)}\partial_v[r^2](u_1, v')dv'\\
=&2\int_{v_{r_1}(u_1)}^{v_{r=0}(u_1)}\{[r\partial_v r(u_1,v')+M-f_2(v')]-M+f_2(v')\} dv'.
\end{split}
\end{equation*}
Thus \eqref{An Zhang r} and \eqref{1st f1 f2} imply $|v_{r_1}(u)-v_{r=0}(u_1)|\leq M^{-1}r_1^2$. Back to \eqref{vr1-v0}, we derive
\begin{equation*}
\begin{split}
&-(u-u_1)[-1-5\e-3(M^{-1}r_1)^{\f{1}{100}}]-M^{-1}r_1^2+v_{r=0}(u_1)\\
\leq& v_{r_1}(u)\\
&\leq -(u-u_1)[-1+5\e+3(M^{-1}r_1)^{\f{1}{100}}]+M^{-1}r_1^2+v_{r=0}(u_1).
\end{split}
\end{equation*}
For $|u|\geq |u_1|+2v_{r=0}(u_1)$ and $v_{r=0}(u_1)$ sufficiently large, we hence prove
$$-[u-u_1-v_{r=0}(u_1)](1-6\e)\leq v_{r_1}(u)\leq -[(u-u_1-v_{r=0}(u_1)](1+6\e).$$

\end{proof}

\subsection{Estimates for $\partial_u \phi(u,v)$ and $\partial_v \phi(u,v)$.}
Let $(u,v)\in J^-(\t q)$. We recall (\ref{f q})
$$\lim_{u\rightarrow u_{\t q}} r\partial_v r(u, v_{\t q})=-M+f_2(v_{\t q}), \mbox{ and }  \lim_{v\rightarrow v_{\t q}} r\partial_u r(u_{\t q}, v)=-M+f_1(u_{\t q}).$$
For further use, we denote $C_1(\t q):=M-f_1(u_{\t q})$ and $C_2(\t q):=M-f_2(v_{\t q})$. And we have

\begin{proposition}\label{phi}
For $(u,v)\in J^-(\t q)$ with $|u|\geq |u_1|+2v_{r=0}(u_1)$ and $r(u,v)\leq r_0$, {\color{black} the following estimates hold:}

\be\label{6.4}
\begin{split}
r^2|\partial_v \phi|(u,v)\leq \f{\tilde{D}_1 \cdot M}{M^{-p}|u-u_1-v_{r=0}(u_1)|^p},\\
r^2|\partial_u \phi|(u,v)\leq \f{\tilde{D}_2 \cdot M}{M^{-p}|u-u_1-v_{r=0}(u_1)|^p}.
\end{split}
\ee
Here $\tilde{D}_1=\tilde{D}_2=8\D_1$ and $\D_1$ is defined in Proposition \ref{prop:baimpphired}. 

\noindent And
\be\label{6.5}
\begin{split}
r^2\partial_v \phi(u,v)\geq&\f{D_1'\cdot M}{M^{-q}|u-u_1-v_{r=0}(u_1)|^q},\\
r^2\partial_u \phi(u,v)\geq&\f{D_2'\cdot M}{M^{-q}|u-u_1-v_{r=0}(u_1)|^q}.
\end{split} 
\ee
Here $D'_1=D'_2=(3-2\sqrt{2})\cdot (D_1-\epsilon)$ and $D_1$ is defined in \eqref{eq:estdataphi}.
\end{proposition}

\begin{proof}
For $(u,v)\in J^-(\t q)$, we treat $C_1(\t q)$ and $C_2(\t q)$ as constants and obtain
\begin{equation*}
\partial_u (C_1 r\partial_v \phi )=\f{C_1}{C_2}\cdot\f{1}{r}\cdot \f{-\partial_v r}{\partial_u r}\cdot C_2\cdot (r\partial_u \phi)\cdot \partial_u r,
\end{equation*}
\begin{equation*}
\partial_v (C_2 r\partial_u \phi )=\f{C_2}{C_1}\f{1}{r}\cdot \f{-\partial_u r}{\partial_v r}\cdot C_1\cdot (r\partial_v \phi)\cdot\partial_v r.
\end{equation*}
The above two equations are equivalent to the following forms: 
\begin{equation}\label{weighted partial v phi}
\partial_u \bigg(C_1 r\partial_v \phi \bigg)=\f{C_1}{C_2}\cdot \f{1}{r}\cdot \f{-\partial_v r}{\partial_u r}\cdot \bigg(C_2 r\partial_u \phi\bigg)\cdot \partial_u r,
\end{equation}
\begin{equation}\label{weighted partial u phi}
\partial_v \bigg(C_2 r\partial_u \phi \bigg)=\f{C_2}{C_1}\cdot\f{1}{r}\cdot \f{-\partial_u r}{\partial_v r}\cdot \bigg(C_1 r\partial_v \phi\bigg)\cdot\partial_v r.
\end{equation}
Note by \eqref{An Zhang r} and \eqref{1st f1 f2}, there exist $h_1(u,v), h_2(u,v)$ satisfying
$$0\leq \f{|h_1(u,v)|}{[M^{-1}r(u,v)]^{\f{1}{100}}}\lesssim 1, \quad 0\leq \f{|h_2(u,v)|}{[M^{-1}r(u,v)]^{\f{1}{100}}}\lesssim 1,$$
and it holds
\bes
\f{C_1}{C_2}\cdot \f{-\partial_v r}{\partial_u r}=-1-h_1, \quad \f{C_2}{C_1}\cdot \f{-\partial_u r}{\partial_v r}=-1-h_2.
\ees
We then rewrite (\ref{weighted partial v phi}), (\ref{weighted partial u phi}) and arrive at
\begin{equation}\label{weighted partial v phi 1}
\partial_u \bigg(C_1 r\partial_v \phi \bigg)=-\f{1+h_1}{r}\cdot \bigg(C_2 r\partial_u \phi\bigg)\cdot \partial_u r,
\end{equation}
\begin{equation}\label{weighted partial u phi 1}
\partial_v \bigg(C_2 r\partial_u \phi \bigg)=-\f{1+h_2}{r}\cdot \bigg(C_1 r\partial_v \phi\bigg)\cdot\partial_v r.
\end{equation}

\noindent We then consider constant $r$-level sets $\{L_r\}$ in $J^-(\t q)$. Let
$$\tilde{\Psi}(r):=\max\{\sup_{P\in L_r} | C_2\cdot r\partial_u \phi|(P), \sup_{Q\in L_r} |C_1\cdot r\partial_v \phi|(Q)\}.$$
In the below, we prove (\ref{6.4}) first.
\begin{center}
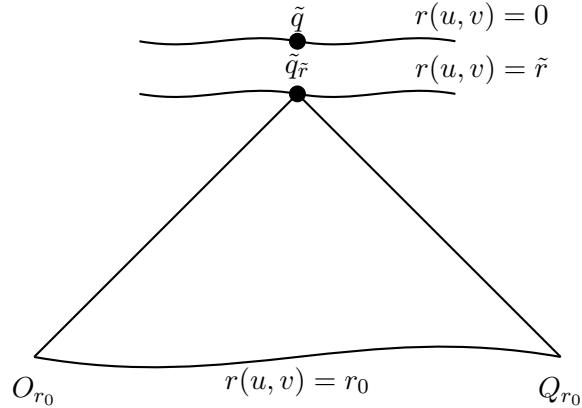
\begin{figure}[H]
\begin{minipage}[!t]{0.4\textwidth}
\begin{tikzpicture}[scale=0.7]
\draw [white](4, 1)-- node[midway, sloped, above,black]{$r(u,v)=0$}(3, 1);
\draw [white](-2, 1)-- node[midway, sloped, above,black]{$\t q$}(2, 1);
\draw [white](-2, 0.1)-- node[midway, sloped, above,black]{$\t q_{\t r}$}(2, 0.1);
\draw [white](-6, -5.2)-- node[midway, sloped, below,black]{$O_{r_0}$}(-4, -5.2);
\draw [white](6, -5.2)-- node[midway, sloped, below,black]{$Q_{r_0}$}(4, -5.2);
\draw [white](4, 0)-- node[midway, sloped, above,black]{$r(u,v)=\t r$}(3, 0);

\draw[thick] (0,0)--(-5,-5);
\draw[thick] (0,0)--(5, -5);
\draw [thick] (-3,0) to [out=-10,in=170] (0,0);
\draw [thick] (0,0) to [out=-10,in=170] (3,0);
\draw [thick] (-3,1) to [out=-10,in=170] (0,1);
\draw [thick] (0,1) to [out=-10,in=170] (3,1);
\draw [white](-5, -5)-- node[midway, sloped, below,black]{$r(u,v)=r_0$}(5, -5);
\draw[thick] (-5, -5) to [out=-10, in=170] (5, -5);
\draw[fill] (0,0) circle [radius=0.15];
\draw[fill] (0,1) circle [radius=0.15];
\end{tikzpicture}
\end{minipage}
\hspace{0.05\textwidth}
\begin{minipage}[!t]{0.6\textwidth}
\end{minipage}
\caption{Estimates for $\partial_u \phi$ and $\partial_v \phi$.}
\end{figure}
\end{center}

\noindent For any $\t q_{\t r}\in J^-(\t q)\cap L_{\t r}$, integrating (\ref{weighted partial v phi 1}), we arrive at
\begin{equation}
\begin{split}
|C_1 r\partial_v \phi|(\t q_n)\leq& \tilde{\Psi}(r_0)+ \int_{u(Q_{r_0})}^{u(\t q_{\t r})} -\f{1+h_1}{r}\cdot |C_2 r\partial_u \phi|\cdot \partial_u r \Bigr\vert_{v=v_{\t q_{\t r}}} \, du\\
=&\tilde{\Psi}(r_0)+\int_{r(Q_{r_0})}^{r(\t q_{\t r})}-\f{1+h_1}{r}\cdot |C_2 r\partial_u \phi| \Bigr\vert_{v=v_{\t q_{\t r}}} \, dr\\
\end{split}
\end{equation}
Similarly, we have
$$|C_2 r\partial_u \phi|(\t q_{\t r})\leq \tilde{\Psi}(r_0)+\int_{r(O_{r_0})}^{r(\t q_{\t r})} -\f{1+h_2}{r}\cdot |C_1 r\partial_v\phi| \Bigr\vert_{u=u_{\t q_{\t r}}} \, dr.$$
Combining these two inequalities together, we have
$$\tilde{\Psi}(\t r)\leq \tilde{\Psi}(r_0)+\int_{r_0}^{\t r} -\f{1+\max\{h_1, h_2\}}{r}\cdot \tilde{\Psi}(r) \, dr.$$
By Gr\"onwall inequality, we have
\begin{align*}
\tilde{\Psi}(\tilde{r})\leq& \tilde{\Psi}(r_0)\times e^{\int_{\tilde{r}}^{r_0} \f{1+\max\{h_1, h_2\}}{r}} dr=\tilde{\Psi}(r_0)\times e^{-\ln \f{\tilde{r}}{r_0}+\int_{\tilde{r}}^{r_0} \f{\max\{h_1, h_2\}}{r} dr}\leq\f{2r_0\tilde{\Psi}(r_0)}{\tilde{r}}. 
\end{align*}
This gives 
$$\tilde{r}\tilde{\Psi}(\tilde{r})\leq 2r_0\tilde{\Psi}(r_0) \mbox{ for any } \tilde{r}>0.$$ 
Hence, for $(u,v)\in J^-(\t q)$ and $r(u,v)\leq r_0$ we have
\be\label{phi middle step}
C_2\cdot r^2|\partial_u \phi|(u,v)\leq 2r_0\tilde{\Psi}(r_0), \quad C_1\cdot r^2|\partial_v \phi|(u,v)\leq 2r_0\tilde{\Psi}(r_0).
\ee
Recall for $P\in J^-(\t q)$, $\tilde{\Psi}(r_0)$ is defined through
$$\tilde{\Psi}(r_0):=\max\{\sup_{r(P)=r_0} |C_2\cdot r\partial_u \phi|(P), \sup_{r(Q)=r_0} |C_1\cdot r\partial_v \phi|(Q)\}.$$ 
To estimate $\tilde{\Psi}(r_0)$ we use {\color{black}Proposition \ref{prop:baimpphired}: for $r(u',v')=r_0$ it holds {\color{black}that} 
\be
r^2|\partial_v\phi|(u_{r_0(v')},v')\leq M{\Delta_1}(M^{-1}v')^{-p}, \quad r^2|\partial_u \phi|(u_{r_0(v')},v')\leq M{\Delta_1}  (M^{-1}v')^{-p}.
\ee}
Using Lemma \ref{lm:relatuvconstr} and Proposition \ref{v u global}, for $(u',v')$ satisfying $|u'|\geq |u_1|+2v_{r=0}(u_1)$ they imply that for $r(u',v')=r_0${\color{black}:} {\color{black}
$$|C_1\cdot r^2\partial_v\phi|(u_{r_0(v')},v')\leq C_1 M{\Delta_1}(M^{-1}v')^{-p}\leq 2C_1\cdot{M\cdot\Delta_1} |u_{r_0(v')}-u_1-v_{r=0}(u_1)|^{-p}M^p,$$
$$|C_2\cdot r^2\partial_u\phi|(u_{r_0(v')},v')\leq C_2M{\Delta_1}(M^{-1}v')^{-p}\leq 2C_2\cdot M\cdot\Delta_1 |u_{r_0(v')}-u_1-v_{r=0}(u_1)|^{-p}M^p.$$}
Note that for $(u_{r_0(v')}, v')\in J^-\big((u,v)\big)${\color{black}, we can estimate}
$$|u_{r_0(v')}-u_1-v_{r=0}(u_1)|^{-p}\leq |u-u_1-v_{r=0}(u_1)|^{-p}.$$
Back to (\ref{phi middle step}){\color{black}:} together with 
$$(1-2\e)M\leq C_1(q)\leq (1+2\e)M, \quad (1-2\e)M\leq C_2(q)\leq (1+2\e)M,$$
we hence prove (\ref{6.4}):
{\color{black}
$$r^2|\partial_v \phi|(u,v)\leq\f{3(C_1+C_2)\cdot M\cdot\Delta_1}{C_1}\cdot \f{1}{|u-u_1-v_{r=0}(u_1)|^p M^{-p}}\leq\f{\tilde{D}_1 \cdot M}{|u-u_1-v_{r=0}(u_1)|^p M^{-p}},$$
$$r^2|\partial_u \phi|(u,v)\leq\f{3(C_1+C_2)\cdot M\cdot\Delta_1}{C_2}\cdot \f{1}{|u-u_1-v_{r=0}(u_1)|^p M^{-p}}=\f{\tilde{D}_2 \cdot M}{|u-u_1-v_{r=0}(u_1)|^p M^{-p}},$$
where we choose $\tilde{D}_1=\tilde{D}_2=8\D_1$. 
}

We then move to prove (\ref{6.5}). Let
$$\Psi(r):=\min\{\inf_{P\in L_r} | C_2\cdot r\partial_u \phi|(P), \inf_{Q\in L_r} |C_1\cdot r\partial_v \phi|(Q)\}.$$
\noindent Integrating (\ref{weighted partial v phi 1}), we arrive at 
\begin{equation}
\begin{split}
C_1 r\partial_v \phi(\t q_{\t r})= &C_1 r\partial_v \phi(Q_{r_0})+\int_{u(Q_{r_0})}^{u(\t q_{\t r})} -\f{1+h_1}{r}\cdot C_2 r\partial_u \phi\cdot \partial_u r \Bigr\vert_{v=v_{\t q_{\t r}}} \, du\\
=&C_1 r\partial_v \phi(Q_{r_0})+\int_{r(Q_{r_0})}^{r(\t q_{\t r})}-\f{1+h_1}{r}\cdot C_2 r\partial_u \phi \Bigr\vert_{v=v_{\t q_{\t r}}} \, dr\\
\end{split}
\end{equation}
Similarly, we have
$$C_2 r\partial_u \phi(\t q_{\t r})= C_2 r\partial_u \phi(O_{r_0})+\int_{r(O_{r_0})}^{\t r} -\f{1+h_2}{r}\cdot C_1 r\partial_v\phi \Bigr\vert_{u=u_{\t q_{\t r}}} \, dr.$$
Combining these two inequalities together, we have
$$\Psi(\t r)\geq \Psi(r_0)+\int_{r_0}^{\t r} -\f{1+\min\{h_1, h_2\}}{r}\cdot \Psi(r) \, dr.$$
By the reverse Gr\"onwall inequality (see Proposition \ref{appendix 1}), we hence obtain
\begin{align*}
\Psi(\tilde{r})\geq& \Psi(r_0)\times e^{\int_{\tilde{r}}^{r_0} \f{1+\min\{h_1,h_2\}}{r}} dr=\Psi(r_0)\times e^{-\ln \f{\tilde{r}}{r_0}+\int_{\tilde{r}}^{r_0} \f{1+\min\{h_1,h_2\}}{r} dr }\geq\f{r_0\Psi(r_0)}{2\tilde{r}}.
\end{align*}
This gives
\bes
\tilde{r}\Psi(\tilde{r})\geq r_0\Psi(r_0)/2 \mbox{ for any } \tilde{r}>0.
\ees
For $(u,v)\in J^-(\t q)$ and $r(u,v)\leq r_0$ we hence have
\be\label{Psi r tilde 2}
r^2|\partial_u \phi|(u,v)\geq r_0\Psi(r_0)/2, \quad r^2|\partial_v \phi|(u,v)\geq r_0\Psi(r_0)/2.
\ee
Recall for $P\in J^-(\t q)$ we define $\Psi(r_0)$ as
$$\Psi(r_0):=\min\{\inf_{r(P)=r_0} C_2\cdot r\partial_u \phi(P), \inf_{r(Q)=r_0} C_1\cdot r\partial_v \phi(Q)\}.$$ 
At the same time, to control $\Psi(r_0)$ we apply Proposition \ref{prop:lowboundsposr}: for $r(u',v')=r_0$ it holds
$$r^2\partial_v\phi(u_{r_0(v')},v')\geq (6-4\sqrt{2})M(D_1-\epsilon) (M^{-1}v')^{-q},$$
$$r^2\partial_u\phi(u_{r_0(v')},v)\geq (6-4\sqrt{2})M(D_1-\epsilon)(M^{-1}v')^{-q}.$$
Using Lemma \ref{lm:relatuvconstr} and Proposition \ref{v u global}, for $(u',v')$ satisfying $u'\geq u_1+2v_0(u_1)$ and $r(u',v')=r_0$, we have 
$$C_1\cdot r^2\partial_v\phi(u_{r_0(v')},v')\geq (6-4\sqrt{2})C_1 M(D_1-\epsilon) (M^{-1}v')^{-q},$$
$$C_2\cdot r^2\partial_v\phi(u_{r_0(v')},v')\geq (6-4\sqrt{2})C_2 M(D_1-\epsilon) (M^{-1}v')^{-q}.$$
Note that for $(u_{r_0(v')}, v')\in J^-\big((u,v)\big)$ it holds $v'\leq v$ and $v'^{-q}\geq v^{-q}$. 
Back to (\ref{Psi r tilde 2}), together with Proposition \ref{v u global} we hence prove (\ref{6.5}){\color{black}:}
\begin{align*}
r^2\partial_v \phi(u,v)\geq&\f{D_1'\cdot M}{M^{-q} |u-u_1-v_{r=0}(u_1)|^q},\\
r^2\partial_u \phi(u,v)\geq&\f{D_2'\cdot M}{M^{-q}|u-u_1-v_{r=0}(u_1)|^q},
\end{align*}
where we choose $D_1'=D_2'=(3-2\sqrt{2})(D_1-\epsilon)$.
\end{proof}

After proving estimates for $\partial_u \phi$ and $\partial_v \phi$, we are ready to control $\O^2(u,v)$. 

\subsection{Estimate for $\O^2(u,v)$} We start from deriving bounds for $\log(r\O^2)$. And we have 
\begin{proposition}\label{Omega 02}
For $(u,v)$, with $|u|\geq |u_1|+2v_{r=0}(u_1)$, $r(u,v)\leq r_0$ and $v$ sufficiently large, we have
\be\label{rOmega2 v}
 \f{{D}'_1 {D}'_2 M}{3 r^2(M^{-1}v)^{2q}}-\f{2}{r(u,v)}\leq -\partial_v\log(r\O^2)(u,v)\leq \f{3\t D_1 \t D_2 M}{r^2(M^{-1}v)^{2p}}+\f{2}{r(u,v)},
\ee
\be\label{rOmega2 u}
\f{{D}'_1 {D}'_2 M}{3r^2(M^{-1}v)^{2q}}-\f{2}{r(u,v)}\leq -\partial_u\log(r\O^2)(u,v)\leq \f{3\t D_1 \t D_2 M}{r^2(M^{-1}v)^{2p}}+\f{2}{r(u,v)},
\ee
where $\t D_1, \t D_2, D_1', D_2'$ are defined in Proposition \ref{phi}.  
\end{proposition}

\begin{center}
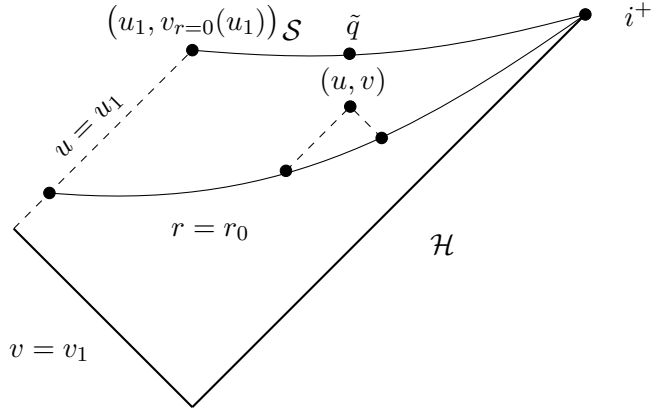
\begin{figure}[H]
\begin{minipage}[!t]{0.4\textwidth}
\begin{tikzpicture}[scale=0.95]
\draw [white](-1, 0)-- node[midway, sloped, above,black]{$\big(u_1, v_{r=0}(u_1)\big)$}(1, 0);
\draw [white](1.3, 0)-- node[midway, sloped, above,black]{$\mathcal{S}$}(1.5, 0);
\draw [white](5.5, 0.2)-- node[midway, sloped, above,black]{$i^+$}(7, 0.2);
\draw (0,0) to [out=-5, in=195] (5.5, 0.5);
\draw [white](0, -3)-- node[midway, sloped, above,black]{$\mathcal{H}$}(7, -3);
\draw [white](0, -4.5)-- node[midway, sloped, above,black]{$v=v_1$}(-4, -4.5);
\draw [white](2, 0)-- node[midway, sloped, above,black]{$\t q$}(2.5, 0); 
\draw [white](0, -2.8)-- node[midway, sloped, above, black]{$r=r_0$}(0.5,-2.8);
\draw [white](-2.5, -2.5)-- node[midway, sloped, above, black]{$u=u_1$}(0,0);
\draw [white](2, -0.79)-- node[midway, sloped, above, black]{$(u,v)$}(2.5, -0.79);
\draw [thick] (-2.5,-2.5)--(0,-5);
\draw [dashed](-2.5, -2.5)--(0,0);
\draw [thick] (5.5, 0.5)--(0,-5);
\draw (-2,-2) to [out=-5, in=215] (5.5, 0.5);
\draw[fill] (0,0) circle [radius=0.08];
\draw[fill] (5.5, 0.5) circle [radius=0.08];
\draw[fill] (-2,-2) circle [radius=0.08];
\draw[fill] (2.65, -1.23) circle [radius=0.08];
\draw[fill] (2.21,-0.79) circle [radius=0.08];
\draw[fill] (2.2,-0.05) circle [radius=0.08];
\draw[fill](1.31, -1.69) circle [radius=0.08];
\draw [dashed](1.31, -1.69)--(2.21,-0.79);
\draw [dashed](2.65, -1.23)--(2.21,-0.79);
\end{tikzpicture}
\end{minipage}
\begin{minipage}[!t]{0.5\textwidth}
\end{minipage}
\hspace{0.05\textwidth}
\caption{Estimate of $\O^2(u,v)$ in $J^-(\t q)$.}
\end{figure}
\end{center}

\begin{proof}
Here we will employ {\color{black}the} key equation \eqref{propeq:Omegarescale1}: 
$$-\partial_u \partial_v \log (r\O^2)=-\f{\O^2}{4r^2}+2\partial_u\phi \partial_v \phi.$$
For $(u,v)\in J^-(\t q)$, we integrate the above equation with respect to $u$ and get
$$-\partial_v \log(r\O^2)(u,v)=-\partial_v \log(r\O^2)(u_{r_0}(v),v)+\int_{u_{r_0(v)}}^u\bigg(-\f{\O^2}{4r^2}+2\partial_u \phi\partial_v \phi\bigg)(u',v)du'.$$
By Proposition \ref{phi}, Proposition \ref{v u global} and estimates for $r\partial_u r$ in \eqref{An Zhang r}, it holds {\color{black}that}
\begin{align*}
\int_{u_{r_0(v)}}^u 2|\partial_u \phi\cdot&\partial_v \phi|(u',v) du'\leq \int_{u_{r_0(v)}}^u \f{3\t D_1 \t D_2 M^2 (M^{-1}v)^{-2p}}{r'^4(u',v)}du'\\
=&(M^{-1}v)^{-2p}\int_{r_0}^{r(u,v)} \f{1}{r'^3}\cdot \f{3\t D_1 \t D_2 M^2}{r'\partial_u r}dr' \leq \f{3\t D_1 \t D_2 M^2}{M\cdot r^2 \cdot (M^{-1}v)^{2p}}{\color{black},}
\end{align*}
{\color{black}and also}
\begin{align*}
\int_{u_{r_0(v)}}^u 2\partial_u \phi\cdot\partial_v \phi&(u',v) du'\geq \int_{u_{r_0(v)}}^u \f{{D}'_1{D}'_2 M^2 (M^{-1}v)^{-2q}}{r'^4(u',v)}du'\\
=&(M^{-1}v)^{-2q}\int_{r_0}^{r(u,v)} \f{1}{r'^3}\cdot \f{{D}'_1 {D}'_2 M^2}{r'\partial_u r}dr' \geq \f{{D}'_1 {D}'_2 r_0^2}{3M\cdot r^2 \cdot (M^{-1}v)^{2q}}.
\end{align*}
Recall $\partial_u (\f{\O^2}{-\partial_u r})\leq 0$. Together with {\color{black} the following estimate from} Proposition \ref{prop:metricest} 
$$\left|- \Omega^{-2}\partial_u r-\frac{1}{2}\right|(u,v)\leq C(\Delta_1,D_2,D_3,v_0) (M^{-1}v)^{-2p+1},$$
we have
$$|\int_{u_{r_0(v)}}^u -\f{\O^2}{4r^2} du'|=\int_{u_{r_0(v)}}^u \f{\O^2}{-\partial_u r}\cdot \f{1}{4r^2}\cdot -\partial_u rdu'\leq -\int_{r_0}^r\f{3}{4r^2}dr\leq \f{1}{r}.$$
{\color{black}And {\color{black} for $v$ sufficiently large, see Remark \ref{v large}, we have that} 
\be
|\partial_v \log(r\O^2)(u_{r_0(v)},v)|\leq\f{1}{r_0}.
\ee}
Combining the estimates above, we hence prove (\ref{rOmega2 v})
\bes
 \f{{D}'_1 {D}'_2 M}{3 r^2 (M^{-1}v)^{2q}}-\f{2}{r(u,v)}\leq -\partial_v\log(r\O^2)(u,v)\leq \f{3\t D_1 \t D_2 M}{r^2 (M^{-1}v)^{2p}}+\f{2}{r(u,v)},
\ees
In the same manner, we also obtain (\ref{rOmega2 u}). 
\end{proof}

With Proposition \ref{Omega 02}, we then estimate $r\O^2(u,v)$. 
\begin{proposition}\label{Omega 03} 
For $(u,v)$, with $|u|\geq |u_1|+2v_{r=0}(u_1)$, $r(u,v)\leq r_0$ and $v$ sufficiently large, we have
$$\f{2M-r_0}{4}\cdot\f{r_0^{-{\f{6 \t D_1 \t D_2}{1-\e}\cdot\f{1}{M^{-2p}|u-u_1+v_{r=0}(u_1)|^{2p}}}}}{r(u,v)^{-{\f{6 \t D_1 \t D_2}{1-\e}\cdot\f{1}{M^{-2p}|u-u_1+v_{r=0}(u_1)|^{2p}}}}} \leq r\O^2(u,v)\leq 4\cdot(2M-r_0)\cdot \f{r_0^{-{\f{D_1' D_2'}{6(1+\e)}\cdot\f{1}{M^{-2p}|u-u_1+v_{r=0}(u_1)|^{2q}}}}}{r(u,v)^{-{\f{D_1' D_2'}{6(1+\e)}\cdot\f{1}{M^{-2p}|u-u_1+v_{r=0}(u_1)|^{2q}}}}},$$
where $\t D_1, \t D_2, D_1', D_2'$ are defined in Proposition \ref{phi}.\\
\end{proposition} 

\begin{proof}
Using the above proposition we have
\bes\begin{split}   
&\log(r\O^2)(u,v)=\log(r\O^2)(u,v_{r_0(u)})+\int^v_{v_{r_0(u)}}\partial_v \log(r\O^2)(u,v')dv'\\
\geq&\log(r\O^2)(u,v_{r_0(u)})+\int^v_{v_{r_0(u)}}-\f{3\t D_1 \t D_2 M}{r(u,v)^2}\f{1}{(M^{-1}v)^{2p}}-\f{2}{r(u,v)}dv'\\
\geq&\log(r\O^2)(u,v_{r_0(u)})+\int^v_{v_{r_0(u)}}-\f{6\t D_1 \t D_2 M}{r(u,v)^2}\f{1}{M^{-2p}|u-u_1+v_{r=0}(u_1)|^{2p}}dv'-\f{2r_0}{M(1-\e)}\\
=&\log(r\O^2)(u,v_{r_0(u)})+\int^v_{v_{r_0(u)}}-\f{6\t D_1 \t D_2 M}{r(r\partial_v r)}\f{1}{M^{-2p}|u-u_1+v_{r=0}(u_1)|^{2p}}\cdot \partial_v r dv'-\f{2r_0}{M(1-\e)}\\
\geq&\log(r\O^2)(u,v_{r_0(u)})+\int^{r(u,v)}_{r_0}\f{1}{r}\cdot\f{6\t D_1 \t D_2}{1-\e}\cdot\f{1}{M^{-2p}|u-u_1+v_{r=0}(u_1)|^{2p}}dr-\f{2r_0}{M(1-\e)}\\
=&\log(r\O^2)(u,v_{r_0(u)})+[\ln(\f{r}{r_0})]\cdot\f{6\t D_1 \t D_2}{1-\e}\cdot\f{1}{M^{-2p}|u-u_1+v_{r=0}(u_1)|^{2p}}-\f{2r_0}{M(1-\e)}.
\end{split}\ees
For $r_0$ sufficiently small, this implies 
$$r\O^2(u,v)\geq \f12 (r\O^2)(u,v_{r_0(u)})\cdot [\f{r(u,v)}{r_0}]^{\f{6\t D_1 \t D_2}{1-\e}\cdot\f{1}{M^{-2p}|u-u_1+v_{r=0}(u_1)|^{2p}}}.$$
From \eqref{eq:rescaledOmegaestprop}, we have
$$r\O^2(u,v_{r_0(u)})\geq \f{2M-r_0}{2},$$
which {\color{black}implies}
$$r\O^2(u,v)\geq \f{2M-r_0}{4}\cdot [\f{r(u,v)}{r_0}]^{\f{6\t D_1 \t D_2}{1-\e}\cdot\f{1}{M^{-2p}|u-u_1+v_{r=0}(u_1)|^{2p}}}.$$ 
In the same manner, we also obtain
$$r\O^2(u,v)\leq 4\cdot(2M-r_0)\cdot [\f{r(u,v)}{r_0}]^{\f{D_1' D_2'}{6(1+\e)}\cdot\f{1}{M^{-2q}|u-u_1+v_{r=0}(u_1)|^{2q}}}.$$
This concludes the proof.
\end{proof} 

As a corollary of above estimates, we also {\color{black}establish \underline{\textit{mass inflation}}, which is the blow-up of the Hawking mass $m$, defined in \eqref{Hawking mass}, at $r=0$.}
\begin{proposition}\label{mass inflation}
For $(u,v)$, with $|u|\geq |u_1|+2v_{r=0}(u_1)$, $r(u,v)\leq r_0$ and $v$ sufficiently large, we have
\be
\begin{split}
\f{M}{8}\cdot{\color{black}\Big[}\f{r_0}{r(u,v)}{\color{black}\Big]}^{\f{D_1' D_2'}{6(1+\e)}\cdot\f{M^{2q}}{|u-u_1+v_{r=0}(u_1)|^{2q}}}\leq m(u,v)\leq \f{r(u,v)}{2}+8M\cdot {\color{black}\Big[}\f{r_0}{r(u,v)}{\color{black}\Big]}^{\f{\t D_1 \t D_2}{6(1-\e)}\cdot\f{M^{2p}}{|u-u_1+v_{r=0}(u_1)|^{2p}}}.
\end{split}
\ee
Here $D'_1=D'_2=(3-2\sqrt{2})\cdot (D_1-\epsilon)$ and $D_1$ is defined in \eqref{eq:estdataphi}
\end{proposition}
\begin{proof}
With Proposition \ref{Omega 03} we have
\begin{equation}\label{mass inflation lower bound}
\begin{split}
m(u,v)\geq&\f{r(u,v)}{2}\bigg(1+\f{2}{4\cdot 2M}\cdot [\f{r_0}{r(u,v)}]^{\f{D_1' D_2'}{6[1+\e]}\cdot\f{1}{M^{-2q}|u-u_1+v_{r=0}(u_1)|^{2q}}}\cdot \f{M^2}{r(u,v)}\bigg)\\
\geq&\f{M}{8}\cdot [\f{r_0}{r(u,v)}]^{\f{D_1' D_2'}{6(1+\e)}\cdot\f{1}{M^{-2q}|u-u_1+v_{r=0}(u_1)|^{2q}}},
\end{split}
\end{equation}
and
\begin{align*}
m(u,v)\leq&\f{r(u,v)}{2}\bigg(1+\f{2\cdot 4\cdot 4}{2M}\cdot [\f{r_0}{r(u,v)}]^{\f{\t D_1 \t D_2'}{6(1-\e)}\cdot\f{1}{M^{-2p}|u-u_1+v_{r=0}(u_1)|^{2p}}}\cdot \f{M^2}{r(u,v)}\bigg)\\
\leq&\f{r(u,v)}{2}+8M\cdot [\f{r_0}{r(u,v)}]^{\f{\t D_1 \t D_2}{6(1-\e)}\cdot\f{1}{M^{-2p}|u-u_1+v_{r=0}(u_1)|^{2p}}}.
\end{align*}
\end{proof}

\section{Estimates for {\color{black}} Kretschmann scalar} {\color{black}In this section, we apply the upper and lower bound estimates from Section \ref{sec:estnearsing} to derive upper and lower bound estimates for the Kretschmann scalar.}
\begin{proposition}\label{thm 8.1} 
For $(u,v)$ with $|u|\geq |u_1|+2v_{r=0}(u_1)$, $r(u,v)\leq r_0$ and $v$ sufficiently large, we have 
$$\f{c(M)}{r^{6+\f{D_1' D_2'}{6(1+\e)}\cdot\f{{\color{black}M^{2q}}}{|u-u_1+v_{r=0}(u_1)|^{2q}}}}\leq R^{\a\b\mu\nu}R_{\a\b\mu\nu}(u,v) \leq \f{C(M)}{r(u,v)^{6+\f{24 \t D_1 \t D_2}{1-\e}\cdot\f{{\color{black}M^{2p}}}{|u-u_1+v_{r=0}(u_1)|^{2p}}}},$$
where $c(M), C(M)$ are constants depending only on $M$ and $\t D_1, \t D_2, D_1', D_2'$ are defined in Proposition \ref{phi}. 
\end{proposition}
 
\begin{proof}  

We start from deriving the lower bound. For Kretschmann scalar, in \cite{DC91} we have 
$$R^{\a\b\mu\nu}R_{\a\b\mu\nu}(u,v)\geq \f{32 m(u,v)^2}{r(u,v)^6}.$$
Via \eqref{mass inflation lower bound}, it holds
\bes\begin{split}   
m(u,v)\geq&\f{M}{8}\cdot [\f{r_0}{r}]^{\f{D_1' D_2'}{6(1+\e)}\cdot\f{1}{M^{-2q}|u-u_1+v_{r=0}(u_1)|^{2q}}}.
\end{split}\ees
Therefore, we obtain
\begin{align*}
R^{\a\b\mu\nu}R_{\a\b\mu\nu}(u,v)\geq& \f{\big( 4M\cdot r_0^{\f{D_1' D_2'}{6(1+\e)}\cdot\f{1}{M^{-2q}|u-u_1+v_{r=0}(u_1)|^{2q}}}\big)^2}{{}{r^{6+\f{D_1' D_2'}{6(1+\e)}\cdot\f{1}{M^{-2q}|u-u_1+v_{r=0}(u_1)|^{2q}}}}}\geq \f{c(M)}{r^{6+\f{D_1' D_2'}{6(1+\e)}\cdot\f{1}{M^{-2q}|u-u_1+v_{r=0}(u_1)|^{2q}}}}.
\end{align*}
Note that for $|u|$ large, $c(M)$ is a constant depending only on $M$.

We then move to derive the sharp upper bound of $R^{\a\b\mu\nu}R_{\a\b\mu\nu}$. In \cite{AZ},  the following expression for the Kretschmann scalar is derived:
\begin{equation*}
\begin{split}
&R^{\alpha\beta {\color{black} \mu \nu}}R_{\alpha\beta {\color{black} \mu \nu}}\\
=&\f{4}{r^4 \O^8}\bigg(16\cdot(\f{\partial^2 r}{\partial u \partial v})^2\cdot r^2\cdot\O^4+16\cdot \f{\partial^2 r}{\partial u^2}\cdot\f{\partial^2 r}{\partial v^2} \cdot r^2\cdot\O^4 \bigg)\\
&+\f{4}{r^4 \O^8}\bigg(-32\cdot \f{\partial^2 r}{\partial u^2}\cdot\partial_v r \cdot r^2 \cdot \O^4\cdot \partial_v \log\O -32\cdot\f{\partial^2 r}{\partial v^2}\cdot r^2\cdot\partial_u r \cdot \O^4\cdot \partial_u \log\O \bigg)\\
&+\f{4}{r^4 \O^8}\bigg( 16\cdot (\partial_v r)^2\cdot (\partial_u r)^2\cdot \O^4+64\cdot \partial_v r\cdot r^2\cdot \partial_u r\cdot \O^4\cdot \partial_u \log\O\cdot \partial_v\log\O+8\cdot\partial_v r\cdot \partial_u r\cdot \O^6\bigg)\\
&+\f{4}{r^4 \O^8}\bigg(16\cdot r^4\cdot (\f{\partial^2 \O}{\partial v \partial u})^2\cdot \O^2-32\cdot r^4\cdot \f{\partial^2 \O}{\partial v \partial u}\cdot \O^3 \cdot \partial_v \log\O \cdot \partial_u \log\O \bigg)\\
&+\f{4}{r^4 \O^8}\bigg( 16\cdot r^4\cdot \O^4\cdot (\partial_v \log\O)^2\cdot (\partial_u \log\O)^2+\O^8 \bigg)\\
=&\f{4}{r^2 \O^4}\bigg(16\cdot(\f{\partial^2 r}{\partial u \partial v})^2+16\cdot \f{\partial^2 r}{\partial u^2}\cdot\f{\partial^2 r}{\partial v^2} \bigg)\\
&+\f{4}{r^2 \O^4}\bigg(-32\cdot \f{\partial^2 r}{\partial u^2}\cdot\partial_v r  \cdot \partial_v \log\O -32\cdot\f{\partial^2 r}{\partial v^2}\cdot\partial_u r \cdot \partial_u \log\O \bigg)\\
&+\f{4}{r^4 \O^4}\bigg( 16\cdot (\partial_v r)^2\cdot (\partial_u r)^2+64\cdot \partial_v r\cdot r^2\cdot \partial_u r\cdot \partial_u \log\O\cdot \partial_v\log\O\bigg)\\
\end{split}
\end{equation*}
\begin{equation}\label{Kretschmann2}
\begin{split}
&+\f{32}{r^4 \O^2}\cdot\partial_v r\cdot \partial_u r+\f{4}{\O^8}\bigg(16\cdot (\f{\partial^2 \O}{\partial v \partial u})^2\cdot \O^2-32\cdot \f{\partial^2 \O}{\partial v \partial u}\cdot \O \cdot \O^2 \cdot \partial_v \log\O \cdot \partial_u \log\O \bigg)\\
&+\f{64}{\O^4}\cdot (\partial_v \log\O)^2\cdot (\partial_u \log\O)^2+\f{4}{r^4}.
\end{split}
\end{equation}
And for $\partial_v \partial_u \O$ we have
\begin{equation}\label{Omega 3}
\O\cdot\partial_v \partial_u \O=\O^2\cdot\partial_v \log\O\cdot \partial_u\log\O+\f12 \O^2\cdot\partial_v\partial_u \log\O^2.
\end{equation}
Allow $\lesssim$ to mean $\leq$ up to constants only depending on $M$. Applying the estimates above we have
\begin{equation*}
\begin{split}
\O^{-2}(u,v)\leq& \f{4 r_0^{\f{6\t D_1 \t D_2}{1-\e}\cdot\f{1}{M^{-2p}|u-u_1+v_{r=0}(u_1)|^{2p}}}}{2M-r_0}\cdot\f{r(u,v)}{r(u,v)^{\f{6\t D_1 \t D_2}{1-\e}\cdot\f{1}{M^{-2p}|u-u_1+v_{r=0}(u_1)|^{2p}}}}\\
\lesssim& \f{r(u,v)}{r(u,v)^{\f{6 \t D_1 \t D_2}{1-\e}\cdot\f{1}{M^{-2p}|u-u_1+v_{r=0}(u_1)|^{2p}}}}.
\end{split}
\end{equation*}
{\color{black}Note that for $|u|$ large, the constant omitted in the second inequality depends only on $M$.}

\noindent Applying the estimates in \cite{AZ}, it also holds that
$$\O^2(u,v)\lesssim \f{1}{r(u,v)}, \quad |\partial_v\log\O(u,v)|\lesssim \f{1}{r(u,v)^2}, \quad |\partial_u \log\O(u,v)|\lesssim \f{1}{r(u,v)^2},$$
$$|\partial_u\partial_v \log\O(u,v)|\lesssim \f{1}{r(u,v)^4}, \quad |\partial_u \O^2(u,v)|=|\O^2\cdot\partial_u \log\O^2(u,v)|\lesssim \f{1}{r(u,v)^3},$$
$$|\O\cdot\partial_v \partial_u \O(u,v)|\lesssim \f{1}{r(u,v)^5}, \quad |\partial_u r(u,v)|\lesssim \f{1}{r(u,v)}, \quad |\partial_v r(u,v)|\lesssim \f{1}{r(u,v)},$$
$$|\partial_u \partial_v r(u,v)|\lesssim \f{1}{r(u,v)^3}, \quad |\partial_u \partial_u r(u,v)|\lesssim \f{1}{r(u,v)^3}, \quad |\partial_v \partial_v r(u,v)|\lesssim \f{1}{r(u,v)^3}.$$
Therefore, we have
\begin{equation*}
\begin{split}
&R^{\alpha\beta\rho\sigma}R_{\alpha\beta\rho\sigma}\\
=&\f{4}{r^2 \O^4}\bigg(16\cdot(\f{\partial^2 r}{\partial u \partial v})^2+16\cdot \f{\partial^2 r}{\partial u^2}\cdot\f{\partial^2 r}{\partial v^2} \bigg)\\
&+\f{4}{r^2 \O^4}\bigg(-32\cdot \f{\partial^2 r}{\partial u^2}\cdot\partial_v r  \cdot \partial_v \log\O -32\cdot\f{\partial^2 r}{\partial v^2}\cdot\partial_u r \cdot \partial_u \log\O \bigg)\\
&+\f{4}{r^4 \O^4}\bigg( 16\cdot (\partial_v r)^2\cdot (\partial_u r)^2+64\cdot \partial_v r\cdot r^2\cdot \partial_u r\cdot \partial_u \log\O\cdot \partial_v\log\O\bigg)\\
&+\f{32}{r^4 \O^2}\cdot\partial_v r\cdot \partial_u r+\f{4}{\O^8}\bigg(16\cdot (\f{\partial^2 \O}{\partial v \partial u})^2\cdot \O^2-32\cdot \f{\partial^2 \O}{\partial v \partial u}\cdot \O \cdot \O^2 \cdot \partial_v \log\O \cdot \partial_u \log\O \bigg)\\
&+\f{64}{\O^4}\cdot (\partial_v \log\O)^2\cdot (\partial_u \log\O)^2+\f{4}{r^4}\\
\lesssim&\f{1}{r(u,v)^{\f{12 \t D_1 \t D_2}{1-\e}\cdot\f{1}{M^{-2p}|u-u_1+v_{r=0}(u_1)|^{2p}}}}\cdot\f{1}{r(u,v)^6}+\f{1}{r(u,v)^{\f{12\t D_1 \t D_2}{1-\e}\cdot\f{1}{M^{-2p}|u-u_1+v_{r=0}(u_1)|^{2p}}}}\cdot\f{1}{r(u,v)^4}\\ 
&+\f{1}{r(u,v)^{\f{24\t D_1 \t D_2}{1-\e}\cdot\f{1}{M^{-2p}|u-u_1+v_{r=0}(u_1)|^{2p}}}}\cdot\f{1}{r(u,v)^6}+\f{1}{r(u,v)^4}\\
\lesssim&\f{1}{r(u,v)^{6+\f{24 \t D_1 \t D_2}{1-\e}\cdot\f{1}{M^{-2p}|u-u_1+v_{r=0}(u_1)|^{2p}}}}.
\end{split}
\end{equation*}
This concludes the proof of Proposition \ref{thm 8.1}.  
\end{proof}

\section{Refined Estimates for $r\partial_v r(u,v)$ and $r\partial_u r(u,v)$}
Along $\mathcal{S}$, now we derive the asymptotic behaviours for $r\partial_u r(u,v)$ and $r\partial_v r(u,v)$:
\begin{theorem}\label{partial r}
For $(u,v)\in J^-(\t q)$ with $|u|\geq |u_1|+2v_{r=0}(u_1)$, $r(u,v)\leq r_0$ and {\color{black}$v$ sufficiently large such that (\ref{additional v}) holds}, we have
\begin{equation}\label{d 1.11}
|r\partial_u r(u,v)+M-f_1(u_{\t q})|\leq M[M^{-1}r(u,v)]^{\f{1}{100}},
\end{equation}
\begin{equation}\label{d 1.12}
|r\partial_v r(u,v)+M-f_2(v_{\t q})|\leq M[M^{-1}r(u,v)]^{\f{1}{100}},
\end{equation}
where   
\begin{align*}
|f_1(u_{\t q})|\leq& \f{M|M^{-1}[u_{\t q}-u_1-v_{r=0}(u_1)]|^{-p}}{2}\leq M(M^{-1}v_{\t q})^{-p},\\ 
|f_2(v_{\t q})|\leq& \f{M(M^{-1}v_{\t q})^{-p}}{2}  \leq M|M^{-1}[u-u_1-v_{r=0}(u_1)]|^{-p}.
\end{align*}
\end{theorem}
\begin{remark}
With this theorem, along $\mathcal{S}$ where $r(u,v)=0$, as $|u_{\t q}|\rightarrow +\infty$ we have
$$r\partial_u r(u_{\t q}, v_{r=0}(u_{\t q}))\rightarrow -M, \mbox{ and } r\partial_u r(u_{\t q}, v_{r=0}(u_{\t q}))\rightarrow -M$$
with an inverse polynomial rate $|u_{\t q}-u_1-v_{r=0}(u_1)|^{-p}$. \\
\end{remark}

\noindent We proceed to prove Theorem \ref{partial r}. 

\begin{proof}
We first use {\color{black}Proposition 5.1} in \cite{AZ}: 

\begin{center}
\begin{figure}[H]
\begin{minipage}[!t]{0.4\textwidth}
\begin{tikzpicture}[scale=0.95]
\draw [white](-1, 0)-- node[midway, sloped, above,black]{$\big(u_1, v_{r=0}(u_1)\big)$}(1, 0);
\draw [white](2.9, 0)-- node[midway, sloped, above,black]{$\t q$}(3.1, 0);
\draw [white](1.7, 0)-- node[midway, sloped, above,black]{$\mathcal{S}$}(2, 0);
\draw [white](5.5, 0.2)-- node[midway, sloped, above,black]{$i^+$}(7, 0.2);
\draw (0,0) to [out=-5, in=195] (5.5, 0.5);
\draw [white](0, -3)-- node[midway, sloped, above,black]{$\mathcal{H}$}(7, -3);
\draw [white](0, -4.5)-- node[midway, sloped, above,black]{$v=v_1$}(-4, -4.5);
\draw [white](0, -2.8)-- node[midway, sloped, above, black]{$r=r_0$}(0.5,-2.8);
\draw [white](-2.5, -2.5)-- node[midway, sloped, above, black]{$u=u_1$}(0,0);
\draw [white](3.58, -0.70)-- node[midway, sloped, below, black]{$\t q_2$}(3.78, -0.70);
\draw [white](1.29, -1.79)-- node[midway, sloped, below, black]{$\t q_1$}(1.38, -1.79);
\draw [thick] (-2.5,-2.5)--(0,-5);
\draw [dashed](-2.5, -2.5)--(0,0);
\draw [thick] (5.5, 0.5)--(0,-5);
\draw (-2,-2) to [out=-5, in=215] (5.5, 0.5);
\draw[fill] (0,0) circle [radius=0.08];
\draw[fill] (5.5, 0.5) circle [radius=0.08];
\draw[fill] (-2,-2) circle [radius=0.08];
\draw[fill] (3.68,-0.68) circle [radius=0.08];
\draw[fill](1.29, -1.71) circle [radius=0.08];
\fill[black!30!white] (2.47, -0.53)--(3,0)--(3.53, -0.53)--(3,-1.06);
\draw [thin](1.29, -1.71)--(3,0);
\draw[fill] (3,0) circle [radius=0.08];
\draw [thin](2.47, -0.53)--(3,-1.06);
\draw [thin](3.53, -0.53)--(3,-1.06);
\draw [thin](3.68, -0.68)--(3,0);
\draw[fill] (3, -1.06) circle [radius=0.08]; 

\end{tikzpicture}
\end{minipage}
\begin{minipage}[!t]{0.5\textwidth}
\end{minipage}
\hspace{0.05\textwidth}
\caption{Asymptotic behaviours for $r\partial_u r$ and $r\partial_v r$.}
\end{figure}
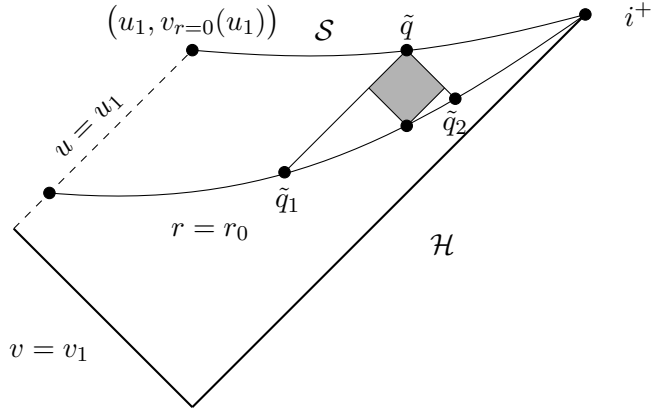
\end{center}

\noindent For any $(u,v)$ in the shadowed diamond region, we have 
\begin{equation}
|r\partial_u r(u,v)-r\partial_u r(u_{\t q}, v_{\t q})|\leq M[M^{-1}r(u,v)]^{\f{1}{100}},
\end{equation}
\begin{equation}
|r\partial_v r(u,v)-r\partial_v r(u_{\t q}, v_{\t q})|\leq M[M^{-1}r(u,v)]^{\f{1}{100}}.
\end{equation}
Now we start to derive estimates for  $r\partial_u r(u_{\t q}, v_{\t q})$ and $r\partial_v r(u_{\t q}, v_{\t q})$. By integrating 
\bes
\partial_u(r\partial_v r+M-\f12 r)=-\f14\O^2-\f12\partial_u r
\ees
we get
\be\label{r partial v r}
\begin{split}
&(r\partial_v r+M-\f12 r)(u,v)\\
=&(r\partial_v r+M-\f12 r)(u_{r_0(v)},v)-\int_{u_{r_0(v)}}^u\f{\O^2}{4\partial_u r}\partial_u r du-\f12\int^u_{u_{r_0(v)}}\partial_u r du\\
=&(r\partial_v r+M-\f12 r)(u_{r_0(v)},v)-\int^u_{u_{r_0(v)}}\bigg(\f{\O^2}{4\partial_u r}+\f12\bigg) \partial_u r du.
\end{split}
\ee
To obtain a precise bound for the integration term, we first estimate $\f{\O^2}{-\partial_u r}$. We change the form of one of the Raychauduri's equations, from 
$$\partial_u\bigg(\f{\partial_u r}{\O^2}\bigg)=-r\f{(\partial_u \phi)^2}{\O^2}$$
into
$$\partial_u \bigg(\log\f{\O^2}{-\partial_u r}\bigg)=\f{r}{\partial_u r}(\partial_u \phi)^2.$$
Via integration, we have
\bes
\begin{split}
\log\f{\O^2(u,v)}{-\partial_u r(u,v)}=&\log\f{\O^2(u_{r_0(v)},v)}{-\partial_u r(u_{r_0(v)},v)}+\int_{u_{r_0(v)}}^{u}\f{r (\partial_u \phi)^2}{(\partial_u r)^2}\partial_u r \,du .\\
\end{split}
\ees
This gives 
\bes
\begin{split}
\f{\O^2(u,v)}{-\partial_u r(u,v)}=&\f{\O^2(u_{r_0(v)},v)}{-\partial_u r(u_{r_0(v)},v)}\cdot \exp\bigg(\int_{r_0(v)}^{r(u,v)}\f{r^3 (\partial_u \phi)^2}{(r\partial_u r)^2}dr  \bigg).\\
\end{split}
\ees
{\color{black} By Proposition \ref{prop:metricest}, we have} 
$$\left|- \Omega^{-2}\partial_u r-\frac{1}{2}\right|(u_{r_0(v)},v)\leq C(\Delta_1,D_2,D_3,v_0) (M^{-1}v)^{-2p+1}.$$
For $p>1$ and $v$ sufficiently large, we hence have
$$\f{\O^2(u_{r_0(v)},v)}{4\partial_u r(u_{r_0(v)},v)}=-\f12+s(u,v)\cdot (M^{-1}v)^{-p}, \quad \mbox{ where } 0\leq |s(u,v)|\ll 1.$$
Using Proposition \ref{phi} and Proposition \ref{v u global}, for $r(u,v)\leq r_0$ and $v$ sufficiently large, there exists $l(u,v)$ satisfying
$0\leq|l(u,v)|\lesssim 1$ and it holds that
$$\exp\bigg(\int_{r_0, v)}^{r(u,v)}\f{r^3 (\partial_u \phi)^2}{(r\partial_u r)^2}dr  \bigg)= \exp\bigg({2\t D_2^2\cdot l(u,v)\cdot (M^{-1}v)^{-2p}}\cdot\ln\f{r(u,v)}{r_0}\bigg)=\bigg(\f{r(u,v)}{r_0}\bigg)^{{2\t D_1^2\cdot l(u,v)\cdot (M^{-1}v)^{-2p}}}.$$
Therefore,
\bes\begin{split} 
&\f{\O^2(u,v)}{4\partial_u r(u,v)}+\f12=\f{\O^2(u_{r_0(v)},v)}{4\partial_u r(u_{r_0(v)},v)}\cdot \exp\bigg(\int_{r_0(v)}^{r(u,v)}\f{r^3 (\partial_u \phi)^2}{(r\partial_u r)^2}dr  \bigg)+\f12\\
=&\bigg(-\f12+s(u,v)\cdot (M^{-1}v)^{-p}\bigg)\cdot \exp\bigg(\int_{r_0(v)}^{r(u,v)}\f{r^3 (\partial_u \phi)^2}{(r\partial_u r)^2}dr  \bigg)+\f12\\
=&\bigg(-\f12+s(u,v)\cdot (M^{-1}v)^{-p}\bigg)\cdot [\exp\bigg(\int_{r_0(v)}^{r(u,v)}\f{r^3 (\partial_u \phi)^2}{(r\partial_u r)^2}dr  \bigg)-1]-\f12+s(u,v)\cdot (M^{-1}v)^{-p}+\f12\\
=&\bigg(-\f12+s(u,v)\cdot (M^{-1}v)^{-p}\bigg)\cdot [\bigg(\f{r(u,v)}{r_0}\bigg)^{{2\t D_1^2\cdot l(u,v)\cdot (M^{-1}v)^{-2p}}}-1]+s(u,v)\cdot (M^{-1}v)^{-p}.
\end{split}\ees
Hence for the integration term in (\ref{r partial v r}), we have
\bes\begin{split}   
&|\int^u_{u_{r_0(v)}}\bigg(\f{\O^2}{4\partial_u r}+\f12\bigg) \partial_u r du|\\
=&|\int^{r(u,v)}_{r_0}\bigg(\f{\O^2}{4\partial_u r}+\f12\bigg) dr |\\
\leq&|\int^{r(u,v)}_{r_0} \big(-\f12+s(u',v)\cdot (M^{-1}v)^{-p}\big)\cdot[\bigg(\f{r(u',v)}{r_0}\bigg)^{{2\t D_1^2\cdot l(u,v)\cdot(M^{-1}v)^{-2p}}}-1] dr\,|\\
&+|\int^{r(u,v)}_{r_0} s(u',v)\cdot (M^{-1}v)^{-p} dr\,|\\
\leq&|\int^{r(u,v)}_{r_0} [\bigg(\f{r(u',v)}{r_0}\bigg)^{{2\t D_1^2\cdot l(u',v)\cdot (M^{-1}v)^{-2p}}}-1] dr\,|\\
&+\sup_{u_{r_0}(v)\leq u'\leq u}|s(u',v)|\cdot (M^{-1}v)^{-p}\cdot [r_0-r(u,v)].
\end{split}\ees

\noindent Denote $x=\f{r(u',v)}{r_0}$. And notice 
$$|{2\t D_1^2 \cdot l(u,v)\cdot (M^{-1}v)^{-2p}}|\leq {3\t D_1^2 \cdot (M^{-1}v)^{-2p}}.$$
we hence deduce
\bes\begin{split}   
&|\int^{r(u,v)}_{r_0} [\bigg(\f{r(u',v)}{r_0}\bigg)^{{2\t D^2_1\cdot l(u',v)\cdot (M^{-1}v)^{-2p}}}-1] dr|\\
\leq &\int^{r(u,v)}_{r_0} [\bigg(\f{r(u',v)}{r_0}\bigg)^{{-3\t D^2_1\cdot (M^{-1}v)^{-2p}}}-1] dr+\int^{r(u,v)}_{r_0} [-\bigg(\f{r(u',v)}{r_0}\bigg)^{{3\t D_1^2\cdot (M^{-1}v)^{-2p}}}+1] dr\\
=&r_0\int^{\f{r(u,v)}{r_0}}_{1} [x^{{-3\t D_1^2\cdot (M^{-1}v)^{-2p}}}-1] dx-r_0\int^{\f{r(u,v)}{r_0}}_{1} [x^{{3\t D_1^2\cdot (M^{-1}v)^{-2p}}}-1] dx\\
=&r_0\cdot \bigg(\f{x^{1-3\t D_1^2\cdot (M^{-1}v)^{-2p}}}{1-3\t D_1^2\cdot (M^{-1}v)^{-2p}} \vert^{x=\f{r(u,v)}{r_0}}_{x=1}  \bigg)-r(u,v)\\
&+r_0-r_0\cdot \bigg(\f{x^{1+3\t D_1^2\cdot (M^{-1}v)^{-2p}}}{1+3\t D^2_1\cdot (M^{-1}v)^{-2p}}\vert^{x=\f{r(u,v)}{r_0}}_{x=1}  \bigg)+r(u,v)-r_0\\
=&r_0\cdot \bigg(\f{[\f{r(u,v)}{r_0}]^{1-3\t D_1^2\cdot (M^{-1}v)^{-2p}}}{1-3\t D_1^2\cdot (M^{-1}v)^{-2p}}  \bigg)-r_0\cdot \f{1}{1-3\t D_1^2\cdot (M^{-1}v)^{-2p}}\\
&-r_0\cdot \bigg(\f{[\f{r(u,v)}{r_0}]^{1+3\t D_1^2\cdot (M^{-1}v)^{-2p}}}{1+3\t D_1^2\cdot (M^{-1}v)^{-2p}}  \bigg)+r_0\cdot \f{1}{1+3\t D_1^2\cdot (M^{-1}v)^{-2p}}\\
=&r^{\f12}\cdot r_0^{\f12}\cdot \bigg(\f{[\f{r(u,v)}{r_0}]^{\f12-3\t D_1^2\cdot (M^{-1}v)^{-2p}}}{1{-3\t D_1^2\cdot (M^{-1}v)^{-2p}}}  \bigg)+r^{\f12}\cdot r_0^{\f12}\cdot \bigg(\f{[\f{r(u,v)}{r_0}]^{\f12+3\t D_1^2\cdot (M^{-1}v)^{-2p}}}{1+{3\t D_1^2\cdot (M^{-1}v)^{-2p}}}  \bigg)\\
&+r_0\cdot \f{-6\t D_1^2\cdot (M^{-1}v)^{-2p}}{(1{-3\t D_1^2\cdot (M^{-1}v)^{-2p}})(1+{3\t D_1^2\cdot (M^{-1}v)^{-2p}})}\\
\leq&2r_0^{\f12}\cdot r(u,v)^{\f12}+7r_0\t D_1^2(M^{-1}v)^{-2p}.
\end{split}\ees
For the first inequality above, we decompose $l(u,v)$ into negative parts and non-negative part. 

\noindent Back to (\ref{r partial v r}), with Proposition \ref{prop:metricest} 
$$\left|r \partial_vr-M+\frac{1}{2}r\right|(u_{r_0(v)},v)\leq C(\Delta_1,D_2,D_3,v_0) M (M^{-1}v)^{-2p+1},$$
we thus obtain  
\bes
\begin{split}
&|(r\partial_v r+M-\f12 r)(u,v)|\\
=&|(r\partial_v r+M-\f12 r)(u_{r_0(v)},v)-\int^u_{u_{r_0(v)}}\bigg(\f{\O^2}{4\partial_u r}+\f12\bigg) \partial_u r du|\\
\leq &C(\Delta_1,D_2,D_3,v_0) M (M^{-1}v)^{-2p+1}+2r_0^{\f12}\cdot r^{\f12}+(r_0+7r_0 \t D_1^2)\cdot (M^{-1}v)^{-2p}.
\end{split}
\ees
For $\t q=(u_{\t q}, v_{\t q})\in \mathcal{S}$, we have $r(u_{\t q}, v_{\t q})=0$. For $p>1$ and $v$ being sufficiently large, the above inequality gives
$$|(r\partial_v r)(u_{\t q}, v_{\t q})+M|\leq \f{M(M^{-1}v_{\t q})^{-p}}{2}.$$
Therefore, (\ref{d 1.12}) is proved. In the same fashion, we also obtain (\ref{d 1.11}).
\end{proof}

\appendix

\section{Basic estimates}
\label{app:basic}
\begin{proof}[Proof of Lemma \ref{lm:expint}]
We apply the fundamental theorem of calculus to obtain
\begin{equation*}
\begin{split}
\int_a^{\infty} e^{- \alpha x} f(x)\,dx\leq &\: B_2\int_a^{\infty} e^{- \alpha x} x^{-p_2} \,dx\\
=&\: -\alpha^{-1}B_2 \int_a^{\infty} \frac{d}{dx}\left(e^{- \alpha x} x^{-p_2}\right) \,dx-p_2\alpha^{-1} B_2 \int_a^{\infty}e^{- \alpha x} x^{-p_2-1} \,dx\\
\leq &\: \alpha^{-1} B_2 e^{- \alpha a} a^{-p_2}.
\end{split}
\end{equation*}

We estimate similarly 
\begin{equation*}
\begin{split}
\int_a^b e^{\alpha x} f(x)\,dx\leq &\: B_2\int_a^b e^{\alpha x} x^{-p} \,dx\\
=&\: \alpha^{-1}B_2 \int_a^b \frac{d}{dx}\left(e^{\alpha x} x^{-p_2}\right) \,dx+p_2\alpha^{-1} B_2 \int_a^be^{\alpha x} x^{-p_2-1} \,dx\\
\leq &\: \alpha^{-1} (1+ b^{-1} p_2 \alpha^{-1})B_2 e^{\alpha b} b^{-p_2}.
\end{split}
\end{equation*}
\end{proof}

\begin{proposition}[Reverse Gr\"onwall inequality]
\label{prop:revgron}
{\color{black}Let $t_0\in \R$ and let $\psi,\beta: [t_0,\infty)\to \R$ be positive continuous functions. Let $A>0$ be a positive constant and assume that $\p(t)$ satisfies:}
\be \label{appendix 1}
\p(t)\geq A+\int_{t_0}^t \b(s)\p(s)ds \quad \mbox{ for any } \quad t\geq t_0.
\ee
Then
$$\p(t)\geq A\cdot e^{\int_{t_0}^t \b(s)ds}.$$
\end{proposition}
\begin{proof}
Let 
$$N(t):=A+\int_{t_0}^t \b(s)\p(s)ds.$$
Taking derivatives with respect to $t$ on both sides and using (\ref{appendix 1}), we have
$$N'(t)=\b(t)\p(t)\geq\b(t)\cdot N(t).$$
This gives
$$N'(t)-\b(t)N(t)\geq 0,$$
which is equivalent to
$$\bigg(e^{\int_{t_0}^t -\b(s)ds}N(t)\bigg)'=e^{\int_{t_0}^t -\b(s)ds}\cdot N'(t)-e^{\int_{t_0}^t -\b(s)ds}\cdot\b(t)\cdot N(t)\geq 0.$$
For $t\geq t_0$, we hence have
$$e^{\int_{t_0}^t -\b(s)ds}N(t)\geq e^{\int_{t_0}^{t_0} -\b(s)ds}N(t_0)=1\cdot A=A.$$
Together with (\ref{appendix 1}), this gives
$$\p(t)\geq N(t)\geq A\cdot e^{\int_{t_0}^t \b(s)ds}.$$
\end{proof}

\bibliographystyle{plain}
\bibliography{bibliography}

\end{document}